\documentclass[a4paper,12pt]{amsart}
\usepackage{fullpage}
\usepackage{amssymb}
\usepackage{wasysym}
\newtheorem{thm}{Theorem}

\newtheorem{lemma}{Lemma}
\newtheorem{cor}{Corollary}

\DeclareSymbolFont{script}{U}{eus}{m}{n}
\DeclareMathSymbol{\Wedge}{0}{script}{"5E}
\newcommand{\rank}{\operatorname{\mathrm{rank}}}
\newcommand{\End}{\operatorname{\mathrm{End}}}
\newcommand{\Hom}{\operatorname{\mathrm{Hom}}}

\newcommand{\bfour}[4]{\begin{picture}(68,11)
\put(4,1.5){\makebox(0,0){$\bullet$}}
\put(24,1.5){\makebox(0,0){$\bullet$}}
\put(44,1.5){\makebox(0,0){$\bullet$}}
\put(64,1.5){\makebox(0,0){$\bullet$}}
\put(4,1.5){\line(1,0){40}}
\put(44,2.5){\line(1,0){20}}
\put(44,0.5){\line(1,0){20}}
\put(54,1.5){\makebox(0,0){$\rangle$}}
\put(4,10){\makebox(0,0){$\vphantom{(}\scriptstyle #1$}}
\put(24,10){\makebox(0,0){$\vphantom{(}\scriptstyle #2$}}
\put(44,10){\makebox(0,0){$\vphantom{(}\scriptstyle #3$}}
\put(64,10){\makebox(0,0){$\vphantom{(}\scriptstyle #4$}}
\end{picture}}

\newcommand{\cthree}[3]{\begin{picture}(48,11)
\put(4,1.5){\makebox(0,0){$\bullet$}}
\put(24,1.5){\makebox(0,0){$\bullet$}}
\put(44,1.5){\makebox(0,0){$\bullet$}}
\put(4,1.5){\line(1,0){20}}
\put(24,2.5){\line(1,0){20}}
\put(24,0.5){\line(1,0){20}}
\put(34,1.5){\makebox(0,0){$\langle$}}
\put(4,10){\makebox(0,0){$\vphantom{(}\scriptstyle #1$}}
\put(24,10){\makebox(0,0){$\vphantom{(}\scriptstyle #2$}}
\put(44,10){\makebox(0,0){$\vphantom{(}\scriptstyle #3$}}
\end{picture}}

\newcommand{\afive}[5]{\begin{picture}(88,11)
\put(4,1.5){\makebox(0,0){$\bullet$}}
\put(24,1.5){\makebox(0,0){$\bullet$}}
\put(44,1.5){\makebox(0,0){$\bullet$}}
\put(64,1.5){\makebox(0,0){$\bullet$}}
\put(84,1.5){\makebox(0,0){$\bullet$}}
\put(4,1.5){\line(1,0){80}}
\put(4,10){\makebox(0,0){$\vphantom{(}\scriptstyle #1$}}
\put(24,10){\makebox(0,0){$\vphantom{(}\scriptstyle #2$}}
\put(44,10){\makebox(0,0){$\vphantom{(}\scriptstyle #3$}}
\put(64,10){\makebox(0,0){$\vphantom{(}\scriptstyle #4$}}
\put(84,10){\makebox(0,0){$\vphantom{(}\scriptstyle #5$}}
\end{picture}}

\newcommand{\esix}[6]{\begin{picture}(88,15)(0,-8)
\put(4,1.5){\makebox(0,0){$\bullet$}}
\put(24,1.5){\makebox(0,0){$\bullet$}}
\put(44,-14.5){\makebox(0,0){$\bullet$}}
\put(44,1.5){\makebox(0,0){$\bullet$}}
\put(64,1.5){\makebox(0,0){$\bullet$}}
\put(84,1.5){\makebox(0,0){$\bullet$}}
\put(4,1.5){\line(1,0){80}}
\put(44,1.5){\line(0,-1){16}}
\put(4,10){\makebox(0,0){$\vphantom{(}\scriptstyle #1$}}
\put(38,-12.5){\makebox(0,0){$\vphantom{(}\scriptstyle #2$}}
\put(24,10){\makebox(0,0){$\vphantom{(}\scriptstyle #3$}}
\put(44,10){\makebox(0,0){$\vphantom{(}\scriptstyle #4$}}
\put(64,10){\makebox(0,0){$\vphantom{(}\scriptstyle #5$}}
\put(84,10){\makebox(0,0){$\vphantom{(}\scriptstyle #6$}}
\end{picture}}
\newcommand{\xesix}[6]{\begin{picture}(88,15)(0,-8)
\put(4,1.5){\makebox(0,0){$\times$}}
\put(24,1.5){\makebox(0,0){$\bullet$}}
\put(44,-14.5){\makebox(0,0){$\bullet$}}
\put(44,1.5){\makebox(0,0){$\bullet$}}
\put(64,1.5){\makebox(0,0){$\bullet$}}
\put(84,1.5){\makebox(0,0){$\bullet$}}
\put(4,1.5){\line(1,0){80}}
\put(44,1.5){\line(0,-1){16}}
\put(4,10){\makebox(0,0){$\vphantom{(}\scriptstyle #1$}}
\put(38,-12.5){\makebox(0,0){$\vphantom{(}\scriptstyle #2$}}
\put(24,10){\makebox(0,0){$\vphantom{(}\scriptstyle #3$}}
\put(44,10){\makebox(0,0){$\vphantom{(}\scriptstyle #4$}}
\put(64,10){\makebox(0,0){$\vphantom{(}\scriptstyle #5$}}
\put(84,10){\makebox(0,0){$\vphantom{(}\scriptstyle #6$}}
\end{picture}}
\newcommand{\ffour}[4]{\begin{picture}(68,11)
\put(4,1.5){\makebox(0,0){$\bullet$}}
\put(24,1.5){\makebox(0,0){$\bullet$}}
\put(44,1.5){\makebox(0,0){$\bullet$}}
\put(64,1.5){\makebox(0,0){$\bullet$}}
\put(4,1.5){\line(1,0){20}}
\put(24,2.5){\line(1,0){20}}
\put(24,0.5){\line(1,0){20}}
\put(44,1.5){\line(1,0){20}}
\put(34,1.5){\makebox(0,0){$\rangle$}}
\put(4,10){\makebox(0,0){$\vphantom{(}\scriptstyle #1$}}
\put(24,10){\makebox(0,0){$\vphantom{(}\scriptstyle #2$}}
\put(44,10){\makebox(0,0){$\vphantom{(}\scriptstyle #3$}}
\put(64,10){\makebox(0,0){$\vphantom{(}\scriptstyle #4$}}
\end{picture}}

\usepackage{euscript,color}

\begin{document}

\title[Killing two-tensors]
{Prolongation and Killing two-tensors}
\author[M.G.~Eastwood]{Michael Eastwood}
\address{\hskip-\parindent
School of Mathematical Sciences\\
Adelaide University\\ 
SA 5005\\ 
Australia}
\email{meastwoo@gmail.com}
\author[T.~Leistner]{Thomas Leistner}
\address{\hskip-\parindent
School of Mathematical Sciences\\
Adelaide University\\ 
SA 5005\\ 
Australia}
\email{thomas.leistner@adelaide.edu.au}

\begin{abstract} We present a systematic prolongation procedure and its
implementation for Killing two-tensors.  We use the resulting machinery to
elucidate the natural quadratic mapping from Killing fields to Killing
two-tensors in general and on irreducible locally symmetric spaces of compact
type.
\end{abstract}

\renewcommand{\subjclassname}{\textup{2020} Mathematics Subject Classification}

\thanks{This work was supported by Australian Research Council (Discovery 
Projects DP190102360 and DP260102671).}

\maketitle

\setcounter{section}{-1}
\section{Introduction}\label{introduction}

This article is concerned with the differential geometry associated to a
torsion-free affine connection, viewed as a differential operator
$\nabla:\Wedge^1\to\Wedge^1\otimes\Wedge^1$, and written as
$\sigma_b\mapsto\nabla_a\sigma_b$, using the {\em abstract index notation\/}
of~\cite{OT}.  Here, we are writing $\Wedge^1$ for the bundle of $1$-forms on a
connected $n$-dimensional smooth manifold~$M$ (which will usually be suppressed
from the notation).  A $1$-form $\sigma_b$ is said to be {\em Killing\/} if and
only if 
\begin{equation}\label{Killing_one_form} \nabla_{(a}\sigma_{b)}=0,
\end{equation}
where round brackets mean to take the symmetric part.  In other words, the
$1$-form $\sigma_b$ should be in the kernel of the {\em Killing operator\/}
$$\textstyle 
\Wedge^1\ni\sigma_b\longmapsto\nabla_{(a}\sigma_{b)}\in\bigodot^2\!\Wedge^1.$$
As is well-known (and explained at the start of \S\ref{interactions}), if
$\nabla_a$ is the Levi-Civita connection of a metric $g_{ab}$, then its Killing
$1$-forms enjoy a geometric interpretation.  Specifically, the flow of a vector
field $X^a$ is by local isometries of $g_{ab}$ if and only if the
corresponding $1$-form $\sigma_b\equiv g_{bc}X^c$ is a Killing $1$-form (and
$X^a$ itself is said to be a {\em Killing field\/}).

A {\em Killing $2$-tensor} is a symmetric covariant tensor field 
$\sigma_{bc}=\sigma_{(bc)}$ such that
$$\nabla_{(a}\sigma_{bc)}=0.$$
Although, Killing $2$-tensors for a metric $g_{ab}$ do not find a direct
geometric interpretation via the symmetries of $g_{ab}$, they do provide
conserved quantities along the geodesics.  In any case, whether $\nabla_a$
preserves a metric or not, the Killing $2$-tensors are the solutions of an
overdetermined system and it follows (see, for example,~\cite{BCEG}) that the
local solution space is finite-dimensional (with dimension bounded by
$n(n+1)^2(n+2)/12$).  The dimension is only known in particular cases, for
example on the compact rank one symmetric spaces:
\begin{equation}\label{some_dimensions}
\begin{tabular}{rl} 
the round $n$-sphere or ${\mathbb{RP}}_n$&\enskip
dimension $=n(n+1)^2(n+2)/12$,\\
${\mathbb{CP}}_m$&\enskip
dimension $=m(m+1)^2(m+2)/2$,\\
${\mathbb{HP}}_k$&\enskip
dimension $=(k+1)(2k+3)(4k^2+6k+5)/3$,\\
${\mathbb{OP}}_2$&\enskip
dimension $=1404$,
\end{tabular}
\end{equation}
these last two cases only recently (due to Matveev and Nikolayevsky~\cite{MN}).  
For any two $1$-forms $\sigma_b$ and $\widetilde\sigma_b$, notice that
$$\nabla_{(a}\sigma_b\widetilde\sigma_{c)}
=\sigma_{(a}\nabla_b\widetilde\sigma_{c)}
+\widetilde\sigma_{(a}\nabla_b\sigma_{c)}$$
and so, if $\sigma_b$ and $\widetilde\sigma_b$ are Killing $1$-forms, then
$\sigma_{(b}\widetilde\sigma_{c)}$ is a Killing $2$-tensor.  The linear span of
such tensors comprises the {\em decomposable\/} Killing $2$-tensors whilst
Killing $2$-tensors not of this form are often referred to as {\em hidden
symmetries\/}.  (The celebrated {\em Carter constant\/} \cite{C} for the Kerr
metric corresponds to a hidden symmetry.)  Even for locally symmetric spaces,
hidden symmetries are somewhat mysterious.  The existence of hidden symmetries
for ${\mathbb{HP}}_k$ (with $k\geq 3$) and ${\mathbb{OP}}_2$ was only recently
discovered~\cite{MN}.  On the other hand, Nguyen~\cite{Nguyen} has shown 
that there are no hidden symmetries for the symmetric spaces
$G_2/{\mathrm{SO}}(4)$, or ${\mathrm{SL}}(n)/{\mathrm{SO}}(n)$; also for the
compact Lie groups ${\mathrm{SU}}(3)$, $G_2$, or ${\mathrm{SO}}(n)$, with 
their bi-invariant metrics.

In this article, we present a systematic prolongation procedure and its
implementation for Killing $1$-forms in \S\ref{K_1forms} and for Killing
$2$-tensors in~\S\ref{K_2tensors}.  In both these cases, it turns out that the
resulting connection is canonical.  Although various prolongation connections
are already known~\cite{BCEG,E,GL,Heil17,HouriTomodaYasui18}, our choices in
this article are arranged to respect the construction of Killing $2$-tensors
from Killing $1$-forms mentioned above.  All of these constructions are fully
understood for the round sphere (behind the scenes, it is because the round
sphere is projectively flat) and, in this case, the {\em Killing-Yano
$3$-forms} (a vector space of dimension $(n-2)(n-1)n(n+1)/24$ on the
$n$-sphere) appear as the kernel of the quadratic mapping from Killing
$1$-forms to Killing $2$-tensors (see \S\ref{the_sphere}).  We shall find (in
the second bullet point just after Corollary~\ref{SES_of_connections} in
\S\ref{2from1-sec}) that the Killing-Yano $3$-forms also control this kernel
for any affine torsion-free connection.

We apply these general results to the affine locally symmetric case (where
$\nabla_aR_{bc}{}^d{}_e=0$, as discussed in~\S\ref{affine_appendix}).
Consequently, we are able to identify (in Corollary~\ref{injection}) the
irreducible locally symmetric spaces of compact type so that the mapping from
Killing $1$-forms to Killing $2$-tensors is injective.  For example, this is
true for $E_6/F_4$ and provides another example where there are hidden
symmetries:
$$\underbrace{\{\mbox{Killing $2$-tensors on 
$E_6/F_4$}\}}_{\mbox{dimension $3159$}}
=\underbrace{\textstyle\bigodot^2\!
\{\mbox{Killing fields on $E_6/F_4$}\}}_{\mbox{dimension $3081$}}\oplus
\underbrace{\{\mbox{Hidden symmetries}\}}_{\mbox{dimension $78$}},$$
as shown in \S\ref{answer_for_E6/F4}.  For general symmetric spaces such as
this one, it is essential to have a more abstract understanding, via the
standard theory of highest weights, of the various homogeneous bundles that
arise.  For this, we find the computer software LiE~\cite{LiE} very useful and,
in particular, its use in branching representations to symmetric subgroups, as
explained in~\cite{EW}.  The LiE commands that we use are given in an
appendix~\S\ref{LiE_appendix}.  Their use is illustrated with
${\mathrm{SU}}(6)/{\mathrm{Sp}}(3)$, starting in \S\ref{interactions} and
continuing into \S\ref{A5/C3} where we show that this particular symmetric
space has no hidden symmetries.  In \S\ref{answer_for_F4/B4} we confirm the
results of Matveev and Nikolayevsky~\cite{MN} on the Octonionic plane (by
combining the methods of this article with some parabolic geometry).  Although
the most serious application of our prolongation procedure is to Killing
$2$-tensors, it also applies to affine Killing fields.  Being essentially a
side issue, we confine their discussion to another appendix
\S\ref{affine_appendix}.

\section{Prolongation}\label{prolongation}

\subsection{A prolongation procedure} In this subsection, we present a general
prolongation procedure.  It is canonical save for one choice, a `lift of
curvature' to a homomorphism~$\tilde\kappa$, as in the following Theorem.  This
procedure may be iterated and, in good circumstances, yields a connection on a
canonically constructed vector bundle whose parallel sections are in one-to-one
correspondence with the solutions of the original linear differential equation.
In \S\ref{K_1forms} we apply Theorem~\ref{prolong} to encode Killing $1$-forms
and in \S\ref{K_2tensors} to Killing $2$-tensors, where we will see the that
the freedom in choice of $\tilde\kappa$ washes out in the final connection,
obtained after two iterations.

\begin{thm}\label{prolong}
Suppose we have, on a smooth manifold~$M$,
\begin{itemize}
\item a smooth vector bundle $U\to M$ equipped with a connection
$\nabla:U\to\Wedge^1\otimes U$, 
\item a subbundle $\partial:V\hookrightarrow\Wedge^1\otimes U$ equipped with a 
connection $\nabla:V\to\Wedge^1\otimes V$ such that 
$\partial\in\Gamma(\Wedge^1\otimes\Hom(V,U))$ satisfies 
$\nabla\wedge\partial=0$, where the coupled exterior derivative
$$\nabla\wedge\underbar{\enskip}:\Wedge^1\otimes\Hom(V,U)\to
\Wedge^2\otimes\Hom(V,U)$$
is that obtained from the induced connection
$\Hom(V,U)\to\Wedge^1\otimes\Hom(V,U)$.
\end{itemize}
Let 
$\partial\wedge\underbar{\enskip}:\Wedge^p\otimes V\to\Wedge^{p+1}\otimes U$
denote the induced homomorphisms, defined for all $p\geq0$
by $\psi_{bc\cdots d}\mapsto\partial_{[a}\psi_{bc\cdots d]}$. Let $\kappa$ 
denote the composition
$$U\xrightarrow{\,\nabla\,}\Wedge^1\otimes U
\xrightarrow{\,\nabla\wedge\underbar{\enskip}\,}\Wedge^2\otimes U,$$
the curvature of the connection on $U$, and suppose that we are given a
homomorphism $\tilde\kappa:U\to\Wedge^1\otimes V$ such that
$\partial\wedge\tilde\kappa:U\to\Wedge^2\otimes U$ coincides with~$\kappa$.
Let
\begin{equation}\label{this_is_W}
W\equiv\ker\partial\wedge\underbar{\enskip}:\Wedge^1\otimes 
V\to\Wedge^2\otimes U.\end{equation}
Then the `prolonged' connection
\begin{equation}\label{ProlongedConnection}
\begin{array}{c}U\\[-5pt] \oplus\\[-3pt] V\end{array}\ni
\left[\!\begin{array}{c}\phi\\ \psi\end{array}\!\right]
\stackrel{\nabla}{\longmapsto}
\left[\!\begin{array}{c}\nabla\phi-\partial\psi\\ 
\nabla\psi+\tilde\kappa\phi\end{array}\!\right]\end{equation}
on $U\oplus V$ has the following properties:
\begin{enumerate}
\item\label{Ein} $\phi\in\Gamma(U)$ satisfies 
$$\nabla\phi=\partial\psi\quad\mbox{for some (unique)}\enskip\psi\in\Gamma(V)$$
if and only if 
$$\nabla\left[\!\begin{array}{c}\phi\\ \psi\end{array}\!\right]
=\left[\!\begin{array}{c}0\\ \theta\end{array}\!\right]
\quad\mbox{for some (unique)}\enskip\theta\in\Gamma(W),$$
\item\label{Zwei} the curvature of the prolonged connection is of the form
$$\nabla\wedge\nabla\left[\!\begin{array}{c}\phi\\ \psi\end{array}\!\right]
=\left[\!\begin{array}{c}0\\ 
\tilde\kappa\bowtie\psi+(\nabla\tilde\kappa)\bowtie\phi
\end{array}\!\right],$$
where both $\tilde\kappa\bowtie\psi$ and $(\nabla\tilde\kappa)\bowtie\phi$ take
values in 
$\ker\partial\wedge\underbar{\enskip}:\Wedge^2\otimes V\to\Wedge^3\otimes U$.
\end{enumerate}
\end{thm}
\begin{proof} Firstly, let us note that the compatibility condition 
$\nabla\wedge\partial=0$ is precisely that
\begin{equation}\label{anticommutes}
\begin{array}{ccccc}
&&\Wedge^{k+1}\otimes U&\xrightarrow{\,\nabla\,}&\Wedge^{k+2}\otimes U\\
&\raisebox{0pt}{\makebox[0pt][r]
         {$\!\partial\wedge\underbar{\enskip}$}}\nearrow&
&\nearrow\raisebox{-4pt}{\makebox[0pt][l]
         {$\!\partial\wedge\underbar{\enskip}$}}\\ 
\Wedge^k\otimes V&\xrightarrow{\,\nabla\,}&\Wedge^{k+1}\otimes V
\end{array}\end{equation}
anticommutes for $k=0$ (and then for all~$k\geq0$).
Now, suppose that $\phi\in\Gamma(U)$ satisfies
$$\nabla\phi=\partial\psi\quad\mbox{for some (unique)}\enskip
\psi\in\Gamma(V).$$
By applying $\nabla\wedge\underbar{\enskip}:\Wedge^1\otimes U\to
\Wedge^2\otimes U$ to both sides of this equation, we find that
$$\kappa\phi=\nabla\wedge\nabla\phi=\nabla\wedge\partial\psi
=-\partial\wedge\nabla\psi$$
and, therefore, conclude that
$$\partial\wedge(\nabla\psi+\tilde\kappa\phi)
=\partial\wedge\nabla\psi+\kappa\phi=0.$$
By definition, this means that $\nabla\psi+\tilde\kappa\phi$ is a section of
$W$ and we have established one direction of conclusion~(\ref{Ein}).  The other
direction, however, is trivial since the first line of the connection
(\ref{ProlongedConnection}) forces $\nabla\phi=\partial\psi$.  For
conclusion~(\ref{Zwei}), we should compute the curvature
of~(\ref{ProlongedConnection}) as follows:
$$\nabla\wedge\nabla\left[\!\begin{array}{c}\phi\\ \psi\end{array}\!\right]
=\nabla\wedge\left[\!\begin{array}{c}\nabla\phi-\partial\psi\\ 
\nabla\psi+\tilde\kappa\phi\end{array}\!\right]
=\left[\!\begin{array}{c}
\nabla\wedge(\nabla\phi-\partial\psi)
-\partial\wedge(\nabla\psi+\tilde\kappa\phi)\\ 
\nabla\wedge(\nabla\psi+\tilde\kappa\phi)
+\tilde\kappa\wedge(\nabla\phi-\partial\psi)\end{array}\!\right],$$
where 
$\tilde\kappa\wedge\underbar{\enskip}:\Wedge^1\otimes U\to\Wedge^2\otimes V$ 
is induced by $\tilde\kappa:U\to\Wedge^1\otimes V$.  The first line vanishes 
(by construction) and the second line yields
\begin{equation}\label{prolonged_curvature}
\lambda\psi+(\nabla\wedge\tilde\kappa)\phi-\tilde\kappa\wedge\nabla\phi
+\tilde\kappa\wedge\nabla\phi-\tilde\kappa\wedge\partial\psi
=\lambda\psi-\tilde\kappa\wedge\partial\psi+(\nabla\wedge\tilde\kappa)\phi,
\end{equation}
where $\lambda:\nabla\wedge\nabla:V\to\Wedge^2\otimes V$ is the curvature of 
$\nabla:V\to\Wedge^1\otimes V$. However, from (\ref{anticommutes}) we see that 
$\partial\wedge\lambda=\kappa\wedge\partial:V\to\Wedge^3\otimes U$ so
$$\partial\wedge(\lambda\psi-\tilde\kappa\wedge\partial\psi)
=\kappa\wedge\partial\psi-\partial\wedge\tilde\kappa\wedge\partial\psi=0,$$
as advertised.  Moreover,
$$\partial\wedge\nabla\wedge\tilde\kappa
=-\nabla\wedge\partial\wedge\tilde\kappa
=-\nabla\wedge\kappa=0,$$
by the Bianchi identity so $\partial\wedge(\nabla\wedge\tilde\kappa)\phi=0$, 
too.
\end{proof}
In operating Theorem~\ref{prolong}, the basic ingredients are the vector bundle
$U$ together with the subbundle $\partial:V\hookrightarrow\Wedge^1\otimes U$.
The connections on $U$ and $V$ are usually self-evident but there is sometimes
a choice of lift $\tilde\kappa:U\to\Wedge^1\otimes V$ of the curvature
$\kappa:U\to\Wedge^2\otimes U$, assuming such a lift exists (and we shall soon
encounter a case like this in~\S\ref{K_2tensors}).  So let us now consider the
freedom in $\tilde\kappa$ and especially how this freedom enters the picture if
we iterate Theorem~\ref{prolong}.

Clearly, if a lift exists, then the space of homomorphisms
$\tilde\kappa:U\to\Wedge^1\otimes V$ that lift to $\kappa$ form an affine space
over the vector space of homomorphisms $U\to W$, where W is defined
by~(\ref{this_is_W}).  Indeed, if $\tilde\kappa^1$ and $\tilde\kappa^2$ lift 
to $\kappa$, then $\tilde\kappa^1-\tilde\kappa^2:U\to W$.  In particular, if 
it so happens that $W=0$, then $\tilde\kappa$ is unique if it exists (as we 
shall encounter in~\S\ref{K_1forms}).   

Now let us see what happens if we iterate this process. For this, let us also 
denote by $\partial:W\hookrightarrow\Lambda^1 \otimes V$ the given injection 
and, in order to satisfy the assumptions of Theorem~\ref{prolong} assume that 
$\nabla:W\to \Lambda^1\otimes W$ is a connection such that
$$\begin{array}{ccccc}
&&\Wedge^{k+1}\otimes V&\xrightarrow{\,\nabla\,}&\Wedge^{k+2}\otimes V\\
&\raisebox{0pt}{\makebox[0pt][r]
         {$\!\partial\wedge\underbar{\enskip}$}}\nearrow&
&\nearrow\raisebox{-4pt}{\makebox[0pt][l]
         {$\!\partial\wedge\underbar{\enskip}$}}\\ 
\Wedge^k\otimes W&\xrightarrow{\,\nabla\,}&\Wedge^{k+1}\otimes W
\end{array}$$
anticommutes for $k=0$ (and then for all $k\geq 0$).
Let $\tilde{\kappa}^1$ and $\tilde{\kappa}^2$ be two lifts of the curvature 
$\kappa$ of the connection $\nabla$ on $U$. With the construction of 
Theorem~\ref{prolong}, this yields two connections 
$\nabla^1$ and $\nabla^2$ on $E=U\oplus V$ 
with curvatures, for $i=1,2$, 
$$\lambda_{ab}^i\left[\!\begin{array}{c}\phi\\ \psi\end{array}\!\right]
=\left[\!\begin{array}{c}0\\ 
\lambda_{ab}\psi
-\tilde\kappa_{[a}\makebox[0pt]{\hspace{-6pt}${}^i$}\partial_{b]}\psi
+(\nabla_{[a}\tilde\kappa_{b]}\makebox[0pt]{\hspace{-8pt}${}^i$})\phi
\end{array}\!\right]$$
following (\ref{prolonged_curvature}), with 
$\lambda_{ab}\psi
-\tilde\kappa_{[a}\makebox[0pt]{\hspace{-6pt}${}^i$}\partial_{b]}\psi$ and 
$(\nabla_{[a}\tilde\kappa_{b]}\makebox[0pt]{\hspace{-8pt}${}^i$})\phi$ 
taking values in
$$\ker\partial\wedge\underbar{\enskip}:
\Wedge^2\otimes V\to\Wedge^3\otimes U,$$
as in conclusion~(\ref{Zwei}).
In particular, if we set $S\equiv\tilde{\kappa}_2-\tilde{\kappa}_1$, then 
$$S:U\to W\stackrel{\partial}{\hookrightarrow}\Lambda^1\otimes V$$
and
\begin{equation}
\label{lambda12}
(\lambda^2_{ab}-\lambda^1_{ab})
\left[\!\begin{array}{c}\phi\\ \psi\end{array}\!\right]
=\left[\!\begin{array}{c}0\\
-S_{[a}\partial_{b]}\psi+(\nabla_{[a}S_{b]})\phi
\end{array}\!\right],\quad\mbox{where 
$S_a\equiv\partial_aS:U\to\Wedge^1\otimes V$}.
\end{equation}
Now, with a view to iteration, let $\tilde{\lambda}^i:E\to\Lambda^1\otimes W$,
$i=1,2$, be two lifts of $\lambda^i$, respectively, and let
$\overline{\nabla}^i$ be the corresponding connections on 
$E\oplus W=U\oplus V\oplus W$ constructed from Theorem~\ref{prolong} applied to
$E$ with the connections $\nabla^i$ and $W $ with the connection $\nabla$.

Using $S\equiv\tilde\kappa^2-\tilde\kappa^1:U\to W$ we now define an 
automorphism $\Phi$ of 
$E\oplus W=U\oplus V\oplus W$ by
$$\Phi\left[\!\begin{array}{c}\phi\\ \psi\\ \rho\end{array}\!\right]
\equiv\left[\!\begin{array}{c}\phi\\ \psi\\ \rho+S\phi
\end{array}\!\right],$$
and claim that 
\begin{equation}\label{claim}
\overline{\nabla}^2_b\circ\Phi-\Phi\circ\overline{\nabla}^1_b
:E\oplus W\longrightarrow 
\ker\partial\wedge\underbar{\enskip}:\Wedge^1\otimes W\to\Wedge^2\otimes E
\subset \Lambda^1\otimes (E\oplus W).
\end{equation}
To see this, we compute
$$\begin{array}{rcl}\overline\nabla_b^2\Phi\!\left[\!\begin{array}{c}
\phi\\ \psi\\ \rho\end{array}\!\right]
-\Phi\overline\nabla_b^1\!\left[\!\begin{array}{c}
\phi\\ \psi\\ \rho\end{array}\!\right]
&=&\left[\begin{array}{c}0\\ 
\tilde\kappa_b^2\phi-\tilde\kappa_b^2\phi-\partial_bS\phi\\ 
\tilde\lambda_b^2(\phi+\psi)-\tilde\lambda_b^1(\phi+\psi)
+\nabla_b(S\phi)-S(\nabla_b\phi-\partial_b\psi)\end{array}\!\right]\\[20pt]
&=&\left[\begin{array}{c}0\\ 
0\\ 
\tilde\lambda_b^2(\phi+\psi)-\tilde\lambda_b^1(\phi+\psi)
+(\nabla_bS)\phi+S\partial_b\psi\end{array}\!\right],
\end{array}$$ 
by the definition of $S$ and the chain rule for $\nabla_b(S\phi)$.  Applying
$\partial\wedge\underbar{\enskip}:\Wedge^1\otimes W\to\Wedge^2\otimes E$ to 
the last line and using (\ref{lambda12}) yields
$$-S_{[a}\partial_{b]}\psi+(\nabla_{[a}S_{b]})\phi
+(\nabla_{[b}S_{a]})\phi+S_{[a}\partial_{b]}\psi=0,$$
thus verifying claim~(\ref{claim}).  As a result, 
we obtain the following:

\medskip\noindent{\bf Addendum to Theorem~\ref{prolong}.}
{\em If 
$\partial\wedge\underbar{\enskip}:\Lambda^1\otimes W\to \Lambda^2\otimes E$
happens to be injective, then
\begin{itemize}
\item the lifts $\tilde{\lambda}^i$ of $\lambda^i$ are unique, respectively, 
\item the automorphism $\Phi$ is an affine map for the connections 
$\overline{\nabla}^1$ and $\overline{\nabla}^2$, i.e.\
$$\overline{\nabla}^2_b \circ \Phi =\Phi\circ \overline{\nabla}^1_b,$$
\item and solutions of\/ $\nabla\phi=\partial \psi$ are in one-to-one 
correspondence with parallel sections of either of these connections 
on $E\oplus W=U\oplus V\oplus W$.
\end{itemize}}

Theorem~\ref{prolong} and its Addendum may be brought to life by the following
examples.

\subsection{Killing one-forms}\label{K_1forms}
In Theorem~\ref{prolong}, let us take $U\equiv\Wedge^1$, equipped with a
torsion-free connection.  Let $V\equiv\Wedge^2$ and write
$\partial:\Wedge^2\hookrightarrow\Wedge^1\otimes\Wedge^1$ for the standard
inclusion.  Then $\sigma\in\Gamma(U)$ is a $1$-form
$\sigma_a\in\Gamma(\Wedge^1)$ and
$$\nabla_{(a}\sigma_{b)}=0\enskip\Leftrightarrow\enskip
\nabla_a\sigma_b=\mu_{ab}\enskip\Leftrightarrow\enskip
\nabla\sigma=\partial\mu,\enskip\mbox{for some (unique) $2$-form }
\mu\in\Gamma(V).$$
Theorem~\ref{prolong} allows us to interpret `Killing $1$-forms,' solutions of
the equation $\nabla_{(a}\sigma_{b)}=0$, as parallel sections of the prolonged
bundle $\Wedge^1\oplus\Wedge^2$.  The details are as follows.  We take the
connection on $V=\Wedge^2$ to be induced by our torsion free connection
on~$\Wedge^1$. In this case, we have $\nabla_a\partial_b=0$ and so 
$\nabla_{[a}\partial_{b]}=0$, as required in the statement of 
Theorem~\ref{prolong}. With conventions as in~\cite{OT}, the curvature on 
$\Wedge^1$ is given by
\begin{equation}\label{curvature_on_one-forms}
\textstyle\sigma_c\xrightarrow{\,\kappa\,}\nabla_{[a}\nabla_{b]}\sigma_c
=-\frac12R_{ab}{}^d{}_c\sigma_d\end{equation}
and, to use Theorem~\ref{prolong}, we are required to have a homomorphism 
$\tilde\kappa:U\to\Wedge^1\otimes V$ such that the following diagram
$$\begin{array}{cccccc}
U&\xrightarrow{\,\nabla\,}&\Wedge^1\otimes U&\xrightarrow{\,\nabla\,}
&\Wedge^2\otimes U\\
&\raisebox{-4pt}{\makebox[0pt]{$\tilde\kappa$}}\searrow
&&\nearrow\raisebox{-4pt}{\makebox[0pt][l]
                         {$\!\!\partial\wedge\underbar{\enskip}$}}\\[-4pt]
&&\Wedge^1\otimes V
\end{array}$$
commutes (with indices $\partial_{[a}\tilde\kappa_{b]}=\kappa_{ab}
=\nabla_{[a}\nabla_{b]}$).  For Killing $1$-forms this homomorphism is forced
because 
$\partial\wedge\underbar{\enskip}:\Wedge^1\otimes V\to\Wedge^2\otimes U$ is
actually an isomorphism.  Specifically, 
\begin{itemize}
\item $\partial:
\begin{array}[t]{ccc}V&\hookrightarrow&\Wedge^1\otimes U\\ \|&&\|\\ 
\Wedge^2&\hookrightarrow&\Wedge^1\otimes\Wedge^1
\makebox[0pt][l]{\quad is given by $\mu_{ab}\mapsto\mu_{ab}$}
\end{array}$
\item and implies that $\partial\wedge\underbar{\enskip}:
\begin{array}[t]{ccc}\Wedge^1\otimes V&\to&\Wedge^2\otimes U\\ \|&&\|\\ 
\Wedge^1\otimes\Wedge^2&\to&\Wedge^2\otimes\Wedge^1
\makebox[0pt][l]{\quad is given by $\mu_{abc}\mapsto\mu_{[ba]c}$,}
\end{array}$
\end{itemize}
which has an inverse given by
\begin{equation}\label{partial_inverse}
\Gamma(\Wedge^2\otimes\Wedge^1)\ni\theta_{abc}\longmapsto
\mu_{abc}\equiv\theta_{acb}+\theta_{bca}-\theta_{abc}.\end{equation}
This explicit inverse allows us to compute $\tilde\kappa$. 
Specifically, from (\ref{curvature_on_one-forms}) we find that
$$\textstyle\sigma_c\xrightarrow{\,\tilde\kappa\,}
-\frac12R_{ac}{}^d{}_b\sigma_d-\frac12R_{bc}{}^d{}_a\sigma_d
+\frac12R_{ab}{}^d{}_c\sigma_d,$$
which we may rewrite using the Bianchi symmetry 
$R_{a[b}{}^d{}_{c]}=-\frac12R_{bc}{}^d{}_a$ as 
\begin{equation}\label{kappatilde_for_Killing_fields}
\sigma_c\xrightarrow{\,\tilde\kappa\,}-R_{bc}{}^d{}_a\sigma_d.\end{equation}
The prolonged connection (\ref{ProlongedConnection}) from Theorem~\ref{prolong}
is, therefore,
\begin{equation}\label{Killing_connection}
\begin{array}{c}\Wedge^1\\[-5pt] \oplus\\[-3pt] \Wedge^2\end{array}\ni
\left[\!\begin{array}{c}\sigma_b\\ \mu_{bc}\end{array}\!\right]
\stackrel{\nabla_a\,}{\longmapsto}
\left[\!\begin{array}{c}\nabla_a\sigma_b-\mu_{ab}\\ 
\nabla_a\mu_{bc}-R_{bc}{}^d{}_a\sigma_d\end{array}\!\right].\end{equation}
Moreover, in this case, since 
$$W\equiv\ker\partial\wedge\underbar{\enskip}:\Wedge^1\otimes\Wedge^2\to
\Wedge^2\otimes\Wedge^1$$
vanishes, we arrive at the following classical result~\cite{CELM,K}.
\begin{thm}\label{killing_prolongation}
The Killing $1$-forms, namely $1$-forms $\sigma_b$ such that
$\nabla_{(a}\sigma_{b)}=0$ on a manifold equipped with a torsion-free affine
connection~$\nabla_a$, are in $1$-$1$ correspondence with parallel sections of
the bundle $\Wedge^1\oplus\Wedge^2$ equipped with the
connection~\eqref{Killing_connection}.
\end{thm}
This is what conclusion (\ref{Ein}) from Theorem~\ref{prolong} says. 
The `Killing connection' (\ref{Killing_connection}) also nicely illustrates
conclusion (\ref{Zwei}) as follows.  The explicit
formula (\ref{Killing_connection}) yields
$$\nabla_a\nabla_b
\left[\!\begin{array}{c}\sigma_c\\ \mu_{cd}\end{array}\!\right]
=\left[\!\begin{array}{c}\nabla_a(\nabla_b\sigma_c-\mu_{bc})
-(\nabla_b\mu_{ac}-R_{ac}{}^d{}_b\sigma_d)\\ 
\nabla_a(\nabla_b\mu_{cd}-R_{cd}{}^e{}_b\sigma_e)
-R_{cd}{}^e{}_a(\nabla_b\sigma_e-\mu_{be})\end{array}\!\right]$$
and, therefore, for the curvature of the Killing connection we find
\begin{equation}\label{full_Killing_curvature}
\nabla_{[a}\nabla_{b]}
\left[\!\begin{array}{c}\sigma_c\\ \mu_{cd}\end{array}\!\right]
=\left[\!\begin{array}{c}\nabla_{[a}\nabla_{b]}\sigma_c
-R_{c[a}{}^d{}_{b]}\sigma_d\\ 
\nabla_{[a}\nabla_{b]}\mu_{cd}-(\nabla_{[a}R_{|cd|}{}^e{}_{b]})\sigma_e
+R_{cd}{}^e{}_{[a}\mu_{b]e}\end{array}\!\right].\end{equation}
Sure enough, the first line vanishes by (\ref{curvature_on_one-forms}) and the
Bianchi symmetry $R_{c[a}{}^d{}_{b]}=-\frac12R_{ab}{}^d{}_c$. The second line 
splits into two parts. Firstly, there is
\begin{equation}\label{in_the_locally_symmetric_case}
\nabla_{[a}\nabla_{b]}\mu_{cd}+R_{cd}{}^e{}_{[a}\mu_{b]e}
=R_{ab}{}^e{}_{[c}\mu_{d]e}+R_{cd}{}^e{}_{[a}\mu_{b]e},\end{equation}
which is evidently a section of the symmetric product
$$\textstyle\Wedge^2\bigodot\Wedge^2
=\begin{picture}(6,12)(0,2)
\put(0,0){\line(1,0){6}}
\put(0,6){\line(1,0){6}}
\put(0,12){\line(1,0){6}}
\put(0,0){\line(0,1){12}}
\put(6,0){\line(0,1){12}}
\end{picture}\bigodot
\begin{picture}(6,12)(0,2)
\put(0,0){\line(1,0){6}}
\put(0,6){\line(1,0){6}}
\put(0,12){\line(1,0){6}}
\put(0,0){\line(0,1){12}}
\put(6,0){\line(0,1){12}}
\end{picture}
=\begin{picture}(12,12)(0,2)
\put(0,0){\line(1,0){12}}
\put(0,6){\line(1,0){12}}
\put(0,12){\line(1,0){12}}
\put(0,0){\line(0,1){12}}
\put(6,0){\line(0,1){12}}
\put(12,0){\line(0,1){12}}
\end{picture}\oplus\Wedge^4,$$
having the additional property that skewing over
all four indices vanishes by the Bianchi symmetry~$R_{[ab}{}^e{}_{c]}=0$. 
Therefore, this tensor has Riemann tensor symmetries
$$\begin{picture}(12,12)(0,2)
\put(0,0){\line(1,0){12}}
\put(0,6){\line(1,0){12}}
\put(0,12){\line(1,0){12}}
\put(0,0){\line(0,1){12}}
\put(6,0){\line(0,1){12}}
\put(12,0){\line(0,1){12}}
\end{picture}
=\{X_{abcd}=X_{[ab][cd]}\mid X_{[abc]d}=0\}
=\ker\partial:\Wedge^2\otimes\Wedge^2\to\Wedge^3\otimes\Wedge^1,$$
in accordance with~(\ref{Zwei}).  For the remaining term,
$$\textstyle(\nabla_{[a}R_{|cd|}{}^e{}_{b]})\sigma_e
=\frac12(\nabla_aR_{cd}{}^e{}_b)\sigma_e
-\frac12(\nabla_bR_{cd}{}^e{}_a)\sigma_e$$
and skewing over $bcd$ gives two terms
$$\textstyle\frac12(\nabla_aR_{[cd}{}^e{}_{b]})\sigma_e
-\frac12(\nabla_{[b}R_{cd]}{}^e{}_a)\sigma_e$$
both of which vanish by the Bianchi symmetry and Bianchi identity,
respectively. This is enough to force this term into 
$\begin{picture}(12,12)(0,2)
\put(0,0){\line(1,0){12}}
\put(0,6){\line(1,0){12}}
\put(0,12){\line(1,0){12}}
\put(0,0){\line(0,1){12}}
\put(6,0){\line(0,1){12}}
\put(12,0){\line(0,1){12}}
\end{picture}\,$, as expected.  In outline, by following through the proof of 
Theorem~\ref{prolong}, this construction for Killing $1$-forms is a matter of
chasing the following diagram concerning the prolonged `Killing bundle'
$E\equiv\begin{picture}(6,6)(0,0)
\put(0,0){\line(1,0){6}}
\put(0,6){\line(1,0){6}}
\put(0,0){\line(0,1){6}}
\put(6,0){\line(0,1){6}}
\end{picture}\oplus\begin{picture}(6,12)(0,2)
\put(0,0){\line(1,0){6}}
\put(0,6){\line(1,0){6}}
\put(0,12){\line(1,0){6}}
\put(0,0){\line(0,1){12}}
\put(6,0){\line(0,1){12}}
\end{picture}$
$$\begin{array}{c}E\\ \|\\
\begin{picture}(6,6)(0,0)
\put(0,0){\line(1,0){6}}
\put(0,6){\line(1,0){6}}
\put(0,0){\line(0,1){6}}
\put(6,0){\line(0,1){6}}
\end{picture}\\
\oplus\\[2pt]
\begin{picture}(6,12)(0,2)
\put(0,0){\line(1,0){6}}
\put(0,6){\line(1,0){6}}
\put(0,12){\line(1,0){6}}
\put(0,0){\line(0,1){12}}
\put(6,0){\line(0,1){12}}
\end{picture}
\end{array}
\begin{picture}(24,30)
\put(0,2){$\xrightarrow{\enskip\nabla\enskip}$}
\put(0,-22){\vector(3,2){28}}
\put(12,-8){\makebox(0,0){\footnotesize$\partial$}}
\put(0,-30){$\xrightarrow{\enskip\nabla\enskip}$}
\end{picture}
\begin{array}{c}\Wedge^1\otimes E\\ \|\\
\Wedge^1\otimes\begin{picture}(6,6)(0,0)
\put(0,0){\line(1,0){6}}
\put(0,6){\line(1,0){6}}
\put(0,0){\line(0,1){6}}
\put(6,0){\line(0,1){6}}
\end{picture}\\
\oplus\\[2pt]
\Wedge^1\otimes\begin{picture}(6,12)(0,2)
\put(0,0){\line(1,0){6}}
\put(0,6){\line(1,0){6}}
\put(0,12){\line(1,0){6}}
\put(0,0){\line(0,1){12}}
\put(6,0){\line(0,1){12}}
\end{picture}
\end{array}
\begin{picture}(24,30)
\put(0,2){$\xrightarrow{\enskip\nabla\enskip}$}
\put(0,-22){\vector(3,2){28}}
\put(12,-8){\makebox(0,0){\footnotesize$\partial$}}
\put(0,-30){$\xrightarrow{\enskip\nabla\enskip}$}
\end{picture}
\begin{array}{c}\Wedge^2\otimes E\\ \|\\
\Wedge^2\otimes\begin{picture}(6,6)(0,0)
\put(0,0){\line(1,0){6}}
\put(0,6){\line(1,0){6}}
\put(0,0){\line(0,1){6}}
\put(6,0){\line(0,1){6}}
\end{picture}\\
\oplus\\[2pt]
\Wedge^2\otimes\begin{picture}(6,12)(0,2)
\put(0,0){\line(1,0){6}}
\put(0,6){\line(1,0){6}}
\put(0,12){\line(1,0){6}}
\put(0,0){\line(0,1){12}}
\put(6,0){\line(0,1){12}}
\end{picture}
\end{array}
\begin{picture}(24,30)
\put(0,2){$\xrightarrow{\enskip\nabla\enskip}$}
\put(0,-22){\vector(3,2){28}}
\put(12,-8){\makebox(0,0){\footnotesize$\partial$}}
\end{picture}
\begin{array}{c}\Wedge^3\otimes E\\ \|\\
\Wedge^3\otimes\begin{picture}(6,6)(0,0)
\put(0,0){\line(1,0){6}}
\put(0,6){\line(1,0){6}}
\put(0,0){\line(0,1){6}}
\put(6,0){\line(0,1){6}}
\end{picture}\\
\oplus\\[2pt]
\begin{picture}(6,12)(0,2)
\put(3,8){\makebox(0,0){$\vdots$}}
\end{picture}\end{array}$$
and the location of its curvature is a consequence of the short exact sequence
$$0\to\begin{picture}(12,12)(0,2)
\put(0,0){\line(1,0){12}}
\put(0,6){\line(1,0){12}}
\put(0,12){\line(1,0){12}}
\put(0,0){\line(0,1){12}}
\put(6,0){\line(0,1){12}}
\put(12,0){\line(0,1){12}}
\end{picture}\to\Wedge^2\otimes\begin{picture}(6,12)(0,2)
\put(0,0){\line(1,0){6}}
\put(0,6){\line(1,0){6}}
\put(0,12){\line(1,0){6}}
\put(0,0){\line(0,1){12}}
\put(6,0){\line(0,1){12}}
\end{picture}\xrightarrow{\,\partial\,}
\Wedge^3\otimes\begin{picture}(6,6)(0,0)
\put(0,0){\line(1,0){6}}
\put(0,6){\line(1,0){6}}
\put(0,0){\line(0,1){6}}
\put(6,0){\line(0,1){6}}
\end{picture}\to 0.$$ 

\subsection{Killing two-tensors}\label{K_2tensors}
For Killing two-tensors,
$$\textstyle\sigma_{ab}\in\Gamma(\bigodot^2\!\Wedge^1)
\quad\mbox{such that}\enskip\nabla_{(a}\sigma_{bc)}=0,$$
there are some choices to be made regarding how certain bundles are realised.
The na\"{\i}ve interpretation of the Killing equation, for example, is to 
require that 
$$\nabla_a\sigma_{bc}=\tilde\mu_{abc}\quad\mbox{for some (unique)}\enskip
\tilde\mu_{abc}=\tilde\mu_{a(bc)}\mbox{ such that }\tilde\mu_{(abc)}=0,$$
in other words that $\nabla\sigma-\partial\mu=0$ for
$$\mu\in\Gamma\big(\,\begin{picture}(12,12)(0,2)
\put(0,0){\line(1,0){6}}
\put(0,6){\line(1,0){12}}
\put(0,12){\line(1,0){12}}
\put(0,0){\line(0,1){12}}
\put(6,0){\line(0,1){12}}
\put(12,6){\line(0,1){6}}
\end{picture}\,\big),
\enskip\mbox{where}\enskip
\begin{picture}(12,12)(0,2)
\put(0,0){\line(1,0){6}}
\put(0,6){\line(1,0){12}}
\put(0,12){\line(1,0){12}}
\put(0,0){\line(0,1){12}}
\put(6,0){\line(0,1){12}}
\put(12,6){\line(0,1){6}}
\end{picture}=\{\tilde\mu_{abc}=\tilde\mu_{a(bc)}\mid\tilde\mu_{(abc)}=0\}.$$
However, as an abstract bundle, we may alternatively adopt the realisation 
\begin{equation}\label{alternative_hook}\begin{picture}(12,12)(0,2)
\put(0,0){\line(1,0){6}}
\put(0,6){\line(1,0){12}}
\put(0,12){\line(1,0){12}}
\put(0,0){\line(0,1){12}}
\put(6,0){\line(0,1){12}}
\put(12,6){\line(0,1){6}}
\end{picture}=\{\mu_{abc}=\mu_{a[bc]}\mid\mu_{[abc]}=0\}
\end{equation}
and encode the Killing equation as $\nabla_a\sigma_{bc}+2\mu_{(bc)a}=0$, 
for example. Schematically, the first prolongation of the Killing equation 
$\nabla_{(a}\sigma_{bc)}=0$ now comes from the diagram
$$E=\begin{array}{c} U\\ \oplus\\[2pt] V\end{array}=
\begin{array}{c}
\begin{picture}(12,6)(0,0)
\put(0,0){\line(1,0){12}}
\put(0,6){\line(1,0){12}}
\put(0,0){\line(0,1){6}}
\put(6,0){\line(0,1){6}}
\put(12,0){\line(0,1){6}}
\end{picture}\\
\oplus\\[2pt]
\begin{picture}(12,12)(0,2)
\put(0,0){\line(1,0){6}}
\put(0,6){\line(1,0){12}}
\put(0,12){\line(1,0){12}}
\put(0,0){\line(0,1){12}}
\put(6,0){\line(0,1){12}}
\put(12,6){\line(0,1){6}}
\end{picture}
\end{array}
\begin{picture}(24,10)(0,-14)
\put(0,2){$\xrightarrow{\enskip\nabla\enskip}$}
\put(0,-22){\vector(3,2){28}}
\put(12,-8){\makebox(0,0){\footnotesize$\partial$}}
\put(0,-30){$\xrightarrow{\enskip\nabla\enskip}$}
\end{picture}
\begin{array}{c}
\Wedge^1\otimes\begin{picture}(12,6)(0,0)
\put(0,0){\line(1,0){12}}
\put(0,6){\line(1,0){12}}
\put(0,0){\line(0,1){6}}
\put(6,0){\line(0,1){6}}
\put(12,0){\line(0,1){6}}
\end{picture}\\
\oplus\\[2pt]
\Wedge^1\otimes\begin{picture}(12,12)(0,2)
\put(0,0){\line(1,0){6}}
\put(0,6){\line(1,0){12}}
\put(0,12){\line(1,0){12}}
\put(0,0){\line(0,1){12}}
\put(6,0){\line(0,1){12}}
\put(12,6){\line(0,1){6}}
\end{picture}
\end{array}
\begin{picture}(24,10)(0,-14)
\put(0,2){$\xrightarrow{\enskip\nabla\enskip}$}
\put(0,-22){\vector(3,2){28}}
\put(12,-8){\makebox(0,0){\footnotesize$\partial$}}
\put(0,-30){$\xrightarrow{\enskip\nabla\enskip}$}
\end{picture}
\begin{array}{c}
\Wedge^2\otimes\begin{picture}(12,6)(0,0)
\put(0,0){\line(1,0){12}}
\put(0,6){\line(1,0){12}}
\put(0,0){\line(0,1){6}}
\put(6,0){\line(0,1){6}}
\put(12,0){\line(0,1){6}}
\end{picture}\\
\oplus\\[2pt]
\Wedge^2\otimes\begin{picture}(12,12)(0,2)
\put(0,0){\line(1,0){6}}
\put(0,6){\line(1,0){12}}
\put(0,12){\line(1,0){12}}
\put(0,0){\line(0,1){12}}
\put(6,0){\line(0,1){12}}
\put(12,6){\line(0,1){6}}
\end{picture}
\end{array}
\begin{picture}(24,10)(0,-14)
\put(0,2){$\xrightarrow{\enskip\nabla\enskip}$}
\put(0,-22){\vector(3,2){28}}
\put(12,-8){\makebox(0,0){\footnotesize$\partial$}}
\end{picture}
\begin{array}{c}
\Wedge^3\otimes\begin{picture}(12,6)(0,0)
\put(0,0){\line(1,0){12}}
\put(0,6){\line(1,0){12}}
\put(0,0){\line(0,1){6}}
\put(6,0){\line(0,1){6}}
\put(12,0){\line(0,1){6}}
\end{picture}\,,\\
\oplus\\[2pt]
\begin{picture}(6,12)(0,2)
\put(3,8){\makebox(0,0){$\vdots$}}
\end{picture}\end{array}$$
provided that we have a factorisation of curvature 
$\kappa:\begin{picture}(12,6)(0,0)
\put(0,0){\line(1,0){12}}
\put(0,6){\line(1,0){12}}
\put(0,0){\line(0,1){6}}
\put(6,0){\line(0,1){6}}
\put(12,0){\line(0,1){6}}
\end{picture}\to\Wedge^2\otimes\begin{picture}(12,6)(0,0)
\put(0,0){\line(1,0){12}}
\put(0,6){\line(1,0){12}}
\put(0,0){\line(0,1){6}}
\put(6,0){\line(0,1){6}}
\put(12,0){\line(0,1){6}}
\end{picture}\,$
according to 
\begin{equation}\label{factorisation}\begin{array}{c}
\begin{picture}(12,6)(0,0)
\put(0,0){\line(1,0){12}}
\put(0,6){\line(1,0){12}}
\put(0,0){\line(0,1){6}}
\put(6,0){\line(0,1){6}}
\put(12,0){\line(0,1){6}}
\end{picture}\\
\phantom{\oplus}\\[2pt]
\phantom{\begin{picture}(12,12)(0,2)
\put(0,0){\line(1,0){6}}
\put(0,6){\line(1,0){12}}
\put(0,12){\line(1,0){12}}
\put(0,0){\line(0,1){12}}
\put(6,0){\line(0,1){12}}
\put(12,6){\line(0,1){6}}
\end{picture}}
\end{array}
\begin{picture}(24,10)(0,-14)
\put(0,2){$\xrightarrow{\enskip\nabla\enskip}$}
\put(0,-4){\vector(3,-2){28}}
\put(14,-6){\makebox(0,0){\footnotesize$\tilde\kappa$}}
\end{picture}
\begin{array}{c}
\Wedge^1\otimes\begin{picture}(12,6)(0,0)
\put(0,0){\line(1,0){12}}
\put(0,6){\line(1,0){12}}
\put(0,0){\line(0,1){6}}
\put(6,0){\line(0,1){6}}
\put(12,0){\line(0,1){6}}
\end{picture}\\
\oplus\\[2pt]
\Wedge^1\otimes\begin{picture}(12,12)(0,2)
\put(0,0){\line(1,0){6}}
\put(0,6){\line(1,0){12}}
\put(0,12){\line(1,0){12}}
\put(0,0){\line(0,1){12}}
\put(6,0){\line(0,1){12}}
\put(12,6){\line(0,1){6}}
\end{picture}
\end{array}
\begin{picture}(24,10)(0,-14)
\put(0,2){$\xrightarrow{\enskip\nabla\enskip}$}
\put(0,-22){\vector(3,2){28}}
\put(12,-8){\makebox(0,0){\footnotesize$\partial$}}
\end{picture}\begin{array}{c}
\Wedge^2\otimes\begin{picture}(12,6)(0,0)
\put(0,0){\line(1,0){12}}
\put(0,6){\line(1,0){12}}
\put(0,0){\line(0,1){6}}
\put(6,0){\line(0,1){6}}
\put(12,0){\line(0,1){6}}
\end{picture}\,,\\
\phantom{\oplus}\\[2pt]
\phantom{\Wedge^2\otimes\begin{picture}(12,12)(0,2)
\put(0,0){\line(1,0){6}}
\put(0,6){\line(1,0){12}}
\put(0,12){\line(1,0){12}}
\put(0,0){\line(0,1){12}}
\put(6,0){\line(0,1){12}}
\put(12,6){\line(0,1){6}}
\end{picture}}
\end{array}\end{equation}
(more on this shortly) with the short exact sequence
\begin{equation}\label{SES1}0\to\begin{picture}(12,12)(0,2)
\put(0,0){\line(1,0){12}}
\put(0,6){\line(1,0){12}}
\put(0,12){\line(1,0){12}}
\put(0,0){\line(0,1){12}}
\put(6,0){\line(0,1){12}}
\put(12,0){\line(0,1){12}}
\end{picture}\to\Wedge^1\otimes\begin{picture}(12,12)(0,2)
\put(0,0){\line(1,0){6}}
\put(0,6){\line(1,0){12}}
\put(0,12){\line(1,0){12}}
\put(0,0){\line(0,1){12}}
\put(6,0){\line(0,1){12}}
\put(12,6){\line(0,1){6}}
\end{picture}
\xrightarrow{\,\partial\,}
\Wedge^2\otimes
\begin{picture}(12,6)(0,0)
\put(0,0){\line(1,0){12}}
\put(0,6){\line(1,0){12}}
\put(0,0){\line(0,1){6}}
\put(6,0){\line(0,1){6}}
\put(12,0){\line(0,1){6}}
\end{picture}\to 0\end{equation}
identifying the bundle $W$ from Theorem~\ref{prolong} as 
\begin{picture}(12,12)(0,2)
\put(0,0){\line(1,0){12}}
\put(0,6){\line(1,0){12}}
\put(0,12){\line(1,0){12}}
\put(0,0){\line(0,1){12}}
\put(6,0){\line(0,1){12}}
\put(12,0){\line(0,1){12}}
\end{picture} and the short exact sequence
\begin{equation}\label{SES2}0\to\Wedge^1\otimes\begin{picture}(12,12)(0,2)
\put(0,0){\line(1,0){12}}
\put(0,6){\line(1,0){12}}
\put(0,12){\line(1,0){12}}
\put(0,0){\line(0,1){12}}
\put(6,0){\line(0,1){12}}
\put(12,0){\line(0,1){12}}
\end{picture}\to\Wedge^2\otimes\begin{picture}(12,12)(0,2)
\put(0,0){\line(1,0){6}}
\put(0,6){\line(1,0){12}}
\put(0,12){\line(1,0){12}}
\put(0,0){\line(0,1){12}}
\put(6,0){\line(0,1){12}}
\put(12,6){\line(0,1){6}}
\end{picture}
\xrightarrow{\,\partial\,}
\Wedge^3\otimes
\begin{picture}(12,6)(0,0)
\put(0,0){\line(1,0){12}}
\put(0,6){\line(1,0){12}}
\put(0,0){\line(0,1){6}}
\put(6,0){\line(0,1){6}}
\put(12,0){\line(0,1){6}}
\end{picture}\to 0\end{equation}
identifying the location of its curvature. To proceed, we now employ 
Theorem~\ref{prolong} once more (as presented abstractly following the proof 
of Theorem~\ref{prolong} and leading to its Addendum) to obtain the bundle 
\begin{equation}\label{second_prolonged_bundle}
E\oplus W=U\oplus V\oplus W=\begin{picture}(12,6)(0,0)
\put(0,0){\line(1,0){12}}
\put(0,6){\line(1,0){12}}
\put(0,0){\line(0,1){6}}
\put(6,0){\line(0,1){6}}
\put(12,0){\line(0,1){6}}
\end{picture}\oplus\begin{picture}(12,12)(0,2)
\put(0,0){\line(1,0){6}}
\put(0,6){\line(1,0){12}}
\put(0,12){\line(1,0){12}}
\put(0,0){\line(0,1){12}}
\put(6,0){\line(0,1){12}}
\put(12,6){\line(0,1){6}}
\end{picture}\oplus\begin{picture}(12,12)(0,2)
\put(0,0){\line(1,0){12}}
\put(0,6){\line(1,0){12}}
\put(0,12){\line(1,0){12}}
\put(0,0){\line(0,1){12}}
\put(6,0){\line(0,1){12}}
\put(12,0){\line(0,1){12}}
\end{picture}\end{equation}
and the diagram
$$\begin{array}{c} E\\ \oplus\\[2pt] W\end{array}=
\begin{array}{c}
E\\
\oplus\\[2pt]
\begin{picture}(12,12)(0,2)
\put(0,0){\line(1,0){12}}
\put(0,6){\line(1,0){12}}
\put(0,12){\line(1,0){12}}
\put(0,0){\line(0,1){12}}
\put(6,0){\line(0,1){12}}
\put(12,0){\line(0,1){12}}
\end{picture}
\end{array}
\begin{picture}(24,10)(0,-14)
\put(0,2){$\xrightarrow{\enskip\nabla\enskip}$}
\put(0,-22){\vector(3,2){28}}
\put(12,-8){\makebox(0,0){\footnotesize$\partial$}}
\put(0,-30){$\xrightarrow{\enskip\nabla\enskip}$}
\end{picture}
\begin{array}{c}
\Wedge^1\otimes E\\
\oplus\\[2pt]
\Wedge^1\otimes\begin{picture}(12,12)(0,2)
\put(0,0){\line(1,0){12}}
\put(0,6){\line(1,0){12}}
\put(0,12){\line(1,0){12}}
\put(0,0){\line(0,1){12}}
\put(6,0){\line(0,1){12}}
\put(12,0){\line(0,1){12}}
\end{picture}
\end{array}
\begin{picture}(24,10)(0,-14)
\put(0,2){$\xrightarrow{\enskip\nabla\enskip}$}
\put(0,-22){\vector(3,2){28}}
\put(12,-8){\makebox(0,0){\footnotesize$\partial$}}
\put(0,-30){$\xrightarrow{\enskip\nabla\enskip}$}
\end{picture}
\begin{array}{c}
\Wedge^2\otimes E\\
\oplus\\[2pt]
\Wedge^2\otimes\begin{picture}(12,12)(0,2)
\put(0,0){\line(1,0){12}}
\put(0,6){\line(1,0){12}}
\put(0,12){\line(1,0){12}}
\put(0,0){\line(0,1){12}}
\put(6,0){\line(0,1){12}}
\put(12,0){\line(0,1){12}}
\end{picture}
\end{array}
\begin{picture}(24,10)(0,-14)
\put(0,2){$\xrightarrow{\enskip\nabla\enskip}$}
\put(0,-22){\vector(3,2){28}}
\put(12,-8){\makebox(0,0){\footnotesize$\partial$}}
\end{picture}
\begin{array}{c}
\Wedge^3\otimes E,\\
\oplus\\[2pt]
\begin{picture}(6,12)(0,2)
\put(3,8){\makebox(0,0){$\vdots$}}
\end{picture}\end{array}$$
where $\nabla:E\to\Wedge^1\otimes E$ is the prolonged connection obtained
by one iteration of Theorem~\ref{prolong} and the inclusion
$\partial:W\to\Wedge^1\otimes E$ is obtained from the first homomorphism of the
exact sequence~(\ref{SES1}).  In order to invoke Theorem~\ref{prolong} a second
time, we just need to factor the curvature of the connection on $E$ through
$$\partial\wedge\underbar{\enskip}:\Wedge^1\otimes W
=\Wedge^1\otimes\begin{picture}(12,12)(0,2)
\put(0,0){\line(1,0){12}}
\put(0,6){\line(1,0){12}}
\put(0,12){\line(1,0){12}}
\put(0,0){\line(0,1){12}}
\put(6,0){\line(0,1){12}}
\put(12,0){\line(0,1){12}}
\end{picture}\xrightarrow{\,\partial\,}\Wedge^2\otimes E$$
but the short exact sequence (\ref{SES2}) implies that this is uniquely
possible.  Moreover, once this second prolongation is done, injectivity of this
same homomorphism implies that solutions of the original Killing equation
$\nabla_{(a}\sigma_{bc)}=0$ are now encoded by parallel sections of the
bundle~(\ref{second_prolonged_bundle}).  Note that the curvature of this bundle
is highly constrained.  According to conclusion~(\ref{Zwei}) of this second
implementation of Theorem~\ref{prolong}, we find, schematically, that
\begin{equation}\label{schematically}\nabla\wedge\nabla
\left[\!\begin{array}{c}\sigma\\ \mu\\ \rho\end{array}\!\right]=
\left[\!\begin{array}{c}0\\ 0\\ R\bowtie\rho+(\nabla R)\bowtie\mu+
(\nabla\nabla R)\bowtie\sigma\end{array}\!\right],\end{equation}
with each term on the last line lying in
\begin{equation}\label{bay_window}\begin{picture}(18,12)(0,2)
\put(0,0){\line(1,0){18}}
\put(0,6){\line(1,0){18}}
\put(0,12){\line(1,0){18}}
\put(0,0){\line(0,1){12}}
\put(6,0){\line(0,1){12}}
\put(12,0){\line(0,1){12}}
\put(18,0){\line(0,1){12}}
\end{picture}=\ker\partial\wedge\underbar{\enskip}:
\Wedge^2\otimes\begin{picture}(12,12)(0,2)
\put(0,0){\line(1,0){12}}
\put(0,6){\line(1,0){12}}
\put(0,12){\line(1,0){12}}
\put(0,0){\line(0,1){12}}
\put(6,0){\line(0,1){12}}
\put(12,0){\line(0,1){12}}
\end{picture}\longrightarrow\Wedge^3\otimes\begin{picture}(12,12)(0,2)
\put(0,0){\line(1,0){6}}
\put(0,6){\line(1,0){12}}
\put(0,12){\line(1,0){12}}
\put(0,0){\line(0,1){12}}
\put(6,0){\line(0,1){12}}
\put(12,6){\line(0,1){6}}
\end{picture}\,.\end{equation}
Finally, we return to the question of factorising the curvature
$$\nabla_{[a}\nabla_{b]}\sigma_{cd}=-R_{ab}{}^e{}_{(c}\sigma_{d)e}$$
as in~(\ref{factorisation}), noting that this is the only ambiguity in the 
double prolongation procedure leading to the final connection on 
$\begin{picture}(12,6)(0,0)
\put(0,0){\line(1,0){12}}
\put(0,6){\line(1,0){12}}
\put(0,0){\line(0,1){6}}
\put(6,0){\line(0,1){6}}
\put(12,0){\line(0,1){6}}
\end{picture}\oplus\begin{picture}(12,12)(0,2)
\put(0,0){\line(1,0){6}}
\put(0,6){\line(1,0){12}}
\put(0,12){\line(1,0){12}}
\put(0,0){\line(0,1){12}}
\put(6,0){\line(0,1){12}}
\put(12,6){\line(0,1){6}}
\end{picture}\oplus\begin{picture}(12,12)(0,2)
\put(0,0){\line(1,0){12}}
\put(0,6){\line(1,0){12}}
\put(0,12){\line(1,0){12}}
\put(0,0){\line(0,1){12}}
\put(6,0){\line(0,1){12}}
\put(12,0){\line(0,1){12}}
\end{picture}\,$.  Factorisations exist but, unfortunately, are not unique (as 
presented abstractly in the lead up to the Addendum to Theorem~\ref{prolong}).  
For example, if we realise $\begin{picture}(12,12)(0,2)
\put(0,0){\line(1,0){6}}
\put(0,6){\line(1,0){12}}
\put(0,12){\line(1,0){12}}
\put(0,0){\line(0,1){12}}
\put(6,0){\line(0,1){12}}
\put(12,6){\line(0,1){6}}
\end{picture}$ as (\ref{alternative_hook}) then, as a consequence of this 
freedom in choice of~$\tilde\kappa$, we can take the prolonged 
connection to be
\begin{equation}\label{first_prolongation}\begin{array}{c}
\begin{picture}(12,6)(0,0)
\put(0,0){\line(1,0){12}}
\put(0,6){\line(1,0){12}}
\put(0,0){\line(0,1){6}}
\put(6,0){\line(0,1){6}}
\put(12,0){\line(0,1){6}}
\end{picture}\\[-4pt]
\oplus\\[-2pt]
\begin{picture}(12,12)(0,2)
\put(0,0){\line(1,0){6}}
\put(0,6){\line(1,0){12}}
\put(0,12){\line(1,0){12}}
\put(0,0){\line(0,1){12}}
\put(6,0){\line(0,1){12}}
\put(12,6){\line(0,1){6}}
\end{picture}
\end{array}
\ni\left[\!\begin{array}{c}\sigma_{bc}\\ \mu_{bcd}\end{array}\!\right]
\stackrel{\nabla_a}{\longmapsto}
\left[\!\begin{array}{c}\nabla_a\sigma_{bc}+2\mu_{(bc)a}\\ 
\nabla_a\mu_{bcd}-(R\triangleleft\sigma)_{abcd}
\end{array}\!\right],\end{equation}
where
\begin{equation}\label{option_one}\textstyle (R\triangleleft\sigma)_{abcd}
=R_{cd}{}^e{}_a\sigma_{be}-R_{[cd}{}^e{}_{|a|}\sigma_{b]e}
=\frac23R_{cd}{}^e{}_a\sigma_{be}
-\frac13R_{bc}{}^e{}_a\sigma_{de}
+\frac13R_{bd}{}^e{}_a\sigma_{ce}\end{equation}
or
\begin{equation}\label{option_two}\textstyle (R\triangleleft\sigma)_{abcd}
=\frac5{12}R_{cd}{}^e{}_a\sigma_{be}+\frac14R_{cd}{}^e{}_b\sigma_{ae}
+\frac16R_{ab}{}^e{}_{[c}\sigma_{d]e}
+\frac13R_{ca}{}^e{}_b\sigma_{de}-\frac13R_{da}{}^e{}_b\sigma_{ce},
\end{equation}
the point being that these two options differ by 
\begin{equation}\label{freedom}
\textstyle\frac12R_{cd}{}^e{}_{[a}\sigma_{b]e}
+\frac12R_{ab}{}^e{}_{[c}\sigma_{d]e},\end{equation}
which lies in $\begin{picture}(12,12)(0,2)
\put(0,0){\line(1,0){12}}
\put(0,6){\line(1,0){12}}
\put(0,12){\line(1,0){12}}
\put(0,0){\line(0,1){12}}
\put(6,0){\line(0,1){12}}
\put(12,0){\line(0,1){12}}
\end{picture}$ and therefore, by~(\ref{SES1}), does not contribute
to~(\ref{factorisation}).  The second option is characterised by having
$(R\triangleleft\sigma)_{[ab]cd}+(R\triangleleft\sigma)_{[cd]ab}=0$.  At first 
glance, this option is the more natural since it is obtained by using the 
{\em canonical\/} splitting 
$$0\to\begin{picture}(12,12)(0,2)
\put(0,0){\line(1,0){12}}
\put(0,6){\line(1,0){12}}
\put(0,12){\line(1,0){12}}
\put(0,0){\line(0,1){12}}
\put(6,0){\line(0,1){12}}
\put(12,0){\line(0,1){12}}
\end{picture}\to\Wedge^1\otimes\begin{picture}(12,12)(0,2)
\put(0,0){\line(1,0){6}}
\put(0,6){\line(1,0){12}}
\put(0,12){\line(1,0){12}}
\put(0,0){\line(0,1){12}}
\put(6,0){\line(0,1){12}}
\put(12,6){\line(0,1){6}}
\end{picture}
\xrightarrow{\quad\,\partial\,\quad}
\Wedge^2\otimes
\begin{picture}(12,6)(0,0)
\put(0,0){\line(1,0){12}}
\put(0,6){\line(1,0){12}}
\put(0,0){\line(0,1){6}}
\put(6,0){\line(0,1){6}}
\put(12,0){\line(0,1){6}}
\end{picture}\to 0\begin{picture}(0,0)
\put(-44,-4){\line(0,-1){8}}
\put(-44,-12){\line(-1,0){85}}
\put(-129,-12){\vector(0,1){8}}
\put(-88,-6){\makebox(0,0){$\delta$}}
\end{picture}$$

\bigskip\noindent of (\ref{SES1}) and taking
$\tilde\kappa\equiv\delta\circ\kappa$.  It is the one employed
in~\cite{BCEG,E,GL,Heil17,HouriTomodaYasui18} but the first option
(\ref{option_one}) has the nice property that $2(R\triangleleft\sigma)_{b(cd)a}
=R_{ca}{}^f{}_b\sigma_{df}+R_{da}{}^f{}_b\sigma_{cf}$ and choosing this one
leads to the following prolongation connection for Killing $2$-tensors (which
will eventually (in Theorem~\ref{Phi_is_parallel} and its corollary) turn out
to be well adapted to the quadratic mapping (\ref{K1->K2}) from the Killing
$1$-forms to Killing $2$-tensors).

\begin{thm}\label{prolonged_2-tensors}
Suppose $\nabla_a$ is a torsion-free affine connection with curvature
$R_{ab}{}^c{}_d$. 
Then the
Killing $2$-tensors, namely $2$-tensors $\sigma_{bc}=\sigma_{(bc)}$ such that
$\nabla_{(a}\sigma_{bc)}=0$, are in $1$-$1$ correspondence with parallel
sections of the bundle
$\begin{picture}(12,6)(0,0)
\put(0,0){\line(1,0){12}}
\put(0,6){\line(1,0){12}}
\put(0,0){\line(0,1){6}}
\put(6,0){\line(0,1){6}}
\put(12,0){\line(0,1){6}}
\end{picture}\oplus
\begin{picture}(12,12)(0,2)
\put(0,0){\line(1,0){6}}
\put(0,6){\line(1,0){12}}
\put(0,12){\line(1,0){12}}
\put(0,0){\line(0,1){12}}
\put(6,0){\line(0,1){12}}
\put(12,6){\line(0,1){6}}
\end{picture}\oplus\begin{picture}(12,12)(0,2)
\put(0,0){\line(1,0){12}}
\put(0,6){\line(1,0){12}}
\put(0,12){\line(1,0){12}}
\put(0,0){\line(0,1){12}}
\put(6,0){\line(0,1){12}}
\put(12,0){\line(0,1){12}}
\end{picture}$\,,
realised as 
$$\begin{picture}(12,12)(0,2)
\put(0,0){\line(1,0){6}}
\put(0,6){\line(1,0){12}}
\put(0,12){\line(1,0){12}}
\put(0,0){\line(0,1){12}}
\put(6,0){\line(0,1){12}}
\put(12,6){\line(0,1){6}}
\end{picture}\equiv\{\mu_{bcd}=\mu_{b[cd]}\mid\mu_{[bcd]}=0\}
\quad\mbox{and}\quad
\begin{picture}(12,12)(0,2)
\put(0,0){\line(1,0){12}}
\put(0,6){\line(1,0){12}}
\put(0,12){\line(1,0){12}}
\put(0,0){\line(0,1){12}}
\put(6,0){\line(0,1){12}}
\put(12,0){\line(0,1){12}}
\end{picture}
=\{\rho_{bcde}=\rho_{[bc][de]}\mid\rho_{b[cde]}=0\}$$
and equipped with the following connection:
\begin{equation}\label{our_tricky_prolongation_connection}
\begin{array}{c}
\begin{picture}(12,6)(0,0)
\put(0,0){\line(1,0){12}}
\put(0,6){\line(1,0){12}}
\put(0,0){\line(0,1){6}}
\put(6,0){\line(0,1){6}}
\put(12,0){\line(0,1){6}}
\end{picture}\\[-4pt]
\oplus\\[-2pt]
\begin{picture}(12,12)(0,2)
\put(0,0){\line(1,0){6}}
\put(0,6){\line(1,0){12}}
\put(0,12){\line(1,0){12}}
\put(0,0){\line(0,1){12}}
\put(6,0){\line(0,1){12}}
\put(12,6){\line(0,1){6}}
\end{picture}\\[-2pt]
\oplus\\[-2pt]
\begin{picture}(12,12)(0,2)
\put(0,0){\line(1,0){12}}
\put(0,6){\line(1,0){12}}
\put(0,12){\line(1,0){12}}
\put(0,0){\line(0,1){12}}
\put(6,0){\line(0,1){12}}
\put(12,0){\line(0,1){12}}
\end{picture}
\end{array}
\ni\left[\!\begin{array}{c}\sigma_{bc}\\ \mu_{bcd}\\
\rho_{bcde}\end{array}\!\right]
\stackrel{\nabla_a}{\longmapsto}
\left[\!\begin{array}{c}\nabla_a\sigma_{bc}+2\mu_{(bc)a}\\ 
\nabla_a\mu_{bcd}-(R\triangleleft\sigma)_{abcd}-\rho_{abcd}\\
\nabla_a\rho_{bcde}-(R\diamond\mu)_{abcde}+(\nabla\!R \ast \sigma)_{abcde}
\end{array}\!\right],\end{equation}
where 
\begin{equation}\label{lefttriangle}\textstyle (R\triangleleft\sigma)_{abcd}
\equiv\frac23R_{cd}{}^e{}_a\sigma_{be}
-\frac13R_{bc}{}^e{}_a\sigma_{de}
+\frac13R_{bd}{}^e{}_a\sigma_{ce},\end{equation}
 $(R\diamond\mu)_{abcde}\equiv$
\begin{equation}
\label{R_diamond_mu}
\begin{array}{l}
R_{bc}{}^f{}_a\mu_{fde}+R_{de}{}^f{}_a\mu_{fbc}
-\frac12R_{be}{}^f{}_a\mu_{fcd}
-\frac12R_{cd}{}^f{}_a\mu_{fbe}
+\frac12R_{bd}{}^f{}_a\mu_{fce}
+\frac12R_{ce}{}^f{}_a\mu_{fbd}\\[4pt]
\enskip{}+\frac23
\big(R_{a(b}{}^f{}_{d)}\mu_{(ce)f}-R_{a(c}{}^f{}_{d)}\mu_{(be)f}
-R_{a(b}{}^f{}_{e)}\mu_{(cd)f}+R_{a(c}{}^f{}_{e)}\mu_{(bd)f}\big)\\[4pt]
\quad{}-\frac23\big(R_{bc}{}^f{}_d\mu_{(ea)f}-R_{bc}{}^f{}_e\mu_{(da)f}
+R_{de}{}^f{}_b\mu_{(ca)f}-R_{de}{}^f{}_c\mu_{(ba)f}\big),\end{array}
\end{equation}
and 
\begin{equation}\label{asterisk}
\begin{array}{rcl}(\nabla\!R\ast\sigma)_{abcde}
&\!\!\!\equiv\!\!\!&\frac23\left(
\sigma_{f[b}\nabla_{c]}R_{de}{}^f{}_a+(\nabla_aR_{de}{}^f{}_{[b})\sigma_{c]f}
+\sigma_{f[d}\nabla_{e]}R_{bc}{}^f{}_a
+(\nabla_aR_{bc}{}^f{}_{[d})\sigma_{e]f}\right)\\[5pt]
&\!\!&{}+\frac13\sigma_{af}
\left(\nabla_{c}R_{de}{}^f{}_{b}-\nabla_{b}R_{de}{}^f{}_{c}\right).
\end{array}\hspace{-8pt}
\end{equation}
\end{thm}
\begin{proof} Firstly, we should check the formula for
$(R\triangleleft\sigma)_{abcd}$, already claimed as an option
in~(\ref{option_one}).  According to (\ref{factorisation}), since
$\nabla_{[a}\nabla_{b]}\sigma_{cd}=-R_{ab}{}^e{}_{(c}\sigma_{d)e}$, we are
required to check that
\begin{equation}\label{to_be_checked}
(R\triangleleft\sigma)_{b(cd)a}-(R\triangleleft\sigma)_{a(cd)b}
=-R_{ab}{}^e{}_{(c}\sigma_{d)e}\end{equation}
but we have already observed that
$2(R\triangleleft\sigma)_{b(cd)a}
=R_{ca}{}^e{}_b\sigma_{de}+R_{da}{}^e{}_b\sigma_{ce}$
so
$$(R\triangleleft\sigma)_{b(cd)a}-(R\triangleleft\sigma)_{a(cd)b}
=R_{c[a}{}^e{}_{b]}\sigma_{de}+R_{d[a}{}^e{}_{b]}\sigma_{ce}$$
and the Bianchi symmetry in the form 
$R_{c[a}{}^e{}_{b]}\sigma_{de}=-\frac12R_{ab}{}^e{}_c\sigma_{de}$ 
yields~(\ref{to_be_checked}), as required.

It remains to check (\ref{factorisation}) at the second stage and, for this, we 
need to know the curvature of the connection
\begin{equation}\label{first_stage_connection}
\nabla_a\left[\!\begin{array}{c}\sigma_{bc}\\ \mu_{bcd}\end{array}\!\right]
=\left[\!\begin{array}{c}\nabla_a\sigma_{bc}+2\mu_{(bc)a}\\ 
\nabla_a\mu_{bcd}-(R\triangleleft\sigma)_{abcd}
\end{array}\!\right]\end{equation}
with $(R\triangleleft\sigma)_{abcd}$ given by (\ref{option_one}).  For clarity,
we restrict momentarily to the locally symmetric
case~$\nabla_aR_{bc}{}^d{}_e=0$, in which the general formula
\begin{equation}
\label{general_formula}
\nabla\wedge\nabla\left[\!\begin{array}{c}\phi\\ 
\psi\end{array}\!\right]=\left[\!\begin{array}{c}0\\ 
\lambda\psi-\tilde\kappa\wedge\partial\psi+(\nabla\wedge\tilde\kappa)\phi
\end{array}\!\right]
\end{equation}
as in (\ref{prolonged_curvature}), reduces to
$$\nabla\wedge\nabla\left[\!\begin{array}{c}\sigma\\ 
\mu\end{array}\!\right]=\left[\!\begin{array}{c}0\\ 
\nabla\wedge\nabla\mu-\tilde\kappa\wedge\partial\mu\end{array}\!\right],$$
where $\tilde\kappa\wedge\underbar{\enskip}:
\Wedge^1\otimes\begin{picture}(12,6)(0,0)
\put(0,0){\line(1,0){12}}
\put(0,6){\line(1,0){12}}
\put(0,0){\line(0,1){6}}
\put(6,0){\line(0,1){6}}
\put(12,0){\line(0,1){6}}
\end{picture}\to
\Wedge^2\otimes\begin{picture}(12,12)(0,2)
\put(0,0){\line(1,0){6}}
\put(0,6){\line(1,0){12}}
\put(0,12){\line(1,0){12}}
\put(0,0){\line(0,1){12}}
\put(6,0){\line(0,1){12}}
\put(12,6){\line(0,1){6}}
\end{picture}$ is induced by~(\ref{option_one}).  Explicitly, we find
$$\textstyle(\nabla\wedge\nabla\mu)_{abcde}=\nabla_{[a}\nabla_{b]}\mu_{cde}
=-\frac12R_{ab}{}^f{}_c\mu_{fde}-\frac12R_{ab}{}^f{}_d\mu_{cfe}
-\frac12R_{ab}{}^f{}_e\mu_{cdf}$$
whilst
$$\begin{picture}(12,12)(0,2)
\put(0,0){\line(1,0){6}}
\put(0,6){\line(1,0){12}}
\put(0,12){\line(1,0){12}}
\put(0,0){\line(0,1){12}}
\put(6,0){\line(0,1){12}}
\put(12,6){\line(0,1){6}}
\end{picture}\ni\mu\mapsto
\partial\mu\in\Wedge^1\otimes
\begin{picture}(12,6)(0,0)
\put(0,0){\line(1,0){12}}
\put(0,6){\line(1,0){12}}
\put(0,0){\line(0,1){6}}
\put(6,0){\line(0,1){6}}
\put(12,0){\line(0,1){6}}
\end{picture}\quad\mbox{is}\enskip
\mu_{cde}\mapsto-2\mu_{(cd)b}$$
and so, according to~(\ref{option_one}), 
$$\Wedge^1\otimes\begin{picture}(12,6)(0,0)
\put(0,0){\line(1,0){12}}
\put(0,6){\line(1,0){12}}
\put(0,0){\line(0,1){6}}
\put(6,0){\line(0,1){6}}
\put(12,0){\line(0,1){6}}
\end{picture}\ni\sigma\mapsto\tilde\kappa\wedge\sigma\in
\Wedge^2\otimes\begin{picture}(12,12)(0,2)
\put(0,0){\line(1,0){6}}
\put(0,6){\line(1,0){12}}
\put(0,12){\line(1,0){12}}
\put(0,0){\line(0,1){12}}
\put(6,0){\line(0,1){12}}
\put(12,6){\line(0,1){6}}
\end{picture}\quad\mbox{is}\enskip
\sigma_{bcd}\mapsto\textstyle
-\frac23R_{de}{}^f{}_{[a}\sigma_{b]cf}
+\frac13R_{cd}{}^f{}_{[a}\sigma_{b]ef}
-\frac13R_{ce}{}^f{}_{[a}\sigma_{b]df}$$
whence $(\tilde\kappa\wedge\partial\mu)_{abcde}$ is given by
$$\textstyle
\frac23R_{de}{}^f{}_a\mu_{(cf)b}-\frac23R_{de}{}^f{}_b\mu_{(cf)a}
-\frac13R_{cd}{}^f{}_a\mu_{(ef)b}+\frac13R_{cd}{}^f{}_b\mu_{(ef)a}
+\frac13R_{ce}{}^f{}_a\mu_{(df)b}-\frac13R_{ce}{}^f{}_b\mu_{(df)a}.$$
Altogether then, the curvature of (\ref{first_stage_connection}) is
\begin{equation}\label{curvature_of_first_stage_connection}
\nabla_{[a}\nabla_{b]}
\left[\!\begin{array}{c}\sigma_{cd}\\ \mu_{cde}\end{array}\!\right]
=\left[\!\begin{array}{c}0\\ 
\left\lceil\!\!\!\!\begin{array}{l}
{}-\frac12R_{ab}{}^f{}_c\mu_{fde}-\frac12R_{ab}{}^f{}_d\mu_{cfe}
-\frac12R_{ab}{}^f{}_e\mu_{cdf}\\
\quad{}-\frac23R_{de}{}^f{}_a\mu_{(cf)b}+\frac23R_{de}{}^f{}_b\mu_{(cf)a}
+\frac13R_{cd}{}^f{}_a\mu_{(ef)b}\\
\qquad{}-\frac13R_{cd}{}^f{}_b\mu_{(ef)a}
-\frac13R_{ce}{}^f{}_a\mu_{(df)b}+\frac13R_{ce}{}^f{}_b\mu_{(df)a}
\end{array}\!\right\rfloor
\end{array}\!\right]\!.\end{equation}
On the other hand, from (\ref{R_diamond_mu}) we note that 
$R\diamond\mu\in\Wedge^1\otimes\begin{picture}(12,12)(0,2)
\put(0,0){\line(1,0){12}}
\put(0,6){\line(1,0){12}}
\put(0,12){\line(1,0){12}}
\put(0,0){\line(0,1){12}}
\put(6,0){\line(0,1){12}}
\put(12,0){\line(0,1){12}}
\end{picture}$\,, as it should be (indeed, each line of $R\diamond\mu$ has 
this property), and now let us compute
$(R\diamond\mu)_{[ab]cde}$. It is convenient to do this one line at a time, 
firstly obtaining (using Bianchi symmetry 
$R_{c[b}{}^f{}_{a]}=\frac12R_{ab}{}^f{}_c$)
$$\begin{array}{l}
-\frac12R_{ab}{}^f{}_c\mu_{fde}+\frac14R_{ab}{}^f{}_e\mu_{fcd}
-\frac14R_{ab}{}^f{}_d\mu_{fce}\\[4pt]
\enskip{}+\frac12R_{de}{}^f{}_a\mu_{fbc}
-\frac12R_{de}{}^f{}_b\mu_{fac}
-\frac14R_{cd}{}^f{}_a\mu_{fbe}+\frac14R_{cd}{}^f{}_b\mu_{fae}
+\frac14R_{ce}{}^f{}_a\mu_{fbd}-\frac14R_{ce}{}^f{}_b\mu_{fad},
\end{array}$$
secondly obtaining (again using Bianchi symmetry)
$$\begin{array}{l}
\frac14R_{ab}{}^f{}_d\mu_{cef}+\frac14R_{ab}{}^f{}_d\mu_{ecf}
-\frac14R_{ab}{}^f{}_e\mu_{cdf}-\frac14R_{ab}{}^f{}_e\mu_{dcf}\\[4pt]
\enskip{}-\frac1{12}R_{ac}{}^f{}_d\mu_{bef}-\frac1{12}R_{ac}{}^f{}_d\mu_{ebf}
-\frac1{12}R_{ad}{}^f{}_c\mu_{bef}-\frac1{12}R_{ad}{}^f{}_c\mu_{ebf}\\[4pt]
\quad{}+\frac1{12}R_{bc}{}^f{}_d\mu_{aef}+\frac1{12}R_{bc}{}^f{}_d\mu_{eaf}
+\frac1{12}R_{bd}{}^f{}_c\mu_{aef}+\frac1{12}R_{bd}{}^f{}_c\mu_{eaf}\\[4pt]
\enskip\quad{}
+\frac1{12}R_{ac}{}^f{}_e\mu_{bdf}+\frac1{12}R_{ac}{}^f{}_e\mu_{dbf}
+\frac1{12}R_{ae}{}^f{}_c\mu_{bdf}+\frac1{12}R_{ae}{}^f{}_c\mu_{dbf}\\[4pt]
\qquad{}-\frac1{12}R_{bc}{}^f{}_e\mu_{adf}-\frac1{12}R_{bc}{}^f{}_e\mu_{daf}
-\frac1{12}R_{be}{}^f{}_c\mu_{adf}-\frac1{12}R_{be}{}^f{}_c\mu_{abf},
\end{array}$$
and thirdly obtaining
$$\begin{array}{l}
-\frac16R_{bc}{}^f{}_d\mu_{eaf}-\frac16R_{bc}{}^f{}_d\mu_{aef}
+\frac16R_{bc}{}^f{}_e\mu_{daf}+\frac16R_{bc}{}^f{}_e\mu_{adf}
-\frac16R_{de}{}^f{}_b\mu_{caf}-\frac16R_{de}{}^f{}_b\mu_{acf}\\[4pt]
\enskip{}+\frac16R_{ac}{}^f{}_d\mu_{ebf}+\frac16R_{ac}{}^f{}_d\mu_{bef}
-\frac16R_{ac}{}^f{}_e\mu_{dbf}-\frac16R_{ac}{}^f{}_e\mu_{bdf}
+\frac16R_{de}{}^f{}_a\mu_{cbf}+\frac16R_{de}{}^f{}_a\mu_{bcf}.
\end{array}$$
It remains to verify that these quantities sum to the expression
$\left\lceil\quad\right\rfloor$ in the
curvature~(\ref{curvature_of_first_stage_connection}).  This verification may
be accomplished term-by-term.  For example,
\begin{itemize}
\item we see $-\frac12R_{ab}{}^f{}_c\mu_{fde}$ in both with no other 
occurrences of $R_{ab}{}^f{}_c$,
\item $\left\lceil\quad\right\rfloor$ contains 
$-\frac12R_{ab}{}^f{}_d\mu_{cfe}$ whereas the only occurrences of 
$R_{ab}{}^f{}_d$ in $(R\diamond\mu)_{[ab]cde}$ combine as
$\frac14R_{ab}{}^f{}_d(\mu_{cef}+\mu_{ecf}-\mu_{fce})$, whilst 
$\mu_{cef}+\mu_{ecf}-\mu_{fce}=-2\mu_{cfe}$ by the symmetries of $\mu_{cef}$,
\item In $\left\lceil\quad\right\rfloor$ we see 
$\frac13R_{cd}{}^f{}_a\mu_{(ef)b}$ as the only occurrence of terms with
indices $acd$ on the curvature tensor $R_{\bullet\bullet}{}^f{}_\bullet$ 
whereas the contributions from $(R\diamond\mu)_{[ab]cde}$ are
$$\begin{array}{l}-\frac14R_{cd}{}^f{}_a\mu_{fbe}
-\frac1{12}R_{ac}{}^f{}_d\mu_{bef}-\frac1{12}R_{ac}{}^f{}_d\mu_{ebf}
-\frac1{12}R_{ad}{}^f{}_c\mu_{bef}-\frac1{12}R_{ad}{}^f{}_c\mu_{ebf}\\[4pt]
\enskip{}+\frac16R_{ac}{}^f{}_d\mu_{ebf}+\frac16R_{ac}{}^f{}_d\mu_{bef},
\end{array}$$
which simplify, with Bianchi symmetry, as
$$\textstyle-\frac14R_{cd}{}^f{}_a\mu_{fbe}
-\frac1{12}R_{cd}{}^f{}_a\mu_{bef}
-\frac1{12}R_{cd}{}^f{}_a\mu_{ebf}.$$
However, the symmetries of $\mu_{bef}$ show that 
$3\mu_{fbe}+\mu_{bef}+\mu_{ebf}
=-4\mu_{(ef)b}$
whence these terms match perfectly.
\end{itemize}
The remaining verifications in the locally symmetric setting are left to the
reader.

\medskip

To complete the proof we must consider the general case when
$\nabla_aR_{bc}{}^d{}_e$ may not vanish.  In this case all terms involving
$\mu_{abc}$ remain unchanged, but the general formula~(\ref{general_formula})
for the curvature of the connection~(\ref{first_stage_connection}) contains the
extra term $(\nabla\wedge\tilde\kappa)\phi$.  Explicitly, this term is
\begin{equation}
\label{nabRtrisigma}
\begin{array}{rcl}
 (\nabla_{[a} R\triangleleft \sigma)_{b]cde}&=&
\frac13 \sigma_{fc}(\nabla_a R_{de}{}^f{}{}_b-\nabla_b R_{de}{}^f{}{}_a)\\[5pt]
&&{}+\frac16 \sigma_{fd}(\nabla_bR_{ec}{}^f{}{}_a- \nabla_a R_{ec}{}^f{}{}_b)
    +\frac16\sigma_{fe}(\nabla_aR_{dc}{}^f{}{}_b-\nabla_bR_{dc}{}^f{}{}_a).
\end{array}
\end{equation}
We have to compare this to 
\begin{eqnarray*}
3(\nabla\!R\ast \sigma)_{[ab]cde}
&=&\sigma_{f[b}\nabla_{c]}R_{de}{}^f{}_a
   +(\nabla_aR_{de}{}^f{}_{[b})\sigma_{c]f}
  +\sigma_{f[d}\nabla_{e]}R_{bc}{}^f{}_a 
   +(\nabla_aR_{bc}{}^f{}_{[d})\sigma_{e]f}\\
&&{}-\sigma_{f[a}\nabla_{c]}R_{de}{}^f{}_b
    -(\nabla_bR_{de}{}^f{}_{[a})\sigma_{c]f}
    -\sigma_{f[d}\nabla_{e]}R_{ac}{}^f{}_b 
    -(\nabla_bR_{ac}{}^f{}_{[d})\sigma_{e]f}\\
&&{}+\tfrac12 \sigma_{af}
  \left(\nabla_{c}R_{de}{}^f{}_{b}-\nabla_{b}R_{de}{}^f{}_{c}\right)
  -\tfrac12 \sigma_{bf}
  \left(\nabla_{c}R_{de}{}^f{}_{a}-\nabla_{a}R_{de}{}^f{}_{c}\right).
\end{eqnarray*}
Note that all $\sigma_{fa}$- and $\sigma_{fb}$-terms cancel, so that
\begin{eqnarray*}
3(\nabla\!R\ast \sigma)_{[ab]cde}
&=&\sigma_{fc}(-\nabla_{b}R_{de}{}^f{}_a+\nabla_aR_{de}{}^f{}_{b})\\
&&{}+\tfrac12 \sigma_{fd}(\nabla_{e}R_{bc}{}^f{}_a-\nabla_aR_{bc}{}^f{}_{e}
 -\nabla_{e}R_{ac}{}^f{}_b+\nabla_bR_{ac}{}^f{}_{e})\\
&&{}-\tfrac12\sigma_{fe}( \nabla_{d}R_{bc}{}^f{}_a-\nabla_aR_{bc}{}^f{}_{d}
 -\nabla_{d}R_{ac}{}^f{}_b +\nabla_bR_{ac}{}^f{}_{d})\\
&=&\sigma_{fc}(\nabla_aR_{de}{}^f{}_{b}-\nabla_{b}R_{de}{}^f{}_a)\\
&&{}+\tfrac12\sigma_{fd}(\nabla_cR_{ab}{}^f{}_{e}-\nabla_{e}R_{ab}{}^f{}_c)
    +\tfrac12\sigma_{fe}(\nabla_{d}R_{ab}{}^f{}_c-\nabla_cR_{ab}{}^f{}_{d}),
\end{eqnarray*}
by the Bianchi symmetry and the Bianchi identity.  Using both again to check
that
$$\nabla_bR_{ec}{}^f{}{}_a- \nabla_a R_{ec}{}^f{}{}_b=
\nabla_cR_{ab}{}^f{}{}_e- \nabla_e R_{ab}{}^f{}{}_c,$$
and comparing to~(\ref{nabRtrisigma}), we obtain 
$$(\nabla \!R\ast \sigma)_{[ab]cde}
=(\nabla_{[a} R\triangleleft \sigma)_{b]cde},$$
completing the proof.
\end{proof}

Given the somewhat complicated formul{\ae} in
Theorem~\ref{prolonged_2-tensors}, one might wonder whether there is a simpler
prolongation connection on this bundle
$\begin{picture}(12,6)(0,0)
\put(0,0){\line(1,0){12}}
\put(0,6){\line(1,0){12}}
\put(0,0){\line(0,1){6}}
\put(6,0){\line(0,1){6}}
\put(12,0){\line(0,1){6}}
\end{picture}\oplus
\begin{picture}(12,12)(0,2)
\put(0,0){\line(1,0){6}}
\put(0,6){\line(1,0){12}}
\put(0,12){\line(1,0){12}}
\put(0,0){\line(0,1){12}}
\put(6,0){\line(0,1){12}}
\put(12,6){\line(0,1){6}}
\end{picture}\oplus\begin{picture}(12,12)(0,2)
\put(0,0){\line(1,0){12}}
\put(0,6){\line(1,0){12}}
\put(0,12){\line(1,0){12}}
\put(0,0){\line(0,1){12}}
\put(6,0){\line(0,1){12}}
\put(12,0){\line(0,1){12}}
\end{picture}$\,. 
As already remarked, we may realise the constituent bundles 
$\begin{picture}(12,12)(0,2)
\put(0,0){\line(1,0){6}}
\put(0,6){\line(1,0){12}}
\put(0,12){\line(1,0){12}}
\put(0,0){\line(0,1){12}}
\put(6,0){\line(0,1){12}}
\put(12,6){\line(0,1){6}}
\end{picture}$ and/or
$\begin{picture}(12,12)(0,2)
\put(0,0){\line(1,0){12}}
\put(0,6){\line(1,0){12}}
\put(0,12){\line(1,0){12}}
\put(0,0){\line(0,1){12}}
\put(6,0){\line(0,1){12}}
\put(12,0){\line(0,1){12}}
\end{picture}$ 
differently.  Whilst this may yield different formul{\ae}, this will not
really affect the properties of this connection: it's still the same 
connection on the same abstract bundle.  More importantly, however,
there was a choice in our construction, where we preferred (\ref{option_one})
over (\ref{option_two}) in defining $(R\triangleleft\sigma)_{abcd}$ for
the connection (\ref{first_prolongation}) on 
$\begin{picture}(12,6)(0,0)
\put(0,0){\line(1,0){12}}
\put(0,6){\line(1,0){12}}
\put(0,0){\line(0,1){6}}
\put(6,0){\line(0,1){6}}
\put(12,0){\line(0,1){6}}
\end{picture}\oplus
\begin{picture}(12,12)(0,2)
\put(0,0){\line(1,0){6}}
\put(0,6){\line(1,0){12}}
\put(0,12){\line(1,0){12}}
\put(0,0){\line(0,1){12}}
\put(6,0){\line(0,1){12}}
\put(12,6){\line(0,1){6}}
\end{picture}$ 
en route to (\ref{our_tricky_prolongation_connection}) on
$\begin{picture}(12,6)(0,0)
\put(0,0){\line(1,0){12}}
\put(0,6){\line(1,0){12}}
\put(0,0){\line(0,1){6}}
\put(6,0){\line(0,1){6}}
\put(12,0){\line(0,1){6}}
\end{picture}\oplus
\begin{picture}(12,12)(0,2)
\put(0,0){\line(1,0){6}}
\put(0,6){\line(1,0){12}}
\put(0,12){\line(1,0){12}}
\put(0,0){\line(0,1){12}}
\put(6,0){\line(0,1){12}}
\put(12,6){\line(0,1){6}}
\end{picture}\oplus\begin{picture}(12,12)(0,2)
\put(0,0){\line(1,0){12}}
\put(0,6){\line(1,0){12}}
\put(0,12){\line(1,0){12}}
\put(0,0){\line(0,1){12}}
\put(6,0){\line(0,1){12}}
\put(12,0){\line(0,1){12}}
\end{picture}$\,. 
This choice really does  give different connections on 
$\begin{picture}(12,6)(0,0)
\put(0,0){\line(1,0){12}}
\put(0,6){\line(1,0){12}}
\put(0,0){\line(0,1){6}}
\put(6,0){\line(0,1){6}}
\put(12,0){\line(0,1){6}}
\end{picture}\oplus
\begin{picture}(12,12)(0,2)
\put(0,0){\line(1,0){6}}
\put(0,6){\line(1,0){12}}
\put(0,12){\line(1,0){12}}
\put(0,0){\line(0,1){12}}
\put(6,0){\line(0,1){12}}
\put(12,6){\line(0,1){6}}
\end{picture}$ 
but, as already noticed in~(\ref{freedom}) and detailed in the Addendum to 
Theorem~\ref{prolong}, this general freedom in
choosing $(R\triangleleft\sigma)_{abcd}$ can be absorbed into an automorphism 
$$\left[\!\begin{array}{c}\sigma_{ab}\\ \mu_{abc}\\
\rho_{abcd}\end{array}\!\right]\stackrel{\mbox{$\Phi$}}{\longmapsto}
\left[\!\begin{array}{c}\sigma_{ab}\\ \mu_{abc}\\
\rho_{abcd}
\end{array}\!\right]
+\frac12
\left[\!\begin{array}{c}0\\ 0\\
R_{ab}{}^e{}_{[c}\sigma_{d]e}+R_{cd}{}^e{}_{[a}\sigma_{b]e}
\end{array}\!\right]$$
of $\begin{picture}(12,6)(0,0)
\put(0,0){\line(1,0){12}}
\put(0,6){\line(1,0){12}}
\put(0,0){\line(0,1){6}}
\put(6,0){\line(0,1){6}}
\put(12,0){\line(0,1){6}}
\end{picture}\oplus
\begin{picture}(12,12)(0,2)
\put(0,0){\line(1,0){6}}
\put(0,6){\line(1,0){12}}
\put(0,12){\line(1,0){12}}
\put(0,0){\line(0,1){12}}
\put(6,0){\line(0,1){12}}
\put(12,6){\line(0,1){6}}
\end{picture}\oplus\begin{picture}(12,12)(0,2)
\put(0,0){\line(1,0){12}}
\put(0,6){\line(1,0){12}}
\put(0,12){\line(1,0){12}}
\put(0,0){\line(0,1){12}}
\put(6,0){\line(0,1){12}}
\put(12,0){\line(0,1){12}}
\end{picture}$\,.  
Hence, regarding this full blown prolongation construction, it's still the same
connection on the same abstract bundle and, as we shall see, has some very good
properties, especially concerning the short exact sequence~(\ref{SES}).

\subsection{Killing $2$-tensors from Killing $1$-forms}
\label{2from1-sec}
We are now in a position to relate the prolongation connection 
(\ref{our_tricky_prolongation_connection})
on $\begin{picture}(12,6)(0,0)
\put(0,0){\line(1,0){12}}
\put(0,6){\line(1,0){12}}
\put(0,0){\line(0,1){6}}
\put(6,0){\line(0,1){6}}
\put(12,0){\line(0,1){6}}
\end{picture}\oplus
\begin{picture}(12,12)(0,2)
\put(0,0){\line(1,0){6}}
\put(0,6){\line(1,0){12}}
\put(0,12){\line(1,0){12}}
\put(0,0){\line(0,1){12}}
\put(6,0){\line(0,1){12}}
\put(12,6){\line(0,1){6}}
\end{picture}\oplus\begin{picture}(12,12)(0,2)
\put(0,0){\line(1,0){12}}
\put(0,6){\line(1,0){12}}
\put(0,12){\line(1,0){12}}
\put(0,0){\line(0,1){12}}
\put(6,0){\line(0,1){12}}
\put(12,0){\line(0,1){12}}
\end{picture}$
for Killing $2$-tensors 
 obtained in
Theorem~\ref{prolonged_2-tensors} to the prolongation connection
(\ref{Killing_connection})
on $\begin{picture}(6,6)(0,0)
\put(0,0){\line(1,0){6}}
\put(0,6){\line(1,0){6}}
\put(0,0){\line(0,1){6}}
\put(6,0){\line(0,1){6}}
\end{picture}\oplus\begin{picture}(6,12)(0,2)
\put(0,0){\line(1,0){6}}
\put(0,6){\line(1,0){6}}
\put(0,12){\line(1,0){6}}
\put(0,0){\line(0,1){12}}
\put(6,0){\line(0,1){12}}
\end{picture}$ for Killing $1$-forms obtained in
Theorem~\ref{killing_prolongation}.  
We start by looking at the connection induced by (\ref{Killing_connection})
on its symmetric power:
\begin{equation}\label{symmetric_power}
\textstyle\bigodot^2\!\left(\!\begin{array}{c}
\begin{picture}(6,6)(0,0)
\put(0,0){\line(1,0){6}}
\put(0,6){\line(1,0){6}}
\put(0,0){\line(0,1){6}}
\put(6,0){\line(0,1){6}}
\end{picture}\\[-4pt] \oplus\\[-2pt]
\begin{picture}(6,12)(0,2)
\put(0,0){\line(1,0){6}}
\put(0,6){\line(1,0){6}}
\put(0,12){\line(1,0){6}}
\put(0,0){\line(0,1){12}}
\put(6,0){\line(0,1){12}}
\end{picture}\end{array}\!\right)=\begin{array}{c}
\begin{picture}(12,6)(0,0)
\put(0,0){\line(1,0){12}}
\put(0,6){\line(1,0){12}}
\put(0,0){\line(0,1){6}}
\put(6,0){\line(0,1){6}}
\put(12,0){\line(0,1){6}}
\end{picture}\\[-4pt]
\oplus\\[-2pt]
\begin{picture}(6,6)(0,0)
\put(0,0){\line(1,0){6}}
\put(0,6){\line(1,0){6}}
\put(0,0){\line(0,1){6}}
\put(6,0){\line(0,1){6}}
\end{picture}\otimes
\begin{picture}(6,12)(0,2)
\put(0,0){\line(1,0){6}}
\put(0,6){\line(1,0){6}}
\put(0,12){\line(1,0){6}}
\put(0,0){\line(0,1){12}}
\put(6,0){\line(0,1){12}}
\end{picture}\\[-2pt]
\oplus\\[-2pt]
\begin{picture}(6,12)(0,2)
\put(0,0){\line(1,0){6}}
\put(0,6){\line(1,0){6}}
\put(0,12){\line(1,0){6}}
\put(0,0){\line(0,1){12}}
\put(6,0){\line(0,1){12}}
\end{picture}\bigodot
\begin{picture}(6,12)(0,2)
\put(0,0){\line(1,0){6}}
\put(0,6){\line(1,0){6}}
\put(0,12){\line(1,0){6}}
\put(0,0){\line(0,1){12}}
\put(6,0){\line(0,1){12}}
\end{picture}
\end{array}
\ni\left[\!\begin{array}{c}\sigma_{bc}\\ \mu_{bcd}\\
\rho_{bcde}\end{array}\!\right]
\stackrel{\nabla_a}{\longmapsto}
\left[\!\begin{array}{c}\nabla_a\sigma_{bc}+2\mu_{(bc)a}\\ 
\nabla_a\mu_{bcd}-R_{cd}{}^e{}_a\sigma_{be}-\rho_{abcd}\\
\nabla_a\rho_{bcde}-R_{bc}{}^f{}_a\mu_{fde}-R_{de}{}^f{}_a\mu_{fbc}
\end{array}\!\right].\end{equation}
\begin{lemma}\label{symmetric_curvature}
The curvature of the connection \eqref{symmetric_power} is given by
\begin{equation}\label{induced_curvature}\nabla_{[a}\nabla_{b]}
\left[\!\begin{array}{c}\sigma_{cd}\\ \mu_{cde}\\
\rho_{cdef}\end{array}\!\right]
=\left[\!\begin{array}{c}0\\ 
R_{ab}{}^f{}_{[d}\mu_{|c|e]f}
+R_{de}{}^f{}_{[a}\mu_{|c|b]f}
-(\nabla_{[a}R_{|de|}{}^f{}_{b]})\sigma_{fc}\\[4pt]
\left\lceil\!\!\!\!\begin{array}{c}
R_{ab}{}^p{}_{[c}\rho_{d]pef}+R_{cd}{}^p{}_{[a}\rho_{b]pef}
+R_{ab}{}^p{}_{[e}\rho_{f]pcd}+R_{ef}{}^p{}_{[a}\rho_{b]pcd}\\[4pt]
{}-(\nabla_{[a}R_{|cd|}{}^p{}_{b]})\mu_{pef}
-(\nabla_{[a}R_{|ef|}{}^p{}_{b]})\mu_{pcd}
\end{array}\!\right\rfloor
\end{array}\!\right].\end{equation}
\end{lemma}
\begin{proof} It follows from (\ref{full_Killing_curvature}) and
(\ref{in_the_locally_symmetric_case}) that
$$\nabla_{[a}\nabla_{b]}
\left[\!\begin{array}{c}\sigma_d\\ \mu_{de}\end{array}\!\right]
=\left[\!\begin{array}{c}0\\ 
R_{ab}{}^f{}_{[d}\mu_{e]f}+R_{de}{}^f{}_{[a}\mu_{b]f} 
-(\nabla_{[a}R_{|de|}{}^f{}_{b]})\sigma_{f}\end{array}\!\right]$$
and (\ref{induced_curvature}) follows because it is true in case 
$\sigma_{cd}=\sigma_c\sigma_d$, $\mu_{cde}=\sigma_c\mu_{cd}$, and 
$\rho_{cdef}=\mu_{cd}\mu_{ef}$.
\end{proof}
It is remarkable that one can modify the induced connection
(\ref{symmetric_power}) by means of its own curvature (\ref{induced_curvature})
without changing its parallel sections. The key to this modification is the 
following algebraic fact.
\begin{lemma} The formula
\begin{equation}\label{fun_formula}
X_{abcde}\longmapsto X_{abcde}-2X_{[bc]ade}-2X_{[de]abc}
\end{equation}
defines an automorphism of  
$$\Wedge^1\otimes\begin{picture}(12,12)(0,2)
\put(0,0){\line(1,0){12}}
\put(0,6){\line(1,0){12}}
\put(0,12){\line(1,0){12}}
\put(0,0){\line(0,1){12}}
\put(6,0){\line(0,1){12}}
\put(12,0){\line(0,1){12}}
\end{picture}\cong
\begin{picture}(18,12)(0,2)
\put(0,0){\line(1,0){12}}
\put(0,6){\line(1,0){18}}
\put(0,12){\line(1,0){18}}
\put(0,0){\line(0,1){12}}
\put(6,0){\line(0,1){12}}
\put(12,0){\line(0,1){12}}
\put(18,6){\line(0,1){6}}
\end{picture}\oplus\begin{picture}(12,18)(0,5)
\put(0,0){\line(1,0){6}}
\put(0,6){\line(1,0){12}}
\put(0,12){\line(1,0){12}}
\put(0,18){\line(1,0){12}}
\put(0,0){\line(0,1){18}}
\put(6,0){\line(0,1){18}}
\put(12,6){\line(0,1){12}}
\end{picture}\,.$$
\end{lemma}
\begin{proof} Certainly, this formula defines an endomorphism of 
$\Wedge^1\otimes\begin{picture}(12,12)(0,2)
\put(0,0){\line(1,0){12}}
\put(0,6){\line(1,0){12}}
\put(0,12){\line(1,0){12}}
\put(0,0){\line(0,1){12}}
\put(6,0){\line(0,1){12}}
\put(12,0){\line(0,1){12}}
\end{picture}$\,.  
Hence, a right inverse for (\ref{fun_formula}) will also be a left inverse and 
it is straightforward to check that if we set 
$$Y_{abcde}\equiv
X_{abcde}-2X_{[bc]ade}-2X_{[de]abc},$$
then  
$$\textstyle X_{abcde}
=-\frac13\big(Y_{abcde}+2Y_{[bc]ade}+2Y_{[de]abc}\big),$$
the verification of which will be left to the reader.
\end{proof}
\noindent We shall now use the automorphism (\ref{fun_formula}) to modify the
induced connection (\ref{symmetric_power}) on 
$\bigodot^2\!\big(\begin{picture}(6,6)(0,0)
\put(0,0){\line(1,0){6}}
\put(0,6){\line(1,0){6}}
\put(0,0){\line(0,1){6}}
\put(6,0){\line(0,1){6}}
\end{picture}\oplus\begin{picture}(6,12)(0,2)
\put(0,0){\line(1,0){6}}
\put(0,6){\line(1,0){6}}
\put(0,12){\line(1,0){6}}
\put(0,0){\line(0,1){12}}
\put(6,0){\line(0,1){12}}
\end{picture}\,\big)$ as follows:
\begin{equation}\label{modification}
\left[\!\begin{array}{c}\sigma_{bc}\\ \mu_{bcd}\\
\rho_{bcde}\end{array}\!\right]
\stackrel{\mbox{$\widetilde\nabla_a\;$}}{\longmapsto}
\left[\!\begin{array}{c}\nabla_a\sigma_{bc}+2\mu_{(bc)a}\\ 
\nabla_a\mu_{bcd}-R_{cd}{}^e{}_a\sigma_{be}-\rho_{abcd}\\
\nabla_a\rho_{bcde}-R_{bc}{}^f{}_a\mu_{fde}-R_{de}{}^f{}_a\mu_{fbc}
\end{array}\!\right]+\left[\!\begin{array}{c}0\\ 0\\
(R\pentagon(\mu,\sigma))_{abcde}\end{array}\!\right]\end{equation}
where 
\begin{equation}\label{this_is_pentagon}
\begin{array}{l}(R\pentagon(\mu,\sigma))_{abcde}\equiv
\frac13\big(X_{abcde}-2X_{[bc]ade}-2X_{[de]abc}\big),\\[3pt]
\enskip\mbox{for}\enskip
X_{abcde}\equiv 2 R_{bc}{}^f{}_{[d}\mu_{|a|e]f}
+2R_{de}{}^f{}_{[b}\mu_{|a|c]f}
-2(\nabla_{[b}R_{|de|}{}^f{}_{c]})\sigma_{fa}  
\in\Gamma (\Wedge^1 \otimes \begin{picture}(12,12)(0,2)
\put(0,0){\line(1,0){12}}
\put(0,6){\line(1,0){12}}
\put(0,12){\line(1,0){12}}
\put(0,0){\line(0,1){12}}
\put(6,0){\line(0,1){12}}
\put(12,0){\line(0,1){12}}
\end{picture}).\enskip\end{array}
\end{equation}
\begin{thm}\label{the_same_parallel_sections}
The connections 
\eqref{symmetric_power} and \eqref{modification} 
on $\bigodot^2\!\big(\begin{picture}(6,6)(0,0)
\put(0,0){\line(1,0){6}}
\put(0,6){\line(1,0){6}}
\put(0,0){\line(0,1){6}}
\put(6,0){\line(0,1){6}}
\end{picture}\oplus\begin{picture}(6,12)(0,2)
\put(0,0){\line(1,0){6}}
\put(0,6){\line(1,0){6}}
\put(0,12){\line(1,0){6}}
\put(0,0){\line(0,1){12}}
\put(6,0){\line(0,1){12}}
\end{picture}\,\big)$ have the same parallel \mbox{sections}.
\end{thm}
\begin{proof} For a section of
$\bigodot^2\!\big(\begin{picture}(6,6)(0,0)
\put(0,0){\line(1,0){6}}
\put(0,6){\line(1,0){6}}
\put(0,0){\line(0,1){6}}
\put(6,0){\line(0,1){6}}
\end{picture}\oplus\begin{picture}(6,12)(0,2)
\put(0,0){\line(1,0){6}}
\put(0,6){\line(1,0){6}}
\put(0,12){\line(1,0){6}}
\put(0,0){\line(0,1){12}}
\put(6,0){\line(0,1){12}}
\end{picture}\,\big)$
annihilated by (\ref{symmetric_power}), it follows from 
Lemma~\ref{symmetric_curvature} that 
$$R_{ab}{}^f{}_{[d}\mu_{|c|e]f}
+R_{de}{}^f{}_{[a}\mu_{|c|b]f}
-(\nabla_{[a}R_{|de|}{}^f{}_{b]})\sigma_{fc}.
$$
But the expression on left hand side of this equation is $X_{cabde}$ with
$X_{abcde}$ as in~(\ref{this_is_pentagon}).  It follows that
$(R\pentagon(\mu,\sigma))_{abcde}=0$ so this section is also annihilated by
$\widetilde\nabla_a$.  

For the converse, we should calculate the curvature of
$\widetilde\nabla_a$ as follows.  Writing (\ref{modification}) as
\begin{equation}\label{slicker_version}
\widetilde\nabla_b\left[\!\begin{array}{c}\sigma_{cd}\\ \mu_{cde}\\
\rho_{cdef}\end{array}\!\right]
=\nabla_b\left[\!\begin{array}{c}\sigma_{cd}\\ \mu_{cde}\\
\rho_{cdef}\end{array}\!\right]+\left[\!\begin{array}{c}0\\ 0\\
(R\pentagon(\mu,\sigma))_{bcdef}\end{array}\!\right]\end{equation}
we see that
$$\widetilde\nabla_a\widetilde\nabla_b
\left[\!\begin{array}{c}\sigma_{cd}\\ \mu_{cde}\\
\rho_{cdef}\end{array}\!\right]
=\widetilde\nabla_a\nabla_b\left[\!\begin{array}{c}\sigma_{cd}\\ \mu_{cde}\\
\rho_{cdef}\end{array}\!\right]
+\widetilde\nabla_a\left[\!\begin{array}{c}0\\ 0\\
(R\pentagon(\mu,\sigma))_{bcdef}\end{array}\!\right]$$
and hence that
$$\widetilde\nabla_a\widetilde\nabla_b
\left[\!\begin{array}{c}\sigma_{cd}\\ \mu_{cde}\\
\rho_{cdef}\end{array}\!\right]
=\nabla_a\nabla_b\left[\!\begin{array}{c}\sigma_{cd}\\ \mu_{cde}\\
\rho_{cdef}\end{array}\!\right]+\left[\!\begin{array}{c}0\\ 0\\ 
\ast\end{array}\!\right]
+\left[\!\begin{array}{c}0\\ {}-(R\pentagon(\mu,\sigma))_{abcde}\\
\ast\end{array}\!\right],$$
where $\ast$ denotes terms that we shall not need. 
Therefore, regarding the curvature of~$\widetilde\nabla_a$, we find that
$$\widetilde\nabla_{[a}\widetilde\nabla_{b]}
\left[\!\begin{array}{c}\sigma_{cd}\\ \mu_{cde}\\
\rho_{cdef}\end{array}\!\right]
=\nabla_{[a}\nabla_{b]}
\left[\!\begin{array}{c}\sigma_{cd}\\ \mu_{cde}\\
\rho_{cdef}\end{array}\!\right]
-\left[\!\begin{array}{c}0\\ 
(R\pentagon(\mu,\sigma))_{[ab]cde}\\
\ast\end{array}\!\right],$$
and thus we need to compute $(R\pentagon(\mu,\sigma))_{[ab]cde}$. 
Well, from~(\ref{this_is_pentagon}), we find that
$$\textstyle(R\pentagon(\mu,\sigma))_{[ab]cde}
=-\frac13X_{cabde}
-\frac13X_{de[ab]c}
+\frac13X_{ed[ab]c}
=-\frac13X_{cabde}
+\frac16X_{dabec}
-\frac16X_{eabdc}$$
and, from~(\ref{induced_curvature}), conclude that   
$$\widetilde\nabla_{[a}\widetilde\nabla_{b]}
\left[\!\begin{array}{c}\sigma_{cd}\\ \mu_{cde}\\
\rho_{cdef}\end{array}\!\right]
=\left[\!\begin{array}{c}0\\ 
\frac12X_{cabde}\\
\ast\end{array}\!\right]
+\left[\!\begin{array}{c}0\\ 
\frac13X_{cabde}
-\frac16X_{dabec}
+\frac16X_{eabdc}\\
\ast\end{array}\!\right],$$
whence
$$\widetilde\nabla_{[a}\widetilde\nabla_{b]}
\left[\!\begin{array}{c}\sigma_{cd}\\ \mu_{cde}\\
\rho_{cdef}\end{array}\!\right]
=\left[\!\begin{array}{c}0\\ 
\frac56X_{cdeab}
-\frac13X_{[de]cab}\\
\ast\end{array}\!\right],$$
where $X_{abcde}$ is defined by~(\ref{this_is_pentagon}).  According to
Lemma~\ref{very_simple_algebra} below, it follows that for any section of 
$\bigodot^2\!\big(\begin{picture}(6,6)(0,0)
\put(0,0){\line(1,0){6}}
\put(0,6){\line(1,0){6}}
\put(0,0){\line(0,1){6}}
\put(6,0){\line(0,1){6}}
\end{picture}\oplus\begin{picture}(6,12)(0,2)
\put(0,0){\line(1,0){6}}
\put(0,6){\line(1,0){6}}
\put(0,12){\line(1,0){6}}
\put(0,0){\line(0,1){12}}
\put(6,0){\line(0,1){12}}
\end{picture}\,\big)$
annihilated by $\widetilde\nabla_a$, we have $X_{cdeab}=0$. 
Hence~$(R\pentagon(\mu,\sigma))_{abcde}=0$ and from (\ref{modification}) it 
follows that this section is also annihilated by $\nabla_a$, as required.
\end{proof}

\begin{lemma}\label{very_simple_algebra} If the tensor $X_{cde}$ is skew in 
$de$ and\/ $5X_{cde}=2X_{[de]c}$, then $X_{cde}=0$.
\end{lemma}
\begin{proof} Firstly, observe that $X_{[cde]}=0$, which we may expand as
$$0=X_{cde}+X_{dec}-X_{edc}=X_{cde}+2X_{[de]c}$$
and immediately conclude that $X_{cde}=0$.	
\end{proof}

The true reason for our introducing the modified connection
(\ref{modification}) is not yet apparent.
Theorem~\ref{the_same_parallel_sections} says that this modification doesn't
change the parallel sections but this is true for many other modifications.
Some hint regarding our particular choice of modification can be spotted in the
following key piece of algebra, in which we may see
$(R\pentagon(\mu,\sigma))_{abcde}$ implicitly lurking and precisely tied to the
somewhat mysterious homomorphism
$$R\diamond\underbar{\enskip}:\begin{picture}(12,12)(0,2)
\put(0,0){\line(1,0){6}}
\put(0,6){\line(1,0){12}}
\put(0,12){\line(1,0){12}}
\put(0,0){\line(0,1){12}}
\put(6,0){\line(0,1){12}}
\put(12,6){\line(0,1){6}}
\end{picture}\longrightarrow
\Wedge^1\otimes\begin{picture}(12,12)(0,2)
\put(0,0){\line(1,0){12}}
\put(0,6){\line(1,0){12}}
\put(0,12){\line(1,0){12}}
\put(0,0){\line(0,1){12}}
\put(6,0){\line(0,1){12}}
\put(12,0){\line(0,1){12}}
\end{picture}$$
arising in our prolongation 
connection~(\ref{our_tricky_prolongation_connection}). 

\begin{lemma}\label{key_piece}
Suppose that $\mu_{bcd}\in\Gamma\big(\,
\begin{picture}(6,6)(0,0)
\put(0,0){\line(1,0){6}}
\put(0,6){\line(1,0){6}}
\put(0,0){\line(0,1){6}}
\put(6,0){\line(0,1){6}}
\end{picture}\otimes\begin{picture}(6,12)(0,2)
\put(0,0){\line(1,0){6}}
\put(0,6){\line(1,0){6}}
\put(0,12){\line(1,0){6}}
\put(0,0){\line(0,1){12}}
\put(6,0){\line(0,1){12}}
\end{picture}\,\big)$ and let\/ 
$\widetilde\mu_{bcd}\equiv\mu_{bcd}-\mu_{[bcd]}\in\Gamma\big(\,
\begin{picture}(12,12)(0,2)
\put(0,0){\line(1,0){6}}
\put(0,6){\line(1,0){12}}
\put(0,12){\line(1,0){12}}
\put(0,0){\line(0,1){12}}
\put(6,0){\line(0,1){12}}
\put(12,6){\line(0,1){6}}
\end{picture}\,\big)$. 
For $X_{abcde}$ and $(R\pentagon(\mu,\sigma))_{abcde}$ as
in~\eqref{this_is_pentagon}, it holds
$$\begin{array}{rcl}
(R\diamond\widetilde\mu)_{abcde}
-  (\nabla\!R \ast \sigma)_{abcde} 
&\!\!=\!\!&R_{bc}{}^f{}_a\mu_{fde}+R_{de}{}^f{}_a\mu_{fbc}
-2R_{[bc}{}^f{}_{|a}\mu_{f|de]}
-(R\pentagon(\mu,\sigma))_{abcde}.
\end{array}$$
\end{lemma}
\begin{proof} Since this equation is linear in $\mu_{abc}$ and $\sigma_{ab}$,
it suffice to verify it in the cases $\sigma_{ab}=0$ and $\mu_{abc}=0$.
In the case $\mu_{abc}=0$, the equality is straightforward to check by writing
out $(\nabla\!R \ast \sigma)_{abcde}$ from~(\ref{asterisk}) and comparing it to
the relevant terms in $(R\pentagon(\mu,\sigma))_{abcde}$.

In the case $\sigma_{ab}=0$ let us observe that
$$\textstyle
\widetilde\mu_{bcd}=\frac23\mu_{bcd}+\frac13\mu_{dcb}-\frac13\mu_{cdb}$$
and that $\mu_{[bcd]}$ only contributes to the first line of 
$(R\diamond\widetilde\mu)_{abcde}$, which therefore becomes
$$\begin{array}{l}
\frac23\big(R_{bc}{}^f{}_a\mu_{fde}+R_{de}{}^f{}_a\mu_{fbc}
-\frac12R_{be}{}^f{}_a\mu_{fcd}
-\frac12R_{cd}{}^f{}_a\mu_{fbe}
+\frac12R_{bd}{}^f{}_a\mu_{fce}
+\frac12R_{ce}{}^f{}_a\mu_{fbd}\big)\\[4pt]
\enskip{}+\frac13\big(R_{bc}{}^f{}_a\mu_{edf}+R_{de}{}^f{}_a\mu_{cbf}
-\frac12R_{be}{}^f{}_a\mu_{dcf}
-\frac12R_{cd}{}^f{}_a\mu_{ebf}
+\frac12R_{bd}{}^f{}_a\mu_{ecf}
+\frac12R_{ce}{}^f{}_a\mu_{dbf}\big)\\[4pt]
\quad{}-\frac13\big(R_{bc}{}^f{}_a\mu_{def}+R_{de}{}^f{}_a\mu_{bcf}
-\frac12R_{be}{}^f{}_a\mu_{cdf}
-\frac12R_{cd}{}^f{}_a\mu_{bef}
+\frac12R_{bd}{}^f{}_a\mu_{cef}
+\frac12R_{ce}{}^f{}_a\mu_{bdf}\big)\\[4pt]
\quad\enskip{}+\frac23
\big(R_{a(b}{}^f{}_{d)}\mu_{(ce)f}-R_{a(c}{}^f{}_{d)}\mu_{(be)f}
-R_{a(b}{}^f{}_{e)}\mu_{(cd)f}+R_{a(c}{}^f{}_{e)}\mu_{(bd)f}\big)\\[4pt]
\qquad{}-\frac23\big(R_{bc}{}^f{}_d\mu_{(ea)f}-R_{bc}{}^f{}_e\mu_{(da)f}
+R_{de}{}^f{}_b\mu_{(ca)f}-R_{de}{}^f{}_c\mu_{(ba)f}\big),\end{array}$$
in which we recognise the first line as
$R_{bc}{}^f{}_a\mu_{fde}+R_{de}{}^f{}_a\mu_{fbc}
-2R_{[bc}{}^f{}_{|a}\mu_{f|de]}$. Hence, we are required to check that the 
remaining four lines may be collected as
$$\textstyle{}
-\frac13\big(X_{abcde}+X_{bacde}-X_{cabde}+X_{daebc}-X_{eadbc}\big).$$
Notice that each of the five terms in this expression has a different index in
the first position and so it is straightforward to find them in the four lines 
in question. We may find contributions to $X_{abcde}$, for example, only on 
the last line, where we encounter 
$$\textstyle{}
-\frac13\big(R_{bc}{}^f{}_d\mu_{aef}-R_{bc}{}^f{}_e\mu_{adf}
+R_{de}{}^f{}_b\mu_{acf}-R_{de}{}^f{}_c\mu_{abf}\big)=-\frac13X_{abcde},$$
as required. Next, in seeking~$X_{bacde}$, we are obliged to consider
these underlined terms
$$\begin{array}{l}
\frac13\big(R_{bc}{}^f{}_a\mu_{edf}+R_{de}{}^f{}_a\mu_{cbf}
-\frac12R_{be}{}^f{}_a\mu_{dcf}
-\frac12R_{cd}{}^f{}_a\mu_{ebf}
+\frac12R_{bd}{}^f{}_a\mu_{ecf}
+\frac12R_{ce}{}^f{}_a\mu_{dbf}\big)\\[4pt]
\enskip{}-\frac13\big(R_{bc}{}^f{}_a\mu_{def}
+\underline{R_{de}{}^f{}_a\mu_{bcf}}
-\frac12R_{be}{}^f{}_a\mu_{cdf}
-\frac12\underline{R_{cd}{}^f{}_a\mu_{bef}}
+\frac12R_{bd}{}^f{}_a\mu_{cef}
+\frac12\underline{R_{ce}{}^f{}_a\mu_{bdf}}\big)\\[4pt]
\quad{}+\frac23
\big(R_{a(b}{}^f{}_{d)}\mu_{(ce)f}-\underline{R_{a(c}{}^f{}_{d)}\mu_{(be)f}}
-R_{a(b}{}^f{}_{e)}\mu_{(cd)f}
+\underline{R_{a(c}{}^f{}_{e)}\mu_{(bd)f}}\big)\\[4pt]
\quad\enskip{}-\frac23\big(R_{bc}{}^f{}_d\mu_{(ea)f}-R_{bc}{}^f{}_e\mu_{(da)f}
+R_{de}{}^f{}_b\mu_{(ca)f}-\underline{R_{de}{}^f{}_c\mu_{(ba)f}}\big),
\end{array}$$
leading to
$$\textstyle{}
-\frac13R_{de}{}^f{}_a\mu_{bcf}+\frac16R_{cd}{}^f{}_a\mu_{bef}
-\frac16R_{ce}{}^f{}_a\mu_{bdf}-\frac13R_{a(c}{}^f{}_{d)}\mu_{bef}
+\frac13R_{a(c}{}^f{}_{e)}\mu_{bdf}+\frac13R_{de}{}^f{}_c\mu_{baf},$$
which may be rearranged using the Bianchi symmetry 
$\frac16R_{cd}{}^f{}_a-\frac13R_{a(c}{}^f{}_{d)}=-\frac13R_{ac}{}^f{}_d$ as
$$\textstyle{}
-\frac13R_{de}{}^f{}_a\mu_{bcf}-\frac13R_{ac}{}^f{}_d\mu_{bef}
+\frac13R_{ac}{}^f{}_e\mu_{bdf}+\frac13R_{de}{}^f{}_c\mu_{baf}
=-\frac13X_{bacde},$$
as required. The remaining three $X$-terms are left to the reader. 
\end{proof}

\begin{thm}\label{Phi_is_parallel}
Let us consider the homomorphism of vector bundles
\begin{equation}\label{this_is_Phi}
\textstyle\bigodot^2\!\left(\!\begin{array}{c}
\begin{picture}(6,6)(0,0)
\put(0,0){\line(1,0){6}}
\put(0,6){\line(1,0){6}}
\put(0,0){\line(0,1){6}}
\put(6,0){\line(0,1){6}}
\end{picture}\\[-4pt] \oplus\\[-2pt]
\begin{picture}(6,12)(0,2)
\put(0,0){\line(1,0){6}}
\put(0,6){\line(1,0){6}}
\put(0,12){\line(1,0){6}}
\put(0,0){\line(0,1){12}}
\put(6,0){\line(0,1){12}}
\end{picture}\end{array}\!\right)=\begin{array}{c}
\begin{picture}(12,6)(0,0)
\put(0,0){\line(1,0){12}}
\put(0,6){\line(1,0){12}}
\put(0,0){\line(0,1){6}}
\put(6,0){\line(0,1){6}}
\put(12,0){\line(0,1){6}}
\end{picture}\\[-4pt]
\oplus\\[-2pt]
\begin{picture}(6,6)(0,0)
\put(0,0){\line(1,0){6}}
\put(0,6){\line(1,0){6}}
\put(0,0){\line(0,1){6}}
\put(6,0){\line(0,1){6}}
\end{picture}\otimes
\begin{picture}(6,12)(0,2)
\put(0,0){\line(1,0){6}}
\put(0,6){\line(1,0){6}}
\put(0,12){\line(1,0){6}}
\put(0,0){\line(0,1){12}}
\put(6,0){\line(0,1){12}}
\end{picture}\\[-2pt]
\oplus\\[-2pt]
\begin{picture}(6,12)(0,2)
\put(0,0){\line(1,0){6}}
\put(0,6){\line(1,0){6}}
\put(0,12){\line(1,0){6}}
\put(0,0){\line(0,1){12}}
\put(6,0){\line(0,1){12}}
\end{picture}\bigodot
\begin{picture}(6,12)(0,2)
\put(0,0){\line(1,0){6}}
\put(0,6){\line(1,0){6}}
\put(0,12){\line(1,0){6}}
\put(0,0){\line(0,1){12}}
\put(6,0){\line(0,1){12}}
\end{picture}
\end{array}
\ni\left[\!\begin{array}{c}\sigma_{bc}\\ \mu_{bcd}\\
\rho_{bcde}\end{array}\!\right]\stackrel{\mbox{$\Phi$}}{\longmapsto}
\left[\!\begin{array}{c}\sigma_{bc}\\ \mu_{bcd}-\mu_{[bcd]}\\
\rho_{bcde}-\rho_{[abcd]}\end{array}\!\right]\in
\begin{array}{c}
\begin{picture}(12,6)(0,0)
\put(0,0){\line(1,0){12}}
\put(0,6){\line(1,0){12}}
\put(0,0){\line(0,1){6}}
\put(6,0){\line(0,1){6}}
\put(12,0){\line(0,1){6}}
\end{picture}\\[-4pt]
\oplus\\[-2pt]
\begin{picture}(12,12)(0,2)
\put(0,0){\line(1,0){6}}
\put(0,6){\line(1,0){12}}
\put(0,12){\line(1,0){12}}
\put(0,0){\line(0,1){12}}
\put(6,0){\line(0,1){12}}
\put(12,6){\line(0,1){6}}
\end{picture}\\[-2pt]
\oplus\\[-2pt]
\begin{picture}(12,12)(0,2)
\put(0,0){\line(1,0){12}}
\put(0,6){\line(1,0){12}}
\put(0,12){\line(1,0){12}}
\put(0,0){\line(0,1){12}}
\put(6,0){\line(0,1){12}}
\put(12,0){\line(0,1){12}}
\end{picture}
\end{array},\end{equation}
where the bundles on the right are realised as in
Theorem~\ref{prolonged_2-tensors}. 
If we equip the bundle on the left with our modified connection
$\widetilde\nabla_a$ from \eqref{modification} and the bundle on the right with
our prolongation connection $\nabla_a$ from
\eqref{our_tricky_prolongation_connection}, then
$\nabla_a\circ\Phi=\Phi\circ\widetilde\nabla_a$.
\end{thm}
\begin{proof}
We compute from (\ref{our_tricky_prolongation_connection}) that
$$\nabla_a\circ\Phi\left[\!\begin{array}{c}\sigma_{bc}\\ \mu_{bcd}\\
\rho_{bcde}\end{array}\!\right]=
\nabla_a\left[\!\begin{array}{c}\sigma_{bc}\\ \widetilde\mu_{bcd}\\
\widetilde\rho_{bcde}\end{array}\!\right]
=\left[\!\begin{array}{c}\nabla_a\sigma_{bc}+2\widetilde\mu_{(bc)a}\\ 
\nabla_a\widetilde\mu_{bcd}-(R\triangleleft\sigma)_{abcd}
-\widetilde\rho_{abcd}\\
\nabla_a\widetilde\rho_{bcde}-(R\diamond\widetilde\mu)_{abcde}
+(\nabla\!R \ast \sigma)_{abcde}
\end{array}\!\right]$$
where $\widetilde\mu_{bcd}\equiv\mu_{bcd}-\mu_{[bcd]}$ and 
$\widetilde\rho_{bcde}\equiv\rho_{bcde}-\rho_{[bcde]}$.  Notice that 
$\widetilde\mu_{(bc)a}=\mu_{(bc)a}$
and also that $\widetilde\rho_{abcd}=\rho_{abcd}-\rho_{a[bcd]}$. 
Therefore,
$$\nabla_a\circ\Phi\left[\!\begin{array}{c}\sigma_{bc}\\ \mu_{bcd}\\
\rho_{bcde}\end{array}\!\right]
=\left[\!\begin{array}{c}\nabla_a\sigma_{bc}+2\mu_{(bc)a}\\ 
\nabla_a\mu_{bcd}-(R\triangleleft\sigma)_{abcd}
-\rho_{abcd}\\
\nabla_a\rho_{bcde}-(R\diamond\widetilde\mu)_{abcde}
+(\nabla\!R \ast\sigma)_{abcde}
\end{array}\!\right]
-\left[\!\begin{array}{c}0\\ \nabla_a\mu_{[bcd]}-\rho_{a[bcd]}\\
\nabla_a\rho_{[bcde]}\end{array}\!\right].$$
However, applying  
Lemma~\ref{key_piece} to the third line
and recalling that
$$\textstyle (R\triangleleft\sigma)_{abcd}
\equiv\frac23R_{cd}{}^e{}_a\sigma_{be}
-\frac13R_{bc}{}^e{}_a\sigma_{de}
+\frac13R_{bd}{}^e{}_a\sigma_{ce}
=R_{cd}{}^e{}_a\sigma_{be}-R_{[cd}{}^e{}_{|a|}\sigma_{b]e},$$
it follows from (\ref{symmetric_power}) that
$$\nabla_a\circ\Phi\left[\!\begin{array}{c}\sigma_{bc}\\ \mu_{bcd}\\
\rho_{bcde}\end{array}\!\right]
=\nabla_a\left[\!\begin{array}{c}\sigma_{bc}\\ \mu_{bcd}\\
\rho_{bcde}\end{array}\!\right]
+\left[\!\begin{array}{c}0\\ 
R_{[cd}{}^e{}_{|a|}\sigma_{b]e}\\
2R_{[bc}{}^f{}_{|a}\mu_{f|de]}
+(R\pentagon(\mu,\sigma)_{abcde}\end{array}\!\right]
-\left[\!\begin{array}{c}0\\ \nabla_a\mu_{[bcd]}-\rho_{a[bcd]}\\
\nabla_a\rho_{[bcde]}\end{array}\!\right],$$
which, according to (\ref{symmetric_power}) and (\ref{slicker_version}), is  
$$\Phi\circ\nabla_a\left[\!\begin{array}{c}\sigma_{bc}\\ \mu_{bcd}\\
\rho_{bcde}\end{array}\!\right]
+\left[\!\begin{array}{c}0\\ 0\\
(R\pentagon(\mu,\sigma))_{abcde}\end{array}\!\right]
=\Phi\circ\widetilde\nabla_a\left[\!\begin{array}{c}\sigma_{bc}\\ \mu_{bcd}\\
\rho_{bcde}\end{array}\!\right]$$
because $(R\pentagon(\mu,\sigma))_{a[bcde]}=0$.\end{proof}

\begin{cor}\label{SES_of_connections} 
There is a short exact sequence of bundles with compatible connections
\begin{equation}\label{SES}
\begin{array}{ccccccccc}
0&\to&\begin{array}{c}\Wedge^3\\[-3pt] \oplus\\[-1pt] \Wedge^4\end{array}
&\longrightarrow&
\bigodot^2\!\left(\!
\begin{array}{c}\Wedge^1\\[-3pt] \oplus\\[-2pt] \Wedge^2\end{array}\!\right)
&\stackrel{\mbox{$\Phi$}}{\longrightarrow}&
\raisebox{8pt}{$\begin{array}{c}
\begin{picture}(12,6)(0,0)
\put(0,0){\line(1,0){12}}
\put(0,6){\line(1,0){12}}
\put(0,0){\line(0,1){6}}
\put(6,0){\line(0,1){6}}
\put(12,0){\line(0,1){6}}
\end{picture}\\[-4pt]
\oplus\\[-2pt]
\begin{picture}(12,12)(0,2)
\put(0,0){\line(1,0){6}}
\put(0,6){\line(1,0){12}}
\put(0,12){\line(1,0){12}}
\put(0,0){\line(0,1){12}}
\put(6,0){\line(0,1){12}}
\put(12,6){\line(0,1){6}}
\end{picture}\\[-2pt]
\oplus\\[-2pt]
\begin{picture}(12,12)(0,2)
\put(0,0){\line(1,0){12}}
\put(0,6){\line(1,0){12}}
\put(0,12){\line(1,0){12}}
\put(0,0){\line(0,1){12}}
\put(6,0){\line(0,1){12}}
\put(12,0){\line(0,1){12}}
\end{picture}
\end{array}$}&\to&0,\\
&&\makebox[0pt][r]{$\nabla_a$}\mbox{\large$\,\downarrow$}
&&\makebox[0pt][r]{$\widetilde\nabla_a$}\mbox{\large$\,\downarrow$}
&&\makebox[0pt][r]{$\nabla_a$}\mbox{\large$\,\downarrow$}
\end{array}\end{equation}
where the connection $\nabla_a$ on $\Wedge^3\oplus\Wedge^4$ is given by 
\begin{equation}\label{killing_yano_prolongation}
\begin{array}{c}\Wedge^3\\[-4pt] \oplus\\[-2pt] \Wedge^4\end{array}\ni
\left[\!\begin{array}{c}\mu_{bcd}\\ \rho_{bcde}\end{array}\!\right]
\stackrel{\nabla_a\,}{\longmapsto}
\left[\!\begin{array}{c}\nabla_a\mu_{bcd}-\rho_{abcd}\\ 
\nabla_a\rho_{bcde}+4R_{a[b}{}^f{}_c\mu_{de]f}\end{array}\!\right],
\end{equation}
whilst the connections on the other two bundles are given by 
\eqref{modification} and \eqref{our_tricky_prolongation_connection}, 
respectively.
\end{cor}
\begin{proof} It is clear that (\ref{SES}) is an exact sequence of bundles and
Theorem~\ref{Phi_is_parallel} states precisely the compatibility between the
connection $\widetilde\nabla_a$ and the prolongation
connection~(\ref{our_tricky_prolongation_connection}) for Killing $2$-tensors.
It remains to examine the connection induced by $\widetilde\nabla_a$ on the
kernel of~$\Phi$.  By definition, this is given by~(\ref{modification}) when
$\sigma_{bc}=0$, $\mu_{bcd}=\mu_{[bcd]}$, and $\rho_{bcde}=\rho_{[bcde]}$.  
Under these circumstances, Lemma~\ref{key_piece} 
implies that
$$(R\pentagon(\mu,\sigma))_{abcde}
=R_{bc}{}^f{}_a\mu_{fde}+R_{de}{}^f{}_a\mu_{fbc}
-2R_{[bc}{}^f{}_{|a}\mu_{f|de]}$$
and (\ref{modification}) reduces to
$$\left[\!\begin{array}{c}0\\
\nabla_a\mu_{bcd}-\rho_{abcd}\\
\nabla_a\rho_{bcde}-R_{bc}{}^f{}_a\mu_{fde}-R_{de}{}^f{}_a\mu_{fbc}
\end{array}\!\right]+\left[\!\begin{array}{c}0\\ 0\\
R_{bc}{}^f{}_a\mu_{fde}+R_{de}{}^f{}_a\mu_{fbc}
-2R_{[bc}{}^f{}_{|a}\mu_{f|de]}
\end{array}\!\right],$$
which yields (\ref{killing_yano_prolongation}) by the Bianchi
symmetry.\end{proof}

\noindent{\bf Remarks}\quad Several remarks are in order.
\begin{itemize}
\item Theorems~\ref{the_same_parallel_sections} and~\ref{Phi_is_parallel}
nicely combine as follows. Suppose that $\sigma_b$ and $\widetilde\sigma_b$ 
are Killing $1$-forms. Setting $\mu_{bc}\equiv\nabla_b\sigma_c$ and 
$\widetilde\mu_{bc}\equiv\nabla_b\widetilde\sigma_c$ gives, in accordance with 
Theorem~\ref{killing_prolongation}, parallel sections of 
$\Wedge^1\oplus\Wedge^2$ equipped with the
connection~\eqref{Killing_connection}. Hence
$$\left[\!\begin{array}{c}
\sigma_b\widetilde\sigma_c+\widetilde\sigma_b\sigma_c\\ 
\sigma_b\widetilde\mu_{cd}+\widetilde\sigma_b\mu_{cd}\\
\mu_{bc}\widetilde\mu_{de}+\widetilde\mu_{bc}\mu_{de}\end{array}\!\right]
\in\begin{array}{c}
\begin{picture}(12,6)(0,0)
\put(0,0){\line(1,0){12}}
\put(0,6){\line(1,0){12}}
\put(0,0){\line(0,1){6}}
\put(6,0){\line(0,1){6}}
\put(12,0){\line(0,1){6}}
\end{picture}\\[-4pt]
\oplus\\[-2pt]
\begin{picture}(6,6)(0,0)
\put(0,0){\line(1,0){6}}
\put(0,6){\line(1,0){6}}
\put(0,0){\line(0,1){6}}
\put(6,0){\line(0,1){6}}
\end{picture}\otimes
\begin{picture}(6,12)(0,2)
\put(0,0){\line(1,0){6}}
\put(0,6){\line(1,0){6}}
\put(0,12){\line(1,0){6}}
\put(0,0){\line(0,1){12}}
\put(6,0){\line(0,1){12}}
\end{picture}\\[-2pt]
\oplus\\[-2pt]
\begin{picture}(6,12)(0,2)
\put(0,0){\line(1,0){6}}
\put(0,6){\line(1,0){6}}
\put(0,12){\line(1,0){6}}
\put(0,0){\line(0,1){12}}
\put(6,0){\line(0,1){12}}
\end{picture}\bigodot
\begin{picture}(6,12)(0,2)
\put(0,0){\line(1,0){6}}
\put(0,6){\line(1,0){6}}
\put(0,12){\line(1,0){6}}
\put(0,0){\line(0,1){12}}
\put(6,0){\line(0,1){12}}
\end{picture}
\end{array}
=\textstyle\bigodot^2\!\left(\!\begin{array}{c}
\Wedge^1\\[-3pt]
\oplus\\[-1pt]
\Wedge^2\end{array}\!\right)$$
is parallel for the connection~(\ref{symmetric_power}). Now, according to 
Theorem~\ref{the_same_parallel_sections}, this same section is also parallel 
for the connection~$\widetilde\nabla_a$ defined by~(\ref{modification}). 
Finally, as a consequence of Theorem~\ref{Phi_is_parallel} we obtain a 
parallel section 
$$\left[\!\begin{array}{c}
\sigma_b\widetilde\sigma_c+\widetilde\sigma_b\sigma_c\\ 
\sigma_b\widetilde\mu_{cd}+\widetilde\sigma_b\mu_{cd}
-\sigma_{[b}\widetilde\mu_{cd]}-\widetilde\sigma_{[b}\mu_{cd]}\\
\mu_{bc}\widetilde\mu_{de}+\widetilde\mu_{bc}\mu_{de}
-2\mu_{[bc}\widetilde\mu_{de]}\end{array}\!\right]
\in\begin{array}{c}
\begin{picture}(12,6)(0,0)
\put(0,0){\line(1,0){12}}
\put(0,6){\line(1,0){12}}
\put(0,0){\line(0,1){6}}
\put(6,0){\line(0,1){6}}
\put(12,0){\line(0,1){6}}
\end{picture}\\[-4pt]
\oplus\\[-2pt]
\begin{picture}(12,12)(0,2)
\put(0,0){\line(1,0){6}}
\put(0,6){\line(1,0){12}}
\put(0,12){\line(1,0){12}}
\put(0,0){\line(0,1){12}}
\put(6,0){\line(0,1){12}}
\put(12,6){\line(0,1){6}}
\end{picture}\\[-2pt]
\oplus\\[-2pt]
\begin{picture}(12,12)(0,2)
\put(0,0){\line(1,0){12}}
\put(0,6){\line(1,0){12}}
\put(0,12){\line(1,0){12}}
\put(0,0){\line(0,1){12}}
\put(6,0){\line(0,1){12}}
\put(12,0){\line(0,1){12}}
\end{picture}
\end{array}$$
for the prolongation connection~(\ref{our_tricky_prolongation_connection}).
This implies, of course, that $\sigma_{(b}\widetilde\sigma_{c)}$ is a Killing
field in accordance with Theorem~\ref{prolonged_2-tensors}.  Although this much
is clear without prolongation:
$$\nabla_{(a}\sigma_{b)}=0=\nabla_{(a}\widetilde\sigma_{b)}\implies
\nabla_{(a}\sigma_b\widetilde\sigma_{c)}=0,$$
we shall find that this line of reasoning via prolongation connections yields
more complete information.
\item 
The connection (\ref{killing_yano_prolongation}) on $\Wedge^3\oplus\Wedge^4$
has a self-contained interpretation as follows.  It is clear from the first
line that its parallel sections satisfy
\begin{equation}\label{KY_3forms}\nabla_{(a}\mu_{b)de}=0\end{equation}
and, indeed, this is precisely the constraint on $\mu_{bcd}\in\Wedge^3$ imposed
by this first line.  These are the {\em Killing-Yano\/} three-forms and
(\ref{killing_yano_prolongation}) is a prolongation connection for the
overdetermined system~(\ref{KY_3forms}):
$$\nabla_a\mu_{bcd}=\rho_{abcd}\iff
\nabla_a\left[\!\begin{array}{c}\mu_{bcd}\\ \rho_{bcde}\end{array}\!\right]
\equiv\left[\!\begin{array}{c}\nabla_a\mu_{bcd}-\rho_{abcd}\\ 
\nabla_a\rho_{bcde}+4R_{a[b}{}^f{}_c\mu_{de]f}\end{array}\!\right]
=0.$$
More precisely, from~(\ref{curvature_on_one-forms}), the induced curvature
$\kappa:\Wedge^3\to\Wedge^2\otimes\Wedge^3$ implies that
$$\textstyle\nabla_{[a}\rho_{b]cde}=\nabla_{[a}\nabla_{b]}\mu_{cde}=
-\frac32R_{ab}{}^f{}_{[c}\mu_{de]f},$$
whilst the de~Rham complex implies that $\nabla_{[a}\rho_{bcde]}=0$ and, by
dint of Lemma~\ref{mindless_algebra} below, these constraints are equivalent to
$\nabla_a\rho_{bcde}+4R_{a[b}{}^f{}_c\mu_{de]f}=0$.
\begin{lemma}\label{mindless_algebra} 
Suppose that $\phi_{abcde}=\phi_{a[bcde]}$ and\/ $\phi_{[abcde]}=0$.  Then
$\phi_{abcde}$ can be recovered from $\theta_{abcde}\equiv\phi_{[ab]cde}$
according to the formula $\phi_{abcde}=\frac83\theta_{a[bcde]}$.
\end{lemma}
\begin{proof} Firstly, from the symmetries of $\phi_{abcde}$, let us note that 
$$\phi_{abcde}=\phi_{bacde}-\phi_{cabde}+\phi_{dabce}-\phi_{eabcd}$$ 
and now we compute
$$\begin{array}{rcl}\theta_{a[bcde]}
&\!\!\!=\!\!\!&\frac14
\big[\theta_{abcde}-\theta_{acbde}+\theta_{adbce}-\theta_{aebcd}\big]\\[5pt]
&\!\!\!=\!\!\!&\frac18
\big[4\phi_{abcde}-\phi_{bacde}+\phi_{cabde}
-\phi_{dabce}+\phi_{eabcd}\big]=\frac38\phi_{abcde},
\end{array}$$
as required.\end{proof}
\item If there are no Killing-Yano $3$-forms, it follows from
Theorem~\ref{the_same_parallel_sections} and
Corollary~\ref{SES_of_connections} that the parallel sections of 
$\bigodot^2\!\big(\begin{picture}(6,6)(0,0)
\put(0,0){\line(1,0){6}}
\put(0,6){\line(1,0){6}}
\put(0,0){\line(0,1){6}}
\put(6,0){\line(0,1){6}}
\end{picture}\oplus\begin{picture}(6,12)(0,2)
\put(0,0){\line(1,0){6}}
\put(0,6){\line(1,0){6}}
\put(0,12){\line(1,0){6}}
\put(0,0){\line(0,1){12}}
\put(6,0){\line(0,1){12}}
\end{picture}\,\big),$
with respect to either $\nabla_a$ as in~(\ref{symmetric_power}) or
$\widetilde\nabla_a$ as in~(\ref{modification}), inject into the parallel
sections of the connection~(\ref{our_tricky_prolongation_connection}).  On 
the other hand, Theorem~\ref{max_parallel_flat_on_symmetric_product} below 
identifies the parallel sections of 
$\bigodot^2\!\big(\begin{picture}(6,6)(0,0)
\put(0,0){\line(1,0){6}}
\put(0,6){\line(1,0){6}}
\put(0,0){\line(0,1){6}}
\put(6,0){\line(0,1){6}}
\end{picture}\oplus\begin{picture}(6,12)(0,2)
\put(0,0){\line(1,0){6}}
\put(0,6){\line(1,0){6}}
\put(0,12){\line(1,0){6}}
\put(0,0){\line(0,1){12}}
\put(6,0){\line(0,1){12}}
\end{picture}\,\big)$. Looking ahead to \S\ref{decomposables}, we find a 
criterion for the injectivity of (\ref{K1->K2}), the homomorphism creating 
Killing $2$-tensors from Killing $1$-forms. 

\item If we try to feed the Killing-Yano equation (\ref{KY_3forms}) directly
into Theorem~\ref{prolong}, then we find that we are already obliged to impose
$$\pi\big(R_{ab}{}^f{}_{[c}\mu_{de]f}\big)=0\quad\mbox{for the natural 
projection}\enskip
\pi:\begin{picture}(6,12)(0,2)
\put(0,0){\line(1,0){6}}
\put(0,6){\line(1,0){6}}
\put(0,12){\line(1,0){6}}
\put(0,0){\line(0,1){12}}
\put(6,0){\line(0,1){12}}
\end{picture}\otimes\begin{picture}(6,18)(0,5)
\put(0,0){\line(1,0){6}}
\put(0,6){\line(1,0){6}}
\put(0,12){\line(1,0){6}}
\put(0,18){\line(1,0){6}}
\put(0,0){\line(0,1){18}}
\put(6,0){\line(0,1){18}}
\end{picture}
\to\begin{picture}(12,18)(0,5)
\put(0,0){\line(1,0){6}}
\put(0,6){\line(1,0){12}}
\put(0,12){\line(1,0){12}}
\put(0,18){\line(1,0){12}}
\put(0,0){\line(0,1){18}}
\put(6,0){\line(0,1){18}}
\put(12,6){\line(0,1){12}}
\end{picture}$$
in order to find a lifted homomorphism~$\tilde\kappa$.  Thus, we encounter a
subbundle of the prolongation bundle~$\Wedge^3\oplus\Wedge^4$ but, otherwise,
end up with the connection~(\ref{killing_yano_prolongation}).
\end{itemize}

\section{Decomposable Killing two-tensors}\label{decomposables}
We are now in a position to investigate the space of {\em decomposable\/}
Killing two-tensors on an affine locally symmetric space by means of the
connections developed in the previous section.  More precisely, following
Matveev and Nikolayevsky~\cite{MN}, let us write
$$\textstyle 
K^k(M,\nabla)\equiv\{\sigma_{bc\cdots d}\in\Gamma(M,\bigodot^k\!\Wedge^1)\mid
\nabla_{(a}\sigma_{bc\cdots d)}=0\}$$
for the space of Killing tensors of rank $k$ on the smooth manifold $M$
equipped with a torsion-free affine connection~$\nabla$.  As parallel sections
of prolongation bundles (e.g.,~\cite{BCEG} for general~$d$), these are
finite-dimensional vector spaces.  If $\sigma_b$ and $\widetilde\sigma_b$ are
Killing $1$-forms, it is clear that their symmetric product
$\sigma_{(b}\widetilde\sigma_{c)}$ is a Killing $2$-tensor.  We obtain,
therefore, a canonical homomorphism
\begin{equation}\label{K1->K2}\textstyle 
\bigodot^2\!K^1(M,\nabla)\to K^2(M,\nabla)\end{equation}
and the {\em decomposable\/} Killing $2$-tensors constitute, by definition, its
range.  Now we may bring in the connections developed in~\S\ref{prolongation}.

Firstly, according to Theorem~\ref{killing_prolongation}, we have a canonical 
identification
$$K^1(M,\nabla)
=\{\Sigma\in\Gamma(M,\Wedge^1\oplus\Wedge^2\mid\nabla_a\Sigma=0\},$$
where $\nabla_a$ is the Killing connection (\ref{Killing_connection}).
Locally, or if $M$ is simply-connected, we may suppress $M$ from the notation
and write, more informally, that the Killing $1$-forms $K^1(\nabla)$ may be
identified as the `parallel sections of $\Wedge^1\oplus\Wedge^2$.'  In the
affine locally symmetric case, i.e.~when $\nabla_aR_{bc}{}^d{}_e=0$, we may
identify these parallel sections more explicitly as follows.  If
$\nabla_b\Sigma=0$, then of course $\nabla_{[a}\nabla_{b]}\Sigma=0$, in other
words that $\Sigma$ is in the kernel of the curvature.  But in the affine
locally symmetric case, we have a formula
\begin{equation}\label{killing_curvature} \nabla_{[a}\nabla_{b]}
\left[\!\begin{array}{c}\sigma_c\\ \mu_{cd}\end{array}\!\right]
=\left[\!\begin{array}{c}0\\ 
R_{ab}{}^e{}_{[c}\mu_{d]e}+R_{cd}{}^e{}_{[a}\mu_{b]e}
\end{array}\!\right]\end{equation}
obtained from (\ref{full_Killing_curvature}) and
(\ref{in_the_locally_symmetric_case}), which immediately identifies this kernel
as $\Wedge^1\oplus K$, where
\begin{equation}\label{this_is_K}
K\equiv\{\mu_{cd}\in\Wedge^2\mid
R_{ab}{}^e{}_{[c}\mu_{d]e}+R_{cd}{}^e{}_{[a}\mu_{b]e}=0\}.\end{equation}
\begin{thm}\label{LTS} 
In the affine locally symmetric case $\nabla_aR_{bc}{}^d{}_e=0$, we have
$$R_{ab}{}^c{}_d\in\Gamma(K\otimes\End(\Wedge^1)).$$
\end{thm}
\begin{proof} Expanding $(\nabla_a\nabla_b-\nabla_b\nabla_a)R_{cd}{}^p{}_q=0$
gives
\begin{equation}\label{gives}2R_{ab}{}^e{}_{[c}R_{d]e}{}^p{}_q
=R_{ab}{}^e{}_qR_{cd}{}^p{}_e-R_{cd}{}^e{}_qR_{ab}{}^p{}_e\end{equation}
and the RHS is skew under the interchange of indices $ab\leftrightarrow cd$. 
\end{proof}
\begin{cor}\label{parallel_subbundle} In the affine locally symmetric case,
the subbundle $\Wedge^1\oplus
K\subseteq\Wedge^1\oplus\Wedge^2$ is parallel, i.e.~if\/
$\Sigma\in\Gamma(\Wedge^1\oplus K)$, then
$\nabla_a\Sigma\in\Gamma(\Wedge^1\otimes(\Wedge^1\oplus K))$.
\end{cor}
\begin{proof} By inspection, from the formula~(\ref{Killing_connection}).
\end{proof} 
Thus, in the affine locally symmetric case, we have completely identified the
space $K^1(\nabla)$ of local Killing $1$-forms.  On the one hand, a parallel
section of $\Wedge^1\oplus\Wedge^2$ must be a parallel section of
$\Wedge^1\oplus K$ in accordance with (\ref{killing_curvature})
and~(\ref{this_is_K}).  On the other hand, Corollary~\ref{parallel_subbundle}
tells us that this subbundle is already parallel so there are no more
restrictions.  More precisely, at a chosen basepoint~$p\in M$, the value
$\sigma_b(p)$ of a Killing $1$-form $\sigma_b$ is arbitrary and the value
$(\nabla_a\sigma_b)(p)$ of its covariant derivative is an arbitrary $2$-form
in~$K_p$.  Parallel propagation within $\Wedge^1\oplus K$ then consistently
gives $\sigma_b$ and $\nabla_a\sigma_b$ in a neighbourhood of $p$.  In other
words, we have proved the following.

\begin{thm}\label{classical_locally_symmetric_conclusion} In case that
$\nabla_aR_{bc}{}^d{}_e=0$, the maximal parallel flat subbundle of
$\Wedge^1\oplus\Wedge^2$ is $\Wedge^1\oplus K$, where $K\subseteq\Wedge^2$ is
the subbundle defined by~\eqref{this_is_K}.
\end{thm}

The corresponding result for $\bigodot^2(\Wedge^1\oplus\Wedge^2)$ is as
follows.  
\begin{thm}\label{max_parallel_flat_on_symmetric_product} In case\/
$\nabla_aR_{bc}{}^d{}_e=0$, the maximal parallel flat subbundle of\/
$\bigodot^2(\Wedge^1\oplus\Wedge^2)$ equipped with either the induced
connection $\nabla_a$ from \eqref{symmetric_power}, or its modification
$\widetilde\nabla_a$ defined by~\eqref{modification}, is
$\bigodot^2(\Wedge^1\oplus K)$.
\end{thm}
\begin{proof} It follows from Theorem~\ref{the_same_parallel_sections} that
these two connections (\ref{symmetric_power}) and (\ref{modification}) have the
same maximal parallel flat subbundle.  Hence it suffices to consider the
induced connection~(\ref{symmetric_power}).  From the second line of its
curvature~(\ref{induced_curvature}), we read off that
$\mu_{cde}\in\Gamma(\Wedge^1\otimes K$) and then, from (\ref{symmetric_power})
and Theorem~\ref{LTS}, conclude that $\rho_{bcde}\in\Gamma(\Wedge^2\otimes K)$.
But $\rho_{bcde}$ is symmetric under the interchange $bc\leftrightarrow de$ so
we conclude that $\rho_{bcde}\in\Gamma(K\bigodot K)$.  Finally, such tensors 
$\rho_{bcde}$ are annihilated by the last line of (\ref{induced_curvature}) so 
there are no more constraints.
\end{proof}

Returning now to the homomorphism (\ref{K1->K2}), we see from
Theorems~\ref{classical_locally_symmetric_conclusion}
and~\ref{max_parallel_flat_on_symmetric_product} that, in the simply connected
and affine locally symmetric case, we may realise the vector space
$\bigodot^2\!K^1(M,\nabla)$ as the parallel sections of the subbundle
$$\begin{array}{c}
\begin{picture}(12,6)(0,0)
\put(0,0){\line(1,0){12}}
\put(0,6){\line(1,0){12}}
\put(0,0){\line(0,1){6}}
\put(6,0){\line(0,1){6}}
\put(12,0){\line(0,1){6}}
\end{picture}\\[-2pt]
\oplus\\
\begin{picture}(6,6)(0,0)
\put(0,0){\line(1,0){6}}
\put(0,6){\line(1,0){6}}
\put(0,0){\line(0,1){6}}
\put(6,0){\line(0,1){6}}
\end{picture}\otimes\! K\\[-2pt]
\oplus\\
K\bigodot K
\end{array}\subseteq
\begin{array}{c}
\begin{picture}(12,6)(0,0)
\put(0,0){\line(1,0){12}}
\put(0,6){\line(1,0){12}}
\put(0,0){\line(0,1){6}}
\put(6,0){\line(0,1){6}}
\put(12,0){\line(0,1){6}}
\end{picture}\\[-4pt]
\oplus\\[-2pt]
\begin{picture}(6,6)(0,0)
\put(0,0){\line(1,0){6}}
\put(0,6){\line(1,0){6}}
\put(0,0){\line(0,1){6}}
\put(6,0){\line(0,1){6}}
\end{picture}\otimes
\begin{picture}(6,12)(0,2)
\put(0,0){\line(1,0){6}}
\put(0,6){\line(1,0){6}}
\put(0,12){\line(1,0){6}}
\put(0,0){\line(0,1){12}}
\put(6,0){\line(0,1){12}}
\end{picture}\\[-2pt]
\oplus\\[-2pt]
\begin{picture}(6,12)(0,2)
\put(0,0){\line(1,0){6}}
\put(0,6){\line(1,0){6}}
\put(0,12){\line(1,0){6}}
\put(0,0){\line(0,1){12}}
\put(6,0){\line(0,1){12}}
\end{picture}\bigodot
\begin{picture}(6,12)(0,2)
\put(0,0){\line(1,0){6}}
\put(0,6){\line(1,0){6}}
\put(0,12){\line(1,0){6}}
\put(0,0){\line(0,1){12}}
\put(6,0){\line(0,1){12}}
\end{picture}
\end{array}$$
with respect to the connection~$\widetilde\nabla_a$, which is flat when
restricted to this subbundle.  Thus, in order to understand the 
homomorphism~(\ref{K1->K2}) we may employ Corollary~\ref{SES_of_connections} 
to identify the space of decomposable Killing $2$-tensors as the parallel 
sections of the bundle
$$\mbox{\large$\Phi$}\!\!\left(\!\!\!\begin{array}{c}
\begin{picture}(12,6)(0,0)
\put(0,0){\line(1,0){12}}
\put(0,6){\line(1,0){12}}
\put(0,0){\line(0,1){6}}
\put(6,0){\line(0,1){6}}
\put(12,0){\line(0,1){6}}
\end{picture}\\[-2pt]
\oplus\\
\begin{picture}(6,6)(0,0)
\put(0,0){\line(1,0){6}}
\put(0,6){\line(1,0){6}}
\put(0,0){\line(0,1){6}}
\put(6,0){\line(0,1){6}}
\end{picture}\otimes\! K\\[-2pt]
\oplus\\
K\bigodot K
\end{array}\!\!\!\!\right)\equiv\Delta\subseteq
\begin{array}{c}
\begin{picture}(12,6)(0,0)
\put(0,0){\line(1,0){12}}
\put(0,6){\line(1,0){12}}
\put(0,0){\line(0,1){6}}
\put(6,0){\line(0,1){6}}
\put(12,0){\line(0,1){6}}
\end{picture}\\[-4pt]
\oplus\\[-2pt]
\begin{picture}(12,12)(0,2)
\put(0,0){\line(1,0){6}}
\put(0,6){\line(1,0){12}}
\put(0,12){\line(1,0){12}}
\put(0,0){\line(0,1){12}}
\put(6,0){\line(0,1){12}}
\put(12,6){\line(0,1){6}}
\end{picture}\\[-2pt]
\oplus\\[-2pt]
\begin{picture}(12,12)(0,2)
\put(0,0){\line(1,0){12}}
\put(0,6){\line(1,0){12}}
\put(0,12){\line(1,0){12}}
\put(0,0){\line(0,1){12}}
\put(6,0){\line(0,1){12}}
\put(12,0){\line(0,1){12}}
\end{picture}
\end{array},$$
where $\Phi$ is the homomorphism (\ref{this_is_Phi}) from
Theorem~\ref{Phi_is_parallel}.  
\begin{thm}\label{small_subbundle} In case $\nabla_aR_{bc}{}^d{}_e=0$, the
subbundle
$\Delta\subseteq\begin{picture}(12,6)(0,0)
\put(0,0){\line(1,0){12}}
\put(0,6){\line(1,0){12}}
\put(0,0){\line(0,1){6}}
\put(6,0){\line(0,1){6}}
\put(12,0){\line(0,1){6}}
\end{picture}\oplus
\begin{picture}(12,12)(0,2)
\put(0,0){\line(1,0){6}}
\put(0,6){\line(1,0){12}}
\put(0,12){\line(1,0){12}}
\put(0,0){\line(0,1){12}}
\put(6,0){\line(0,1){12}}
\put(12,6){\line(0,1){6}}
\end{picture}\oplus\begin{picture}(12,12)(0,2)
\put(0,0){\line(1,0){12}}
\put(0,6){\line(1,0){12}}
\put(0,12){\line(1,0){12}}
\put(0,0){\line(0,1){12}}
\put(6,0){\line(0,1){12}}
\put(12,0){\line(0,1){12}}
\end{picture}$ of the prolongation bundle is parallel and flat with respect to 
the prolongation connection~\eqref{our_tricky_prolongation_connection}.  All 
Killing $2$-tensors are decomposable if and only if $\Delta$ is the maximal 
subbundle with these properties. 
\end{thm}
\begin{proof} According to
Theorem~\ref{max_parallel_flat_on_symmetric_product}, the subbundle
$\bigodot^2(\Wedge^1\oplus K)$ is generated by its parallel sections and
according to Theorem~\ref{Phi_is_parallel}, these sections have parallel images
under~$\Phi$.  It follows that $\Delta$ is parallel and flat.  But now, if
(\ref{K1->K2}) is not surjective and $\sigma_{bc}$ is a {\em hidden symmetry},
i.e.~a Killing $2$-tensor not in its range, then we may adjoin the
prolonged section 
$\big(\sigma_{bc},\mu_{bcd}\equiv\frac23\nabla_{[c}\sigma_{d]b},
\rho_{bcde}\equiv\nabla_b\mu_{cde}-(R\triangleleft\sigma)_{bcde}\big)$ 
to obtain a subbundle, strictly larger than~$\Delta$, that is still parallel
and flat.
\end{proof}
\noindent{\bf Remark}\quad In summary, here is the parallel flat subbundle of
the prolongation bundle from Theorem~\ref{prolonged_2-tensors} whose parallel
sections constitute the decomposable Killing $2$-tensors in the affine locally
symmetric case ($\nabla_aR_{bc}{}^d{}_e=0$).
\begin{equation}\label{decomposable_killing_2-tensors}
\begin{array}{ccc}
\begin{picture}(12,6)(0,0)
\put(0,0){\line(1,0){12}}
\put(0,6){\line(1,0){12}}
\put(0,0){\line(0,1){6}}
\put(6,0){\line(0,1){6}}
\put(12,0){\line(0,1){6}}
\end{picture}&\equiv&\big\{\sigma_{bc}=\sigma_{(bc)}\big\}\\ \oplus&&\oplus\\
\Phi(\,\begin{picture}(6,6)(0,0)
\put(0,0){\line(1,0){6}}
\put(0,6){\line(1,0){6}}
\put(0,0){\line(0,1){6}}
\put(6,0){\line(0,1){6}}
\end{picture}\otimes\! K)&\equiv&
\big\{\mu_{bcd}=\widehat\mu_{bcd}-\widehat\mu_{[bcd]},
\mbox{ for }
\widehat\mu_{bcd}\in
\Gamma(\,\begin{picture}(6,6)(0,0)
\put(0,0){\line(1,0){6}}
\put(0,6){\line(1,0){6}}
\put(0,0){\line(0,1){6}}
\put(6,0){\line(0,1){6}}
\end{picture}\otimes K)\big\}\\ \oplus&&\oplus\\
\Phi(K\bigodot K)
&\equiv&\big\{\rho_{bcde}=\widehat\rho_{bcde}-\widehat\rho_{[bcde]},
\mbox{ for }\widehat\rho_{bcde}\in\Gamma(K\bigodot K)\big\},
\end{array}\end{equation}
where $K\subseteq\Wedge^2$ is defined by (\ref{this_is_K}).  We may check
more directly that this subbundle is parallel and flat (rather than employing
Theorems~\ref{the_same_parallel_sections} and~\ref{Phi_is_parallel}).  In any
case, it is useful to have an abstract characterisation of the curvature of the
connection~(\ref{our_tricky_prolongation_connection}) as follows.  Recall that,
as a consequence of Theorem~\ref{prolong}, the curvature of
(\ref{our_tricky_prolongation_connection}) has the form~(\ref{schematically}),
which, in the affine locally symmetric case, is carried by
\begin{equation}\label{abstract_bowtie}
\begin{picture}(12,12)(0,2)
\put(0,0){\line(1,0){12}} 
\put(0,6){\line(1,0){12}}
\put(0,12){\line(1,0){12}}
\put(0,0){\line(0,1){12}}
\put(6,0){\line(0,1){12}}
\put(12,0){\line(0,1){12}}
\end{picture}\ni\rho\longmapsto R\bowtie\rho\in
\begin{picture}(18,12)(0,2)
\put(0,0){\line(1,0){18}}
\put(0,6){\line(1,0){18}}
\put(0,12){\line(1,0){18}}
\put(0,0){\line(0,1){12}}
\put(6,0){\line(0,1){12}}
\put(12,0){\line(0,1){12}}
\put(18,0){\line(0,1){12}}
\end{picture}\end{equation}
in accordance with~(\ref{bay_window}).  The symmetries of $\rho_{defg}$ ensure 
that any contraction with the contravariant index of $R_{ab}{}^g{}_c$ may as 
well be with its last index and now notice that 
$$R_{ab}{}^g{}_c\rho_{defg}\in
\begin{picture}(12,12)(0,2)
\put(0,0){\line(1,0){6}}
\put(0,6){\line(1,0){12}}
\put(0,12){\line(1,0){12}}
\put(0,0){\line(0,1){12}}
\put(6,0){\line(0,1){12}}
\put(12,6){\line(0,1){6}}
\end{picture}\otimes\begin{picture}(12,12)(0,2)
\put(0,0){\line(1,0){6}}
\put(0,6){\line(1,0){12}}
\put(0,12){\line(1,0){12}}
\put(0,0){\line(0,1){12}}
\put(6,0){\line(0,1){12}}
\put(12,6){\line(0,1){6}}
\end{picture}=\begin{picture}(18,12)(0,2)
\put(0,0){\line(1,0){18}}
\put(0,6){\line(1,0){18}}
\put(0,12){\line(1,0){18}}
\put(0,0){\line(0,1){12}}
\put(6,0){\line(0,1){12}}
\put(12,0){\line(0,1){12}}
\put(18,0){\line(0,1){12}}
\end{picture}\oplus\cdots,\enskip\mbox{where}\enskip
\begin{picture}(18,12)(0,2)
\put(0,0){\line(1,0){18}}
\put(0,6){\line(1,0){18}}
\put(0,12){\line(1,0){18}}
\put(0,0){\line(0,1){12}}
\put(6,0){\line(0,1){12}}
\put(12,0){\line(0,1){12}}
\put(18,0){\line(0,1){12}}
\end{picture}\enskip\mbox{occurs with multiplicity one}.$$
Up to scale, therefore, $R\bowtie\rho$ is the projection of 
$R_{ab}{}^g{}_c\rho_{defg}$ onto this 
$\begin{picture}(18,12)(0,2)
\put(0,0){\line(1,0){18}}
\put(0,6){\line(1,0){18}}
\put(0,12){\line(1,0){18}}
\put(0,0){\line(0,1){12}}
\put(6,0){\line(0,1){12}}
\put(12,0){\line(0,1){12}}
\put(18,0){\line(0,1){12}}
\end{picture}$ component. Similarly, the bundle $K$ has an 
abstract characterisation as the kernel of 
$$\begin{picture}(6,12)(0,2)
\put(0,0){\line(1,0){6}}
\put(0,6){\line(1,0){6}}
\put(0,12){\line(1,0){6}}
\put(0,0){\line(0,1){12}}
\put(6,0){\line(0,1){12}}
\end{picture}\ni\mu_{fg}\longmapsto
R_{ab}{}^g{}_c\mu_{fg}\in\begin{picture}(12,12)(0,2)
\put(0,0){\line(1,0){6}}
\put(0,6){\line(1,0){12}}
\put(0,12){\line(1,0){12}}
\put(0,0){\line(0,1){12}}
\put(6,0){\line(0,1){12}}
\put(12,6){\line(0,1){6}}
\end{picture}\otimes\begin{picture}(6,6)(0,0)
\put(0,0){\line(1,0){6}}
\put(0,6){\line(1,0){6}}
\put(0,0){\line(0,1){6}}
\put(6,0){\line(0,1){6}}
\end{picture}=\begin{picture}(12,12)(0,2)
\put(0,0){\line(1,0){12}}
\put(0,6){\line(1,0){12}}
\put(0,12){\line(1,0){12}}
\put(0,0){\line(0,1){12}}
\put(6,0){\line(0,1){12}}
\put(12,0){\line(0,1){12}}
\end{picture}\oplus\cdots
\xrightarrow{\,\mbox{projection}\,}
\begin{picture}(12,12)(0,2)
\put(0,0){\line(1,0){12}}
\put(0,6){\line(1,0){12}}
\put(0,12){\line(1,0){12}}
\put(0,0){\line(0,1){12}}
\put(6,0){\line(0,1){12}}
\put(12,0){\line(0,1){12}}
\end{picture}.$$
Now, since
$$\textstyle\begin{picture}(6,12)(0,2)
\put(0,0){\line(1,0){6}}
\put(0,6){\line(1,0){6}}
\put(0,12){\line(1,0){6}}
\put(0,0){\line(0,1){12}}
\put(6,0){\line(0,1){12}}
\end{picture}\bigodot
\begin{picture}(6,12)(0,2)
\put(0,0){\line(1,0){6}}
\put(0,6){\line(1,0){6}}
\put(0,12){\line(1,0){6}}
\put(0,0){\line(0,1){12}}
\put(6,0){\line(0,1){12}}
\end{picture}=\begin{picture}(12,12)(0,2)
\put(0,0){\line(1,0){12}}
\put(0,6){\line(1,0){12}}
\put(0,12){\line(1,0){12}}
\put(0,0){\line(0,1){12}}
\put(6,0){\line(0,1){12}}
\put(12,0){\line(0,1){12}}
\end{picture}\oplus\begin{picture}(6,18)(0,8)
\put(0,0){\line(1,0){6}}
\put(0,6){\line(1,0){6}}
\put(0,12){\line(1,0){6}}
\put(0,18){\line(1,0){6}}
\put(0,24){\line(1,0){6}}
\put(0,0){\line(0,1){24}}
\put(6,0){\line(0,1){24}}
\end{picture}\enskip\mbox{and}\enskip\begin{picture}(12,12)(0,2)
\put(0,0){\line(1,0){12}}
\put(0,6){\line(1,0){12}}
\put(0,12){\line(1,0){12}}
\put(0,0){\line(0,1){12}}
\put(6,0){\line(0,1){12}}
\put(12,0){\line(0,1){12}}
\end{picture}
\enskip\mbox{does not occur in the decomposition of}
\enskip
\begin{picture}(12,12)(0,2)
\put(0,0){\line(1,0){6}}
\put(0,6){\line(1,0){12}}
\put(0,12){\line(1,0){12}}
\put(0,0){\line(0,1){12}}
\put(6,0){\line(0,1){12}}
\put(12,6){\line(0,1){6}}
\end{picture}\otimes\begin{picture}(6,18)(0,5)
\put(0,0){\line(1,0){6}}
\put(0,6){\line(1,0){6}}
\put(0,12){\line(1,0){6}}
\put(0,18){\line(1,0){6}}
\put(0,0){\line(0,1){18}}
\put(6,0){\line(0,1){18}}
\end{picture}$$
we conclude that (\ref{abstract_bowtie}) lifts to
$\begin{picture}(6,12)(0,2)
\put(0,0){\line(1,0){6}}
\put(0,6){\line(1,0){6}}
\put(0,12){\line(1,0){6}}
\put(0,0){\line(0,1){12}}
\put(6,0){\line(0,1){12}}
\end{picture}\bigodot
\begin{picture}(6,12)(0,2)
\put(0,0){\line(1,0){6}}
\put(0,6){\line(1,0){6}}
\put(0,12){\line(1,0){6}}
\put(0,0){\line(0,1){12}}
\put(6,0){\line(0,1){12}}
\end{picture}$ and factors through
$$\textstyle\begin{picture}(6,12)(0,2)
\put(0,0){\line(1,0){6}}
\put(0,6){\line(1,0){6}}
\put(0,12){\line(1,0){6}}
\put(0,0){\line(0,1){12}}
\put(6,0){\line(0,1){12}}
\end{picture}\bigodot
\begin{picture}(6,12)(0,2)
\put(0,0){\line(1,0){6}}
\put(0,6){\line(1,0){6}}
\put(0,12){\line(1,0){6}}
\put(0,0){\line(0,1){12}}
\put(6,0){\line(0,1){12}}
\end{picture}\ni\rho_{abfg}\longmapsto R_{cd}{}^g{}_e\rho_{abfg}\in
\begin{picture}(6,12)(0,2)
\put(0,0){\line(1,0){6}}
\put(0,6){\line(1,0){6}}
\put(0,12){\line(1,0){6}}
\put(0,0){\line(0,1){12}}
\put(6,0){\line(0,1){12}}
\end{picture}\otimes\begin{picture}(12,12)(0,2)
\put(0,0){\line(1,0){12}}
\put(0,6){\line(1,0){12}}
\put(0,12){\line(1,0){12}}
\put(0,0){\line(0,1){12}}
\put(6,0){\line(0,1){12}}
\put(12,0){\line(0,1){12}}
\end{picture}=\begin{picture}(18,12)(0,2)
\put(0,0){\line(1,0){18}}
\put(0,6){\line(1,0){18}}
\put(0,12){\line(1,0){18}}
\put(0,0){\line(0,1){12}}
\put(6,0){\line(0,1){12}}
\put(12,0){\line(0,1){12}}
\put(18,0){\line(0,1){12}}
\end{picture}\oplus\cdots\xrightarrow{\,\mbox{projection}\,}
\begin{picture}(18,12)(0,2)
\put(0,0){\line(1,0){18}}
\put(0,6){\line(1,0){18}}
\put(0,12){\line(1,0){18}}
\put(0,0){\line(0,1){12}}
\put(6,0){\line(0,1){12}}
\put(12,0){\line(0,1){12}}
\put(18,0){\line(0,1){12}}
\end{picture}.$$
It follows that $\rho\mapsto R\bowtie\rho$ vanishes on 
$\rho_{bcde}=\widehat\rho_{bcde}-\widehat\rho_{[bcde]}$ for 
$\widehat\rho_{bcde}\in \Gamma(K\bigodot K)$. Hence the subbundle 
(\ref{decomposable_killing_2-tensors}) is flat provided that it is parallel.  
Since $\nabla_aR_{bc}{}^d{}_e=0$, any natural tensor bundle such as $K$, or 
$\Phi(\,\begin{picture}(6,6)(0,0)
\put(0,0){\line(1,0){6}}
\put(0,6){\line(1,0){6}}
\put(0,0){\line(0,1){6}}
\put(6,0){\line(0,1){6}}
\end{picture}\otimes\! K)$, or $\Phi(K\bigodot K)$
defined in terms of $R_{bc}{}^d{}_e$ is automatically parallel.  Looking at the
formula (\ref{our_tricky_prolongation_connection}) therefore, to verify that 
(\ref{decomposable_killing_2-tensors}) is parallel, we must show that
\begin{itemize}
\item $(R\triangleleft\sigma)_{abcd}\in
\Wedge^1\otimes\Phi(\,\begin{picture}(6,6)(0,0)
\put(0,0){\line(1,0){6}}
\put(0,6){\line(1,0){6}}
\put(0,0){\line(0,1){6}}
\put(6,0){\line(0,1){6}}
\end{picture}\otimes\! K)$, 
for all $\sigma_{bc}\in\bigodot^2\!\Wedge^1$, 
\item $\rho_{abcd}\in\Wedge^1\otimes\Phi(\,\begin{picture}(6,6)(0,0)
\put(0,0){\line(1,0){6}}
\put(0,6){\line(1,0){6}}
\put(0,0){\line(0,1){6}}
\put(6,0){\line(0,1){6}}
\end{picture}\otimes\! K)$, for all $\rho_{bcde}\in\Phi(K\bigodot K)$,
\item $(R\diamond\mu)_{abcde}\in\Wedge^1\otimes\Phi(K\bigodot K)$, for all
$\mu_{bcd}\in\Phi(\,\begin{picture}(6,6)(0,0)
\put(0,0){\line(1,0){6}}
\put(0,6){\line(1,0){6}}
\put(0,0){\line(0,1){6}}
\put(6,0){\line(0,1){6}}
\end{picture}\otimes\! K)$. 
\end{itemize}
The first of these is immediate from (\ref{option_one}) in conjunction with 
Theorem~\ref{LTS}.  For the second constraint, it suffices to observe that 
$$\widehat\rho_{abcd}=\widehat\rho_{(ab)(cd)}\enskip\mbox{and}\enskip
\widehat\rho_{abcd}=\widehat\rho_{cdab}\enskip\implies\enskip
\widehat\rho_{abcd}-\widehat\rho_{[abcd]}
=\widehat\rho_{abcd}-\widehat\rho_{a[bcd]}.$$
The third constraint is an easy consequence of Lemma~\ref{key_piece}. 
Specifically, notice that
$$\widehat{X}_{abcde}\equiv 
R_{bc}{}^f{}_d\widehat\mu_{aef}-R_{bc}{}^f{}_e\widehat\mu_{adf}
+R_{de}{}^f{}_b\widehat\mu_{acf}-R_{de}{}^f{}_c\widehat\mu_{abf}$$
vanishes if and only if 
$\widehat\mu_{bcd}\in
\begin{picture}(6,6)(0,0)
\put(0,0){\line(1,0){6}}
\put(0,6){\line(1,0){6}}
\put(0,0){\line(0,1){6}}
\put(6,0){\line(0,1){6}}
\end{picture}\otimes\! K$, in which case Lemma~\ref{key_piece} says that
$$(R\diamond\mu)_{abcde}=
R_{bc}{}^f{}_a\widehat\mu_{fde}+R_{de}{}^f{}_a\widehat\mu_{fbc}
-2R_{[bc}{}^f{}_{|a}\widehat\mu_{f|de]},$$
which is evidently a section of $\Wedge^1\otimes\Phi(K\bigodot K)$, as 
required.

Roughly speaking, Theorem~\ref{small_subbundle} gives in $\Delta$ a natural
lower bound concerning the parallel flat subbundles of
$\,\begin{picture}(12,6)(0,0)
\put(0,0){\line(1,0){12}}
\put(0,6){\line(1,0){12}}
\put(0,0){\line(0,1){6}}
\put(6,0){\line(0,1){6}}
\put(12,0){\line(0,1){6}}
\end{picture}\oplus
\begin{picture}(12,12)(0,2)
\put(0,0){\line(1,0){6}}
\put(0,6){\line(1,0){12}}
\put(0,12){\line(1,0){12}}
\put(0,0){\line(0,1){12}}
\put(6,0){\line(0,1){12}}
\put(12,6){\line(0,1){6}}
\end{picture}\oplus\begin{picture}(12,12)(0,2)
\put(0,0){\line(1,0){12}}
\put(0,6){\line(1,0){12}}
\put(0,12){\line(1,0){12}}
\put(0,0){\line(0,1){12}}
\put(6,0){\line(0,1){12}}
\put(12,0){\line(0,1){12}}
\end{picture}\,$ in the affine locally symmetric case.  Concerning an upper 
bound, recall that the connection 
(\ref{our_tricky_prolongation_connection}) on\/
$\begin{picture}(12,6)(0,0)
\put(0,0){\line(1,0){12}}
\put(0,6){\line(1,0){12}}
\put(0,0){\line(0,1){6}}
\put(6,0){\line(0,1){6}}
\put(12,0){\line(0,1){6}}
\end{picture}\oplus
\begin{picture}(12,12)(0,2)
\put(0,0){\line(1,0){6}}
\put(0,6){\line(1,0){12}}
\put(0,12){\line(1,0){12}}
\put(0,0){\line(0,1){12}}
\put(6,0){\line(0,1){12}}
\put(12,6){\line(0,1){6}}
\end{picture}\oplus\begin{picture}(12,12)(0,2)
\put(0,0){\line(1,0){12}}
\put(0,6){\line(1,0){12}}
\put(0,12){\line(1,0){12}}
\put(0,0){\line(0,1){12}}
\put(6,0){\line(0,1){12}}
\put(12,0){\line(0,1){12}}
\end{picture}$ is built from
$$R\triangleleft\underline{\enskip}:\begin{picture}(6,12)(0,2)
\put(0,0){\line(1,0){6}}
\put(0,6){\line(1,0){6}}
\put(0,12){\line(1,0){6}}
\put(0,0){\line(0,1){12}}
\put(6,0){\line(0,1){12}}
\end{picture}\to\Wedge^1\otimes
\begin{picture}(12,12)(0,2)
\put(0,0){\line(1,0){6}}
\put(0,6){\line(1,0){12}}
\put(0,12){\line(1,0){12}}
\put(0,0){\line(0,1){12}}
\put(6,0){\line(0,1){12}}
\put(12,6){\line(0,1){6}}
\end{picture},\enskip
R\diamond\underline{\enskip}:\begin{picture}(12,12)(0,2)
\put(0,0){\line(1,0){6}}
\put(0,6){\line(1,0){12}}
\put(0,12){\line(1,0){12}}
\put(0,0){\line(0,1){12}}
\put(6,0){\line(0,1){12}}
\put(12,6){\line(0,1){6}}
\end{picture}\to\Wedge^1\otimes\begin{picture}(12,12)(0,2)
\put(0,0){\line(1,0){12}}
\put(0,6){\line(1,0){12}}
\put(0,12){\line(1,0){12}}
\put(0,0){\line(0,1){12}}
\put(6,0){\line(0,1){12}}
\put(12,0){\line(0,1){12}}
\end{picture}\,,\enskip\mbox{and}\enskip
\partial:\begin{picture}(12,12)(0,2)
\put(0,0){\line(1,0){12}}
\put(0,6){\line(1,0){12}}
\put(0,12){\line(1,0){12}}
\put(0,0){\line(0,1){12}}
\put(6,0){\line(0,1){12}}
\put(12,0){\line(0,1){12}}
\end{picture}\hookrightarrow
\Wedge^1\otimes\begin{picture}(12,12)(0,2)
\put(0,0){\line(1,0){6}}
\put(0,6){\line(1,0){12}}
\put(0,12){\line(1,0){12}}
\put(0,0){\line(0,1){12}}
\put(6,0){\line(0,1){12}}
\put(12,6){\line(0,1){6}}
\end{picture}$$
with curvature carried by 
$R\bowtie\underline{\enskip}:\begin{picture}(12,12)(0,2)
\put(0,0){\line(1,0){12}} 
\put(0,6){\line(1,0){12}}
\put(0,12){\line(1,0){12}}
\put(0,0){\line(0,1){12}}
\put(6,0){\line(0,1){12}}
\put(12,0){\line(0,1){12}}
\end{picture}\to\begin{picture}(18,12)(0,2)
\put(0,0){\line(1,0){18}}
\put(0,6){\line(1,0){18}}
\put(0,12){\line(1,0){18}}
\put(0,0){\line(0,1){12}}
\put(6,0){\line(0,1){12}}
\put(12,0){\line(0,1){12}}
\put(18,0){\line(0,1){12}}
\end{picture}\,.$

\begin{thm} A parallel flat subbundle of\/
$\begin{picture}(12,6)(0,0)
\put(0,0){\line(1,0){12}}
\put(0,6){\line(1,0){12}}
\put(0,0){\line(0,1){6}}
\put(6,0){\line(0,1){6}}
\put(12,0){\line(0,1){6}}
\end{picture}\oplus
\begin{picture}(12,12)(0,2)
\put(0,0){\line(1,0){6}}
\put(0,6){\line(1,0){12}}
\put(0,12){\line(1,0){12}}
\put(0,0){\line(0,1){12}}
\put(6,0){\line(0,1){12}}
\put(12,6){\line(0,1){6}}
\end{picture}\oplus\begin{picture}(12,12)(0,2)
\put(0,0){\line(1,0){12}}
\put(0,6){\line(1,0){12}}
\put(0,12){\line(1,0){12}}
\put(0,0){\line(0,1){12}}
\put(6,0){\line(0,1){12}}
\put(12,0){\line(0,1){12}}
\end{picture}\,$ must be a subbundle of 
\begin{equation}\label{first_damper}
\left\{(\sigma,\mu,\rho)\raisebox{-1pt}{\LARGE$\mid$} 
\begin{array}{l}\rho\in\ker(R\bowtie\underline{\quad}),\enskip
(R\diamond\mu)\in\Wedge^1\otimes\ker(R\bowtie\underline{\quad}),\\[2pt]
\partial\rho\in\Wedge^1\otimes
\ker\big(({\mathrm{Id}}\otimes R\bowtie\underline{\enskip})\circ
(R\diamond\underline{\enskip})\big)\end{array}\right\}.\end{equation}
\end{thm}
\begin{proof}The initial constraint that
$\rho\in\ker(R\bowtie\underline{\enskip})$ is precisely that the curvature of
(\ref{our_tricky_prolongation_connection}) vanish on $(\sigma,\mu,\rho)$.  For
a parallel subbundle, the formula (\ref{our_tricky_prolongation_connection})
now constrains~$\mu$, namely that $(R\diamond\mu)$ end up in 
$\Wedge^1\otimes\ker(R\bowtie\underline{\quad})$, as stated. In other words, 
the composition
\begin{equation}\label{thorny}\begin{picture}(12,12)(0,2)
\put(0,0){\line(1,0){6}}
\put(0,6){\line(1,0){12}}
\put(0,12){\line(1,0){12}}
\put(0,0){\line(0,1){12}}
\put(6,0){\line(0,1){12}}
\put(12,6){\line(0,1){6}}
\end{picture}
\xrightarrow{\,\mbox{$R\diamond\underline{\enskip}$}\,}
\Wedge^1\otimes\begin{picture}(12,12)(0,2)
\put(0,0){\line(1,0){12}} 
\put(0,6){\line(1,0){12}}
\put(0,12){\line(1,0){12}}
\put(0,0){\line(0,1){12}}
\put(6,0){\line(0,1){12}}
\put(12,0){\line(0,1){12}}
\end{picture}
\xrightarrow{\,\mbox{${\mathrm{Id}}\otimes R\bowtie\underline{\enskip}$}\,}
\Wedge^1\otimes\begin{picture}(18,12)(0,2)
\put(0,0){\line(1,0){18}}
\put(0,6){\line(1,0){18}}
\put(0,12){\line(1,0){18}}
\put(0,0){\line(0,1){12}}
\put(6,0){\line(0,1){12}}
\put(12,0){\line(0,1){12}}
\put(18,0){\line(0,1){12}}
\end{picture}\end{equation}
kills $\mu$ and now, again
consulting~(\ref{our_tricky_prolongation_connection}), it follows that
$$\partial\rho\in\Wedge^1\otimes
\ker\big(({\mathrm{Id}}\otimes R\bowtie\underline{\enskip})\circ
(R\diamond\underline{\enskip})\big),$$
as stated.
\end{proof}
We are not claiming that (\ref{first_damper}) {\em is\/} parallel
but only that this subbundle necessarily dominates any parallel flat
subbundle.  In fact, the composition (\ref{thorny}) simplifies a little: 
\begin{lemma} The first and third lines of \eqref{R_diamond_mu} do not 
contribute to the composition~\eqref{thorny}. 
\end{lemma}
\begin{proof} As already discussed around (\ref{abstract_bowtie}), the pairing
$R\otimes\rho\mapsto R\bowtie\rho$ is, up to scale, the only non-zero
${\mathrm{GL}}(n,{\mathbb{R}})$-invariant possibility.  Hence, we may also view
$\begin{picture}(12,12)(0,2)
\put(0,0){\line(1,0){12}} 
\put(0,6){\line(1,0){12}}
\put(0,12){\line(1,0){12}}
\put(0,0){\line(0,1){12}}
\put(6,0){\line(0,1){12}}
\put(12,0){\line(0,1){12}}
\end{picture}
\xrightarrow{\,\mbox{$R\bowtie\underline{\enskip}$}\,}
\begin{picture}(18,12)(0,2)
\put(0,0){\line(1,0){18}}
\put(0,6){\line(1,0){18}}
\put(0,12){\line(1,0){18}}
\put(0,0){\line(0,1){12}}
\put(6,0){\line(0,1){12}}
\put(12,0){\line(0,1){12}}
\put(18,0){\line(0,1){12}}
\end{picture}$\,, up to scale, as the composition
$$\begin{picture}(12,12)(0,2)
\put(0,0){\line(1,0){12}} 
\put(0,6){\line(1,0){12}}
\put(0,12){\line(1,0){12}}
\put(0,0){\line(0,1){12}}
\put(6,0){\line(0,1){12}}
\put(12,0){\line(0,1){12}}
\end{picture}\subset\begin{picture}(6,12)(0,2)
\put(0,0){\line(1,0){6}}
\put(0,6){\line(1,0){6}}
\put(0,12){\line(1,0){6}}
\put(0,0){\line(0,1){12}}
\put(6,0){\line(0,1){12}}
\end{picture}\otimes\begin{picture}(6,12)(0,2)
\put(0,0){\line(1,0){6}}
\put(0,6){\line(1,0){6}}
\put(0,12){\line(1,0){6}}
\put(0,0){\line(0,1){12}}
\put(6,0){\line(0,1){12}}
\end{picture}
\xrightarrow{\,\mbox{\small${\mathcal{R}}\otimes{\mathrm{Id}}$}\,}
\begin{picture}(12,12)(0,2)
\put(0,0){\line(1,0){12}} 
\put(0,6){\line(1,0){12}}
\put(0,12){\line(1,0){12}}
\put(0,0){\line(0,1){12}}
\put(6,0){\line(0,1){12}}
\put(12,0){\line(0,1){12}}
\end{picture}\otimes\begin{picture}(6,12)(0,2)
\put(0,0){\line(1,0){6}}
\put(0,6){\line(1,0){6}}
\put(0,12){\line(1,0){6}}
\put(0,0){\line(0,1){12}}
\put(6,0){\line(0,1){12}}
\end{picture}=
\begin{picture}(18,12)(0,2)
\put(0,0){\line(1,0){18}}
\put(0,6){\line(1,0){18}}
\put(0,12){\line(1,0){18}}
\put(0,0){\line(0,1){12}}
\put(6,0){\line(0,1){12}}
\put(12,0){\line(0,1){12}}
\put(18,0){\line(0,1){12}}
\end{picture}\oplus\cdots
\xrightarrow{\,\mbox{projection}\,}
\begin{picture}(18,12)(0,2)
\put(0,0){\line(1,0){18}}
\put(0,6){\line(1,0){18}}
\put(0,12){\line(1,0){18}}
\put(0,0){\line(0,1){12}}
\put(6,0){\line(0,1){12}}
\put(12,0){\line(0,1){12}}
\put(18,0){\line(0,1){12}}
\end{picture}\,,$$
where 
$${\mathcal{R}}:\begin{picture}(6,12)(0,2)
\put(0,0){\line(1,0){6}}
\put(0,6){\line(1,0){6}}
\put(0,12){\line(1,0){6}}
\put(0,0){\line(0,1){12}}
\put(6,0){\line(0,1){12}}
\end{picture}\to\begin{picture}(12,12)(0,2)
\put(0,0){\line(1,0){12}} 
\put(0,6){\line(1,0){12}}
\put(0,12){\line(1,0){12}}
\put(0,0){\line(0,1){12}}
\put(6,0){\line(0,1){12}}
\put(12,0){\line(0,1){12}}
\end{picture}\quad\mbox{is given by}\enskip
\mu_{bc}\longmapsto
R_{ab}{}^e{}_{[c}\mu_{d]e}+R_{cd}{}^e{}_{[a}\mu_{b]e}.$$
According to (\ref{this_is_K}), the kernel of ${\mathcal{R}}$ is
$K$.  In particular, we see that
$$\textstyle\begin{picture}(6,12)(0,2)
\put(0,0){\line(1,0){6}}
\put(0,6){\line(1,0){6}}
\put(0,12){\line(1,0){6}}
\put(0,0){\line(0,1){12}}
\put(6,0){\line(0,1){12}}
\end{picture}\otimes\begin{picture}(6,12)(0,2)
\put(0,0){\line(1,0){6}}
\put(0,6){\line(1,0){6}}
\put(0,12){\line(1,0){6}}
\put(0,0){\line(0,1){12}}
\put(6,0){\line(0,1){12}}
\end{picture}\to\begin{picture}(6,12)(0,2)
\put(0,0){\line(1,0){6}}
\put(0,6){\line(1,0){6}}
\put(0,12){\line(1,0){6}}
\put(0,0){\line(0,1){12}}
\put(6,0){\line(0,1){12}}
\end{picture}\bigodot\begin{picture}(6,12)(0,2)
\put(0,0){\line(1,0){6}}
\put(0,6){\line(1,0){6}}
\put(0,12){\line(1,0){6}}
\put(0,0){\line(0,1){12}}
\put(6,0){\line(0,1){12}}
\end{picture}\to
\begin{picture}(12,12)(0,2)
\put(0,0){\line(1,0){12}} 
\put(0,6){\line(1,0){12}}
\put(0,12){\line(1,0){12}}
\put(0,0){\line(0,1){12}}
\put(6,0){\line(0,1){12}}
\put(12,0){\line(0,1){12}}
\end{picture}
\xrightarrow{\,\mbox{$R\bowtie\underline{\enskip}$}\,}
\begin{picture}(18,12)(0,2)
\put(0,0){\line(1,0){18}}
\put(0,6){\line(1,0){18}}
\put(0,12){\line(1,0){18}}
\put(0,0){\line(0,1){12}}
\put(6,0){\line(0,1){12}}
\put(12,0){\line(0,1){12}}
\put(18,0){\line(0,1){12}}
\end{picture}$$
vanishes on 
$K\otimes\begin{picture}(6,12)(0,2)
\put(0,0){\line(1,0){6}}
\put(0,6){\line(1,0){6}}
\put(0,12){\line(1,0){6}}
\put(0,0){\line(0,1){12}}
\put(6,0){\line(0,1){12}}
\end{picture}$\,.
Now, the first line of (\ref{R_diamond_mu}) starts with 
$$\textstyle R_{bc}{}^f{}_a\mu_{fde}\in
\Wedge^1\otimes K\otimes\begin{picture}(6,12)(0,2)
\put(0,0){\line(1,0){6}}
\put(0,6){\line(1,0){6}}
\put(0,12){\line(1,0){6}}
\put(0,0){\line(0,1){12}}
\put(6,0){\line(0,1){12}}
\end{picture}$$
and the rest of this line is, in fact, the composition
$\textstyle
K\otimes\begin{picture}(6,12)(0,2)
\put(0,0){\line(1,0){6}}
\put(0,6){\line(1,0){6}}
\put(0,12){\line(1,0){6}}
\put(0,0){\line(0,1){12}}
\put(6,0){\line(0,1){12}}
\end{picture}\subseteq\begin{picture}(6,12)(0,2)
\put(0,0){\line(1,0){6}}
\put(0,6){\line(1,0){6}}
\put(0,12){\line(1,0){6}}
\put(0,0){\line(0,1){12}}
\put(6,0){\line(0,1){12}}
\end{picture}\otimes\begin{picture}(6,12)(0,2)
\put(0,0){\line(1,0){6}}
\put(0,6){\line(1,0){6}}
\put(0,12){\line(1,0){6}}
\put(0,0){\line(0,1){12}}
\put(6,0){\line(0,1){12}}
\end{picture}\to\begin{picture}(6,12)(0,2)
\put(0,0){\line(1,0){6}}
\put(0,6){\line(1,0){6}}
\put(0,12){\line(1,0){6}}
\put(0,0){\line(0,1){12}}
\put(6,0){\line(0,1){12}}
\end{picture}\bigodot\begin{picture}(6,12)(0,2)
\put(0,0){\line(1,0){6}}
\put(0,6){\line(1,0){6}}
\put(0,12){\line(1,0){6}}
\put(0,0){\line(0,1){12}}
\put(6,0){\line(0,1){12}}
\end{picture}\to
\begin{picture}(12,12)(0,2)
\put(0,0){\line(1,0){12}} 
\put(0,6){\line(1,0){12}}
\put(0,12){\line(1,0){12}}
\put(0,0){\line(0,1){12}}
\put(6,0){\line(0,1){12}}
\put(12,0){\line(0,1){12}}
\end{picture}$ with the index $a$ as a passenger.  We have just checked that 
(\ref{abstract_bowtie}) vanishes on such tensors. 

Similarly, the third line is a linear combination of tensors such as 
$$R_{bc}{}^f{}_d\mu_{eaf}-R_{bc}{}^f{}_e\mu_{daf}
\quad\mbox{or}\quad
R_{bc}{}^f{}_d\mu_{aef}-R_{bc}{}^f{}_e\mu_{adf}\quad\mbox{lying in}\enskip
\Wedge^1\otimes K\otimes\begin{picture}(6,12)(0,2)
\put(0,0){\line(1,0){6}}
\put(0,6){\line(1,0){6}}
\put(0,12){\line(1,0){6}}
\put(0,0){\line(0,1){12}}
\put(6,0){\line(0,1){12}}
\end{picture}$$
together with
$$R_{de}{}^f{}_b\mu_{caf}-R_{de}{}^f{}_c\mu_{baf}
\quad\mbox{or}\quad
R_{de}{}^f{}_b\mu_{acf}-R_{de}{}^f{}_c\mu_{abf}\quad\mbox{lying in}\enskip
\Wedge^1\otimes \begin{picture}(6,12)(0,2)
\put(0,0){\line(1,0){6}}
\put(0,6){\line(1,0){6}}
\put(0,12){\line(1,0){6}}
\put(0,0){\line(0,1){12}}
\put(6,0){\line(0,1){12}}
\end{picture}\otimes K,$$
and we are done.\end{proof} 

Thus, neglecting the first and third lines of (\ref{R_diamond_mu}) gives an
equivalent set of constraints on a parallel section $(\sigma,\mu,\rho)$ and, in 
particular, constraints on the component 
$\rho\in\begin{picture}(12,12)(0,2)
\put(0,0){\line(1,0){12}} 
\put(0,6){\line(1,0){12}}
\put(0,12){\line(1,0){12}}
\put(0,0){\line(0,1){12}}
\put(6,0){\line(0,1){12}}
\put(12,0){\line(0,1){12}}
\end{picture}$\,, firstly that
$$\rho\in\ker(R\bowtie\underline\enskip)$$
and secondly that 
$$R_{b(c}{}^p{}_{e)}\rho_{a(df)p}-R_{b(d}{}^p{}_{e)}\rho_{a(cf)p}
-R_{b(c}{}^p{}_{f)}\rho_{a(de)p}+R_{b(d}{}^p{}_{f)}\rho_{a(ce)p}\,,$$
which evidently lies in $\Wedge^1\otimes\Wedge^1\otimes
\begin{picture}(12,12)(0,2)
\put(0,0){\line(1,0){12}} 
\put(0,6){\line(1,0){12}}
\put(0,12){\line(1,0){12}}
\put(0,0){\line(0,1){12}}
\put(6,0){\line(0,1){12}}
\put(12,0){\line(0,1){12}}
\end{picture}$\,, actually lies in 
$\Wedge^1\otimes\Wedge^1\otimes\ker(R\bowtie\underline\enskip)$.  These
constraints on $\rho$ have also been obtained by Matveev and
Nikolayevsky~\cite{MNprivate}.  They phrase them as conditions on the second
derivative $\nabla_a\nabla_b\sigma_{cd}$ of a Killing $2$-tensor $\sigma_{cd}$
at $p\in M$ under the assumption that $\sigma_{cd}|_p=0$ and
$\nabla_b\sigma_{cd}|_p=0$.

\section{Interactions with representation theory}\label{interactions}
We pause our exposition of the general theory to consider what it tells us
about specific examples and also how our various constructions interact with
Lie derivative in case we have a locally symmetric (pseudo-)Riemannian metric.
In this case, by definition, not only does our connection annihilate its
curvature $\nabla_aR_{bc}{}^d{}_e=0$ (as we have been supposing) but this
connection is also supposed to be the Levi-Civita connection for a
metric~$g_{ab}$.  By `raising and lowering indices' with $g_{ab}$ and its
inverse $g^{ab}$, there is no distinction between a $1$-form $\sigma_b$ and the
corresponding vector field $\sigma^a\equiv g^{ab}\sigma_b$.  The {\em Lie
derivative\/} of a covariant $2$-tensor $\phi_{bc}$ along a vector field $X^a$
can be written as
$${\mathcal{L}}_X\phi_{bc}
=X^a\nabla_a\phi_{bc}+(\nabla_bX^a)\phi_{ac}+(\nabla_cX^a)\phi_{ba}$$
for any torsion-free connection~$\nabla_a$. In particular, if we take 
$\phi_{bc}$ to be our metric $g_{bc}$ and $\nabla_a$ to be its Levi-Civita 
connection, then we conclude that 
$${\mathcal{L}}_Xg_{bc}=(\nabla_bX^a)g_{ac}+(\nabla_cX^a)g_{ba}
=\nabla_b(X^ag_{ac})+\nabla_c(X^ag_{ba})
=\nabla_bX_c+\nabla_cX_b=2\nabla_{(a}X_{b)}.$$
We obtain, therefore, a geometric interpretation of a Killing $1$-form
$\nabla_{(a}X_{b)}=0$, namely that the Lie derivative of the corresponding
vector field $X^a$ annihilate the metric.  One usually says that $X^a$ is a
{\em Killing field\/} and it follows that the local flow of $X^a$ is an
isometry.  In particular, if $\phi_{bc}$ is a Killing $2$-tensor, then so
is~${\mathcal{L}}_X\phi_{bc}$.
Of course, this will be reflected in all of our constructions and the Killing
connection (\ref{Killing_connection}) is the tool that makes it all work,
especially in conjunction with Theorem~\ref{LTS} above.  
\begin{thm}
On a locally symmetric (pseudo-)Riemannian manifold, the action of the local
isometries is transitive.
\end{thm}
\begin{proof} We have already observed, from the formula
(\ref{killing_curvature}), that the kernel of the curvature of
(\ref{Killing_connection}) is $\Wedge^1\oplus K$ where $K$ is given
by~(\ref{this_is_K}).  But, according to Corollary~\ref{parallel_subbundle}
this subbundle is also parallel.  So we can generate the Killing fields near
any $p\in M$ by choosing an arbitrary value for $X^a|_p\in T_pM\cong\Wedge_p^1$
and for $\nabla_aX_b|_p\in K_p$ and then propagating these using the flat 
Killing connection (\ref{Killing_connection}) on $\Wedge^1\oplus K$.  In 
particular, the Killing fields near $p$ point in any direction at~$p$. 
\end{proof}

Either working locally or supposing that $M$ is connected and simply 
connected, let us write the Killing fields as~${\mathfrak{g}}$, noting that 
${\mathfrak{g}}$ is a Lie algebra under Lie bracket
$$[X,Y]^b\equiv X^a\nabla_aY^b-(\nabla_aX^b)Y^a.$$
We have just observed that, for any point $p\in M$, we can identify 
${\mathfrak{g}}$ as a vector space with $(\Wedge^1\oplus K)_p$ and it is 
traditional to write this direct sum decomposition as
\begin{equation}\label{traditional}
{\mathfrak{g}}={\mathfrak{m}}\oplus{\mathfrak{k}}.
\end{equation}
Since $K_p$ may be identified with the Killing fields near $p$ that vanish
at~$p$, it is clear that ${\mathfrak{k}}$ is Lie subalgebra
of~${\mathfrak{g}}$.  We can confirm this directly---it
is~(\ref{K_is_a_bundle_of_subalgebras}) in the following theorem, viewed at any
chosen basepoint~$p\in M$.  
\begin{thm} In the locally symmetric (pseudo-)Riemannian case, suppose
$(\sigma_b,\mu_{bc})$ and $(\widetilde\sigma_b,\widetilde\mu_{bc})$ are
parallel sections of\/ $\Wedge^1\oplus K$ for the Killing
connection~\eqref{Killing_connection}.  Then
\begin{equation}\label{Lie_bracket}\big(
\sigma^a\widetilde\mu_{ab}-\widetilde\sigma^a\mu_{ab},
2\mu{}^a{}_{[b}\widetilde\mu_{c]a}
-R_{bc}{}^{ad}\sigma_a\widetilde\sigma_d
\big),\end{equation}
is also a parallel section of $\Wedge^1\oplus K$.
\begin{equation}\label{K_is_a_bundle_of_subalgebras}
\mbox{If}\enskip\mu_{bc},\widetilde\mu_{bc}\in\Gamma(K),\ \mbox{then}
\enskip\mu^a{}_{[b}\widetilde\mu_{c]a}\in\Gamma(K).\end{equation}
\end{thm}
\begin{proof}
In the (pseudo-)Riemannian locally symmetric case,
Theorem~\ref{killing_prolongation} allows us to regard the parallel section
$(\sigma_b,\mu_{bc})$ of $\Wedge^1\oplus K$ as a Killing field 
$X^a\equiv g^{ab}\sigma_b$ with covariant derivative 
$\mu_b{}^c\equiv\nabla_bX^c$. Applying ${\mathcal{L}}_X$ to 
$(\widetilde\sigma_b,\widetilde\mu_{bc})$ gives
$$\big(X^a\nabla_a\widetilde\sigma_b+(\nabla_bX^a)\widetilde\sigma_a,
X^a\nabla_a\widetilde\mu_{bc}
+(\nabla_bX^a)\widetilde\mu_{ac}+(\nabla_cX^a)\widetilde\mu_{ba}
\big),$$
but now $\nabla_a\widetilde\mu_{bc}=R_{bc}{}^d{}_a\widetilde\sigma_d$ since
$(\widetilde\sigma_b,\widetilde\mu_{bc})$ is parallel
for~(\ref{Killing_connection}). Thus, we obtain
$$\big(
\sigma^a\widetilde\mu_{ab}+\mu_b{}^a\widetilde\sigma_a,
\sigma^aR_{bc}{}^d{}_a\widetilde\sigma_d
+\mu_b{}^a\widetilde\mu_{ac}+\mu_c{}^a\widetilde\mu_{ba}
\big),$$
equivalently,
$$\big(
\sigma^a\widetilde\mu_{ab}-\widetilde\sigma^a\mu_{ab},
\sigma^aR_{bc}{}^d{}_a\widetilde\sigma_d
+2\mu{}^a{}_{[b}\widetilde\mu_{c]a}
\big).$$
Finally, the conclusion (\ref{K_is_a_bundle_of_subalgebras}) follows from
(\ref{Lie_bracket}) because it is true at any basepoint $p\in M$, in a
neighbourhood of which we may take parallel sections $(\sigma_b,\mu_{bc})$ and
$(\widetilde\sigma_b,\widetilde\mu_{bc})$ so that
$\sigma_b(p)=0=\widetilde\sigma_b(p)$ but 
$\mu_{bc}(p),\widetilde\mu_{bc}(p)\in K_p$ are, otherwise, unconstrained. 
Alternatively, it is an interesting exercise to check 
(\ref{K_is_a_bundle_of_subalgebras}) algebraically from~(\ref{this_is_K}). 
\end{proof}
In fact, the Lie derivative ${\mathcal{L}}_X$ acts on any of our constructions
of bundles with their connections and, in the (pseudo-)Riemannian case, being
constructed functorially from the metric, any Killing field will take parallel
sections to parallel sections.  In particular, this applies to the connection
(\ref{our_tricky_prolongation_connection}).  This will enable us to draw
conclusions concerning the action of the local isometries on the space of
Killing $2$-tensors.

To see how this works, let us consider a specific example in more detail.  For
the rest of this section, let $M={\mathrm{SU}}(6)/{\mathrm{Sp}}(3)$, the space
of quaternionic structures on ${\mathbb{C}}^6$ compatible with a fixed
Hermitian metric.  It is a Riemannian symmetric space of dimension~$14$.  We
shall denote the semisimple Lie algebras by their Dynkin diagrams and record
their finite-dimensional irreducible representations by attaching non-negative
integers to the nodes following the conventions of~\cite{Beastwood}.  (Strictly
speaking, this notation concerns {\em complex\/} semisimple Lie algebras and
their {\em complex\/} representations so we should at least work with
complex-valued tensor fields and perhaps also complexify the manifold~$M$ but,
for the sake of brevity, we shall leave such niceties to the reader.)  The free
software LiE~\cite{LiE} does an excellent job of working with semisimple
algebras and their representations in terms of roots and weights.  In
particular, we may restrict the adjoint representation of ${\mathfrak{su}}(6)$
to ${\mathfrak{sp}}(3)$ and write
\begin{equation}\label{adjoint_A5}
\afive{1}{0}{0}{0}{1}=\cthree{0}{1}{0}\oplus\cthree{2}{0}{0}\end{equation}
for the resulting branching, which in view of
Theorem~\ref{classical_locally_symmetric_conclusion}, we may regard as
identifying $\Wedge^1\oplus K$ as a homogeneous bundle on 
${\mathrm{SU}}(6)/{\mathrm{Sp}}(3)$, specifically
$$\begin{array}{ccccc}\Wedge^1&=&\Wedge^1&=&\cthree{0}{1}{0}\\[-3pt] 
\oplus&&\oplus&&\oplus\\
\Wedge^2&\supset&K&=&\cthree{2}{0}{0}\end{array}.$$
LiE programs for branching an irreducible representation of any compact Lie
group to any symmetric subgroup are given in~\cite{EW} (and the possible
compact symmetric spaces $G/K$ are listed in~\cite[\S8.11]{W}).  A discussion
of these and other relevant LiE routines are confined to an
appendix~\S\ref{LiE_appendix}.  LiE checks, for example, that
$$\Wedge^2=\Wedge^2(\cthree{0}{1}{0})=\cthree{1}{0}{1}\oplus\cthree{2}{0}{0}$$
and so the whole prolonged connection (\ref{Killing_connection}) is defined on 
the homogeneous bundle
$$\begin{array}{ccc}\Wedge^1&=&\cthree{0}{1}{0}\\[-3pt]
\oplus&&\oplus\\
\Wedge^2&=&\cthree{1}{0}{1}\oplus\cthree{2}{0}{0}
\end{array}$$
and the flat parallel subbundle $\Wedge^1\oplus K$ is neatly viewed as
\begin{equation}\label{neat_view}\begin{array}{c}\cthree{0}{1}{0}\\[-3pt]
\oplus\\
\cthree{2}{0}{0}
\end{array}\begin{picture}(35,0)
\put(0,10){$\xrightarrow{\quad\nabla\quad}$}
\put(0,-16){$\xrightarrow{\quad\nabla\quad}$}
\thicklines
\put(2,-5){\vector(3,1){33}}
\put(2,6){\vector(3,-1){33}}
\end{picture}
\begin{array}{c}\cthree{2}{0}{0}\oplus\cdots\\[-3pt]
\oplus\phantom{\oplus\cdots{}}\\
\cthree{0}{1}{0}\oplus\cdots\end{array}
\begin{array}{c}=\Wedge^1\otimes\cthree{0}{1}{0}\\[-3pt]
\phantom{=\Wedge^1\otimes{}}\oplus\\
=\Wedge^1\otimes\cthree{2}{0}{0}\end{array},\end{equation}
where the homomorphisms 
$\begin{picture}(35,0)(0,-3) 
\thicklines
\put(2,-5){\vector(3,1){33}} 
\put(2,6){\vector(3,-1){33}} \end{picture}$ 
are canonical inclusions, unique up to scale by Schur's Lemma.  {From} this
point of view, it looks like the two lines in this connection are on an equal
footing and, indeed, since ${\mathrm{SU}}(6)/{\mathrm{Sp}}(3)$ is neither
Hermitian nor flat, this is in accordance with~\cite[Lemma~4]{CELM}, which
says that for $\mu_{bc}$ a $2$-form under these circumstances,  
$$\nabla_a\mu_{bc}=R_{bc}{}^d{}_a\sigma_d\implies\nabla_a\sigma_b=\mu_{ab}
\enskip\mbox{(so $\sigma_b$ is a Killing $1$-form).}$$
In particular, although the two homomorphisms making 
up~(\ref{neat_view}), namely
$$K\hookrightarrow\Wedge^2\ni\mu_{bc}\stackrel{\partial}{\longmapsto}
\mu_{ab}\in\Wedge^1\otimes\Wedge^1$$
and
$$\Wedge^1\ni\sigma_b\longmapsto 
R_{bc}{}^d{}_a\sigma_d\in\Wedge^1\otimes K$$
have different origins, both are certain equivariant homomorphisms on the 
homogeneous space~${\mathrm{SU}}(6)/{\mathrm{Sp}}(3)$ (and are very much 
constrained by this equivariance).
On any tensor field $\phi_{bc\cdots d}$, if $X^a$ is a Killing field such that 
$\nabla_aX^b$ vanishes at $p\in M$, then 
$${\mathcal{L}}_X\phi_{bc\cdots d}|_p=X^a\nabla_a\phi_{bc\cdots d}|_p$$ 
and so this same diagram~(\ref{neat_view}) also records the action of the
Killing fields on the space of parallel sections viewed as propagated
from~$p\in M$, specifically that the action of ${\mathfrak{k}}$ in the
decomposition (\ref{traditional}) of ${\mathfrak{g}}$ is just the infinitesimal
action of ${\mathrm{Sp}}(3)$ on the homogeneous bundles (as written) and the
action of $X^a\in{\mathfrak{m}}$ is by $X^a\nabla_a$ with the result viewed
back in $\Wedge^1\oplus K$ via the canonical inclusions
$\begin{picture}(35,0)(0,-3) \thicklines \put(2,-5){\vector(3,1){33}}
\put(2,6){\vector(3,-1){33}} \end{picture}$.  In particular, at $p\in M$, we
see that the Lie bracket
$[\enskip\,,\enskip]:{\mathfrak{g}}\times{\mathfrak{g}}\to{\mathfrak{g}}$ 
interacts with the decomposition (\ref{traditional}) as follows
$$\begin{array}{ll}
{}[\enskip\,,\enskip]:{\mathfrak{k}}\times{\mathfrak{m}}\to{\mathfrak{m}}
&\qquad{}
[\enskip\,,\enskip]:{\mathfrak{m}}\times{\mathfrak{m}}\to{\mathfrak{k}}\\[2pt]
{}[\enskip\,,\enskip]:{\mathfrak{k}}\times{\mathfrak{k}}\to{\mathfrak{k}}
&\qquad{}
[\enskip\,,\enskip]:{\mathfrak{m}}\times{\mathfrak{k}}\to{\mathfrak{m}}\,,
\end{array}$$
which is a general feature~\cite[Theorem~3]{eschenbu}
of~(\ref{traditional}).

This observation regarding $\Wedge^1\oplus K$ (that the connection encodes the
action of ${\mathfrak{g}}$ as Killing fields) applies to all our flat
prolongation bundles, for example $\bigodot^2(\Wedge^1\oplus K)$, as seen in
Theorem~\ref{max_parallel_flat_on_symmetric_product}.  Here, for example, is
the bundle $\bigodot^2(\Wedge^1\oplus K)$ for
${\mathrm{SU}}(6)/{\mathrm{Sp}}(3)$ viewed in this way:
\begin{equation}\label{viewed_as_homogeneous}\begin{array}{ccc}
\bigodot^2\!\Wedge^1
&=&\cthree{0}{0}{0}\oplus\cthree{0}{1}{0}\oplus\cthree{0}{2}{0}\\
\oplus\\
\Wedge^1\otimes K
&=&\cthree{0}{1}{0}\oplus\cthree{1}{0}{1}
\oplus\cthree{2}{0}{0}\oplus\cthree{2}{1}{0}\\
\oplus\\
K\bigodot K
&=&\cthree{0}{0}{0}\oplus\cthree{0}{1}{0}
\oplus\cthree{0}{2}{0}\oplus\cthree{4}{0}{0}
\end{array}\end{equation}
and we can spot some parallel subbundles by representation theory.  Firstly,
since the metric itself is parallel $\nabla_ag_{bc}=0$, this is certainly a
Killing $2$-tensor. The corresponding parallel section for the connection 
(\ref{our_tricky_prolongation_connection}) is
\begin{equation}\label{metric_as_parallel_section}
\left[\!\begin{array}{c}\sigma_{bc}\\ \mu_{bcd}\\
\rho_{bcde}\end{array}\!\right]=
\left[\!\begin{array}{c}g_{bc}\\ 0\\
R_{bcde}\end{array}\!\right]\stackrel{\nabla_a}{\longmapsto}
\left[\!\begin{array}{c}\nabla_ag_{bc}\\
-(R\triangleleft g)_{abcd}-R_{abcd}\\
\nabla_aR_{bcde}
\end{array}\!\right]
=\left[\!\begin{array}{c}0\\
-(R\triangleleft g)_{abcd}-R_{abcd}\\
0\end{array}\!\right]\end{equation}
since, according to (\ref{lefttriangle}) and the Bianchi symmetry,
$$\textstyle(R\triangleleft g)_{abcd}
\equiv\frac23R_{cdba}-\frac13R_{bcda}+\frac13R_{bdca}
=-\frac23R_{cdab}-\frac13R_{abcd}=-R_{abcd}.$$
It is clear from the decomposition (\ref{viewed_as_homogeneous}) how
(\ref{metric_as_parallel_section}) fits nicely as a parallel section and, in
particular, that the metric $g_{ab}$ is decomposable as a Killing $2$-tensor.
This observation, that
$$\textstyle [g_{bc},0,R_{bcde}]\quad\mbox{is a parallel section of}\enskip
\bigodot^2(\Wedge^1\oplus K),$$
works for all Riemannian locally symmetric spaces.  As for the case
$M={\mathrm{SU}}(6)/{\mathrm{Sp}}(3)$, from (\ref{branch20002}) we can spot
another flat bundle with connection
\begin{equation}\label{well_spotted}
\begin{array}{c}\cthree{0}{2}{0}\\[-3pt]
\oplus\\
\cthree{2}{1}{0}\\[-3pt]
\oplus\\
\cthree{4}{0}{0}
\end{array}\begin{picture}(35,0)
\put(0,25){$\xrightarrow{\qquad\nabla\qquad}$}
\put(0,-2){$\xrightarrow{\quad\enskip\nabla\enskip}$}
\put(0,-29){$\xrightarrow{\hspace{58pt}\nabla\quad\enskip}$}
\thicklines
\put(2,9){\vector(4,1){56}}
\put(2,21){\vector(3,-1){33}}
\put(2,-22){\line(6,1){102}}\put(104,-5){\vector(4,1){0}}
\put(2,-2){\line(6,-1){84}}\put(86,-16){\vector(4,-1){0}}
\end{picture}
\begin{array}{c}\cthree{2}{1}{0}\oplus\cdots\\[-3pt]
\oplus\phantom{\oplus\cdots{}}\\
\cthree{0}{2}{0}\oplus\cthree{4}{0}{0}\oplus\cdots\\[-3pt]
\quad\oplus\\
\qquad\qquad\cthree{2}{1}{0}\oplus\cdots\end{array}
\begin{array}{c}=\Wedge^1\otimes\cthree{0}{2}{0}\\[-3pt]
\phantom{=\Wedge^1\otimes{}}\oplus\\
=\Wedge^1\otimes\cthree{2}{1}{0}\\[-3pt]
\phantom{=\Wedge^1\otimes{}}\oplus\\
=\Wedge^1\otimes\cthree{4}{0}{0}\end{array}\end{equation}
sitting as a subbundle of (\ref{viewed_as_homogeneous}) by means of
\begin{equation}\label{how_it_sits}
\begin{array}{c}\cthree{0}{2}{0}\\[-3pt]
\oplus\\
\cthree{2}{1}{0}\\[-3pt]
\oplus\\
\cthree{4}{0}{0}
\end{array}\ni\left[\!\begin{array}{c}\sigma_{bc}\\ \mu_{bcd}\\
\rho_{bcde}\end{array}\!\right]\longmapsto
\left[\!\begin{array}{c}\sigma_{bc}\\ \mu_{bcd}\\
\rho_{bcde}+CR_{bc}{}^{pq}R_{dep}{}^r\sigma_{qr}
\end{array}\!\right]\end{equation}
for some suitable constant~$C$ (as detailed in~\S\ref{A5/C3}).

\section{Riemannian locally symmetric spaces of compact type}
%
\underline{\underbar{For the rest of this article}} (save
for~\S\ref{affine_appendix}), we shall suppose that $g_{ab}$ is an irreducible
locally symmetric Riemannian metric of compact type, normalised so that
$R_{ab}=g_{ab}$.
\begin{thm}\label{non-zero_KY_three-forms}
The only $(M,g_{ab})$ that admit a non-zero Killing-Yano\/ $3$-form
\begin{equation}\label{KY_equation}
\mu_{bcd}\in\Gamma(M,\Wedge^3)\enskip\mbox{such that}\enskip
\nabla_{(a}\mu_{b)cd}=0,\end{equation}
are
\begin{itemize}
\item $M\subseteq$ the sphere $S^n$ with its round metric for $n\geq3$,
\item $M\subseteq G$, a compact Lie group with its bi-invariant metric.
\end{itemize}
In this second case, the $3$-form $\mu_{bcd}$ is a constant multiple of the 
Cartan $3$-form and satisfies\/~$\nabla_a\mu_{bcd}=0$. 
\end{thm}
\begin{proof} If $M$ is a subset of the round sphere \big(so
$R_{abcd}=\frac1{n-1}(g_{ac}g_{bd}-g_{bc}g_{ad})$\big), then it is easy to
check that the prolongation connection (\ref{killing_yano_prolongation}) for
$3$-forms satisfying (\ref{KY_equation}) is flat.  It follows that the 
solution space has dimension 
$$\textstyle\rank\Wedge^3+\rank\Wedge^4=\frac16n(n-1)(n-2)
+\frac1{24}n(n-1)(n-2)(n-3)=\frac1{24}(n+1)n(n-1)(n-2)$$
(and can be identified with $\Wedge^4({\mathbb{R}}^{n+1})$ under the natural
action of ${\mathrm{SL}}(n+1,{\mathbb{R}})$ on~$S^n$).  Otherwise, Belgun,
Moroianu, and Semmelmann~\cite{BMS} have shown that the Killing-Yano forms of
any rank are parallel.  Suppose $\phi_{bcd}$ is such a parallel $3$-form
$\nabla_a\phi_{bcd}=0$.  Compact Lie groups certainly carry such a form, namely
the Cartan $3$-form:
$$\phi_{bcd}X^bY^cZ^d\equiv\left<[X,Y],Z\right>,\quad\mbox{where }
\begin{array}[t]{l}
X^b,Y^c,Z^d\enskip\mbox{are left-invariant vector fields},\\
\mbox{and }\left<\enskip\,,\enskip\right>\enskip\mbox{is the Killing form}.
\end{array}$$
Conversely, a parallel form $\phi_{bcd}$ may
be used to break up the Killing fields into two types.  Specifically, if we
normalise $\phi_{bcd}$ such that $\phi_a{}^{cd}\phi_{bcd}=g_{ab}$ (see, for
example~\cite[Lemma~1]{CELM}, which says that any parallel symmetric $2$-tensor
on an irreducible locally symmetric space must be a constant multiple of the
metric), then we claim that
\begin{equation}\label{claimed_splitting}
K^1(M,\nabla)=\{\sigma_b\mid\nabla_a\sigma_b=\phi_{abc}\sigma^c\}\oplus
\{\sigma_b\mid\nabla_a\sigma_b=-\phi_{abc}\sigma^c\}.\end{equation}
To see this, let us firstly observe that $\nabla_b\phi_{cde}=0$ implies
$$R_{ab}{}^f{}_c\phi_{fde}+R_{ab}{}^f{}_d\phi_{cfe}
+R_{ab}{}^f{}_e\phi_{cdf}=0.$$
Tracing over $bc$ and employing the Bianchi symmetry in the form 
$R_a{}^{cf}{}_d=-\frac12R^{cf}{}_{ad}$ yields
$$\phi_{ade}=R^{cf}{}_{a[d}\phi_{e]cf}.$$
However, since $R^{cf}{}_{ad}\phi_{ecf}$ is skew in~$ad$, it 
follows from (\ref{partial_inverse}) that  
$$R^{cf}{}_{ad}\phi_{ecf}=\phi_{ead},$$
i.e.~that each $X^e\phi_{eab}$ is an eigenvector of $R_{ab}{}^{cd}$ as an
operator $\Wedge^2\ni\mu_{ab}\mapsto R_{ab}{}^{cd}\mu_{cd}\in\Wedge^2$, with
eigenvalue~$1$.  But, for example~\cite[Theorem~2]{CELM}, implies that $K$ is
the direct sum of the eigenspaces of $R_{ab}{}^{cd}$ with non-zero eigenvalues.
In particular, we conclude that
$$\phi_{eab}\in\Gamma\big(\Wedge^3\cap(\Wedge^1\otimes K)\big).$$
Furthermore, since $R_{ab}{}^{ab}=\dim M$, it follows that $X^e\phi_{eab}$ 
spans all eigenspaces with eigenvalue~$1$. In other words 
\begin{equation}\label{group_isomorphism}
\Wedge^1\ni\sigma_b\longmapsto\phi_{bcd}\sigma^d\in K\subset\Wedge^2
\end{equation}
is an isomorphism.  Our normalisation $\phi_a{}^{bc}\phi_{bc}{}^d=\delta_a{}^d$
now implies that the endomorphism
\begin{equation}\label{really_nice_endomorphism}
\begin{array}{c}\Wedge^1\\[-4pt] \oplus\\[-2pt] \Wedge^2\end{array}\ni
\left[\!\begin{array}{c}\sigma_b\\ \mu_{bc}\end{array}\!\right]\longmapsto
\left[\!\begin{array}{c}\phi_b{}^{cd}\mu_{cd}\\ 
\phi_{bc}{}^d\sigma_d\end{array}\!\right]\end{equation}
squares to the identity on $\Wedge^1\oplus K$. Therefore, we may write 
\begin{equation}\label{splitting}
\begin{array}{c}\Wedge^1\\[-4pt] \oplus\\[-1pt] K\end{array}=
\left\{\left[\!\begin{array}{c}\sigma_b\\ \mu_{bc}\end{array}\!\right]
\mid\mu_{ab}=\phi_{abc}\sigma^c\right\}\oplus
\left\{\left[\!\begin{array}{c}\sigma_b\\ \mu_{bc}\end{array}\!\right]
\mid\mu_{ab}=-\phi_{abc}\sigma^c\right\}.\end{equation}
But now it is straightforward to check that the endomorphism
(\ref{really_nice_endomorphism}) commutes with the 
connection (\ref{Killing_connection}) and so the splitting (\ref{splitting}) 
decomposes $\Wedge^1\oplus K$ into parallel subbundles.  The decomposition
(\ref{claimed_splitting}) follows. Finally, we may employ the formula 
(\ref{Lie_bracket}) to conclude that the subspaces in 
(\ref{claimed_splitting}) are closed under Lie bracket and these two subspaces 
commute. We have found the left-invariant and right-invariant vector fields 
on~$G$.
\end{proof}

\begin{cor}\label{injection}
The homomorphism \eqref{K1->K2} is injective unless
\begin{itemize}
\item $M\subseteq$ the sphere with its round metric,
\item $M\subseteq G$, a compact Lie group with its bi-invariant metric. 
\end{itemize}
\end{cor}
\begin{proof} As already remarked in the third bullet point just after
Corollary~\ref{SES_of_connections}, this follows from
Theorems~\ref{max_parallel_flat_on_symmetric_product}
and~\ref{non-zero_KY_three-forms}.  Alternatively, Nagy~\cite[Theorem~3.1]{N}
gives a purely algebraic argument to show that $\Wedge^3\cap(\Wedge^1\otimes
K)=0$ except for the sphere or a Lie group, in which case it is $1$-dimensional
and spanned by the Cartan form.  He also shows that
$\Wedge^4\cap(\Wedge^2\otimes K)=0$ except for the sphere.  (His arguments also
cover the non-compact case and hyperbolic space.)
\end{proof}
Bearing in mind the arguments in the proof of
Theorem~\ref{non-zero_KY_three-forms}, especially the emergence of the parallel
Cartan $3$-form $\phi_{bcd}$ in case $M\subseteq G$, it is straightforward to
write down explicitly, a non-zero element in the kernel of the homomorphism
(\ref{K1->K2}) in this case.  The key observation is that
$R_{bc}{}^{de}\phi_{ade}=\phi_{abc}$.  In particular, the Cartan $3$-form
$\phi_{bcd}$ is a section of $\Wedge^1\otimes K$.  Also, we can identify the
Riemann curvature tensor as follows.  
\begin{lemma}\label{group_curvature} The Riemann curvature tensor on a compact
Lie group $G$ with its bi-invariant metric is given by
\begin{equation}\label{group_curvature_formula}
R_{abcd}=\phi_{ab}{}^e\phi_{cde},\end{equation}
where $\phi_{abc}$ is the Cartan form normalised so that 
$\phi_a{}^{cd}\phi_{bcd}=g_{ab}$.
\end{lemma}
\begin{proof} Since both sides of (\ref{group_curvature_formula}) are sections
of $K\bigodot K$ and since (\ref{group_isomorphism}) is an isomorphism, it
suffices to check that
$$R_{abcd}\phi^{cd}{}_p=\phi_{ab}{}^e\phi_{cde}\phi^{cd}{}_p,$$
which is true by the same key observation.
\end{proof} 
Now we claim that
$$\left[\!\begin{array}{c}0\\ \phi_{bcd}\\
0\end{array}\!\right]\in
\begin{array}{c}
\bigodot^2\!\Wedge^1\\[-4pt] \oplus\\[-2pt]
\Wedge^1\otimes K\\[-2pt]
\oplus\\[-1pt]
\bigodot^2\!K
\end{array}
=\textstyle\bigodot^2\!\left(\!\begin{array}{c}
\Wedge^1\\[-4pt] \oplus\\[-2pt]
K\end{array}\!\right)$$
is parallel for the connection (\ref{symmetric_power}).  Accordingly, we only
need check that
$$R_{bc}{}^f{}_a\phi_{fde}+R_{de}{}^f{}_a\phi_{fbc}=0$$
and, by Lemma~\ref{group_curvature}, this expression becomes
$$\phi_{bc}{}^p\phi_{fap}\phi^f{}_{de}+\phi_{de}{}^p\phi_{fap}\phi^f{}_{bc}
=\phi_{bc}{}^p\phi_{fap}\phi^f{}_{de}+\phi_{de}{}^f\phi_{paf}\phi^p{}_{bc}
=0,$$
as required.

More abstractly, on any compact symmetric space $M=G/K$, the Killing form on
${\mathfrak{g}}$ regarded as a element of
$\bigodot^2\!{\mathfrak{g}}=\bigodot^2\!K^1(M,\nabla)$, maps to the metric
$g_{ab}$ under the canonical homomorphism~(\ref{K1->K2}).  On a Lie group $G$,
however, there are two possible realisations of~${\mathfrak{g}}$, namely as the
left-invariant or right-invariant vector fields.  Consequently, there are two 
different elements in $\bigodot^2\!K^1(M,\nabla)$ that map to $g_{ab}$ 
under~(\ref{K1->K2}) and their difference hence lies in the kernel. 

As noted in the alternative proof of Corollary~\ref{injection}, if $M$ is
neither spherical nor a {\em group manifold\/}~\cite[Table~8.11.5]{W}, then
Nagy~\cite[Theorem~3.1]{N} implies that the compositions
$$\Wedge^1\otimes K\hookrightarrow\Wedge^1\otimes\Wedge^2\to
\begin{picture}(12,12)(0,2)
\put(0,0){\line(1,0){6}}
\put(0,6){\line(1,0){12}}
\put(0,12){\line(1,0){12}}
\put(0,0){\line(0,1){12}}
\put(6,0){\line(0,1){12}}
\put(12,6){\line(0,1){6}}
\end{picture}$$
and 
$$\textstyle K\bigodot K\hookrightarrow\Wedge^2\bigodot\Wedge^2\to
\begin{picture}(12,12)(0,2)
\put(0,0){\line(1,0){12}}
\put(0,6){\line(1,0){12}}
\put(0,12){\line(1,0){12}}
\put(0,0){\line(0,1){12}}
\put(6,0){\line(0,1){12}}
\put(12,0){\line(0,1){12}}
\end{picture}$$
are injective. Let us write $P$ and $Q$ for the respective cokernels
$$\Wedge^1\otimes K\hookrightarrow
\begin{picture}(12,12)(0,2)
\put(0,0){\line(1,0){6}}
\put(0,6){\line(1,0){12}}
\put(0,12){\line(1,0){12}}
\put(0,0){\line(0,1){12}}
\put(6,0){\line(0,1){12}}
\put(12,6){\line(0,1){6}}
\end{picture}\to P\to 0\quad\mbox{and}\quad
\textstyle K\bigodot K\hookrightarrow
\begin{picture}(12,12)(0,2)
\put(0,0){\line(1,0){12}}
\put(0,6){\line(1,0){12}}
\put(0,12){\line(1,0){12}}
\put(0,0){\line(0,1){12}}
\put(6,0){\line(0,1){12}}
\put(12,0){\line(0,1){12}}
\end{picture}\to Q\to 0.$$
Thus, according to Theorem~\ref{Phi_is_parallel} and the alternative proof
of Corollary~\ref{injection} above, we find
$$\textstyle\bigodot^2\!\left(\!\begin{array}{c}
\Wedge^1\\[-4pt] \oplus\\[-2pt]
K\end{array}\!\right)=\begin{array}{c}
\bigodot^2\!\Wedge^1\\[-4pt] \oplus\\[-2pt]
\Wedge^1\otimes K\\[-2pt]
\oplus\\[-1pt]
K\bigodot K
\end{array}$$
realised as a flat parallel subbundle of 
$\,\begin{picture}(12,6)(0,0)
\put(0,0){\line(1,0){12}}
\put(0,6){\line(1,0){12}}
\put(0,0){\line(0,1){6}}
\put(6,0){\line(0,1){6}}
\put(12,0){\line(0,1){6}}
\end{picture}\oplus
\begin{picture}(12,12)(0,2)
\put(0,0){\line(1,0){6}}
\put(0,6){\line(1,0){12}}
\put(0,12){\line(1,0){12}}
\put(0,0){\line(0,1){12}}
\put(6,0){\line(0,1){12}}
\put(12,6){\line(0,1){6}}
\end{picture}\oplus\begin{picture}(12,12)(0,2)
\put(0,0){\line(1,0){12}}
\put(0,6){\line(1,0){12}}
\put(0,12){\line(1,0){12}}
\put(0,0){\line(0,1){12}}
\put(6,0){\line(0,1){12}}
\put(12,0){\line(0,1){12}}
\end{picture}\,$ with its prolongation connection 
(\ref{our_tricky_prolongation_connection}) from 
Theorem~\ref{prolonged_2-tensors}. We obtain, therefore, a connection on the 
quotient bundle $P\oplus Q$.
\begin{thm}\label{hidden_criterion}
Suppose $M$ is neither spherical nor a group manifold. If there 
are hidden symmetries on $M$, then $P\oplus Q$ admits non-trivial parallel 
sections. \end{thm}
\begin{proof} A hidden symmetry gives a parallel section of 
$\,\begin{picture}(12,6)(0,0)
\put(0,0){\line(1,0){12}}
\put(0,6){\line(1,0){12}}
\put(0,0){\line(0,1){6}}
\put(6,0){\line(0,1){6}}
\put(12,0){\line(0,1){6}}
\end{picture}\oplus
\begin{picture}(12,12)(0,2)
\put(0,0){\line(1,0){6}}
\put(0,6){\line(1,0){12}}
\put(0,12){\line(1,0){12}}
\put(0,0){\line(0,1){12}}
\put(6,0){\line(0,1){12}}
\put(12,6){\line(0,1){6}}
\end{picture}\oplus\begin{picture}(12,12)(0,2)
\put(0,0){\line(1,0){12}}
\put(0,6){\line(1,0){12}}
\put(0,12){\line(1,0){12}}
\put(0,0){\line(0,1){12}}
\put(6,0){\line(0,1){12}}
\put(12,0){\line(0,1){12}}
\end{picture}\,$ not contained in the subbundle 
$\bigodot^2(\Wedge^1\oplus K)$ and hence a parallel section of the quotient 
bundle.
\end{proof}
Conversely, we do not know whether a parallel section of $P\oplus Q$ need lift
to a parallel section of the prolongation bundle
$\,\begin{picture}(12,6)(0,0)
\put(0,0){\line(1,0){12}}
\put(0,6){\line(1,0){12}}
\put(0,0){\line(0,1){6}}
\put(6,0){\line(0,1){6}}
\put(12,0){\line(0,1){6}}
\end{picture}\oplus
\begin{picture}(12,12)(0,2)
\put(0,0){\line(1,0){6}}
\put(0,6){\line(1,0){12}}
\put(0,12){\line(1,0){12}}
\put(0,0){\line(0,1){12}}
\put(6,0){\line(0,1){12}}
\put(12,6){\line(0,1){6}}
\end{picture}\oplus\begin{picture}(12,12)(0,2)
\put(0,0){\line(1,0){12}}
\put(0,6){\line(1,0){12}}
\put(0,12){\line(1,0){12}}
\put(0,0){\line(0,1){12}}
\put(6,0){\line(0,1){12}}
\put(12,0){\line(0,1){12}}
\end{picture}\,$ 
(and therefore to a genuine hidden symmetry).  But in any particular case, if
$P\oplus Q$ has parallel sections, then, in conjunction with the computation of
$P$ and $Q$ as homogeneous bundles via the LiE routines
in~\S\ref{LiE_appendix}, we will know where to look (see, for example,
\S\ref{answer_for_E6/F4}).  Moreover, with a view to finding hidden symmetries,
the following construction is sometimes useful (see, for example, the proof of
Theorem~\ref{E6_F4_hidden}).
\begin{thm}\label{using_a_parallel_symmetric_3-tensor}
As a homogeneous space $M\subseteq G/K$, let us suppose that the bundle
$\bigodot^3\!\Wedge^1$ of symmetric $3$-tensors contains the trivial bundle
$\Wedge^0$ as a homogeneous subbundle.  Let $\phi_{bcd}$ denote the parallel
symmetric $3$-tensor corresponding to a constant section of\/~$\Wedge^0$.
Suppose $X^b$ is a Killing field on~$M$.  Then $\sigma_{bc}\equiv\phi_{bcd}X^d$
is a Killing $2$-tensor.
\end{thm}
\begin{proof} Since $X^d$ is a Killing field, we know that
$\mu_{ab}\equiv\nabla_aX_b$ is section of $K\subseteq\Wedge^2$ and since
$\nabla_a\phi_{bcd}=0$, we find that
$$\nabla_{(a}\sigma_{bc)}=\mu_{(a}{}^d\phi_{bc)d}.$$
However, the right hand side of this equation is the action of 
${\mathfrak{k}}$ on $\phi_{bcd}\in\Gamma(\bigodot^3\!\Wedge^1)$, which is 
trivial by assumption.
\end{proof}

\subsection{The round sphere}\label{the_sphere} The general answer, well-known
since~\cite{D,MMS,Takeuchi, Thompson}, is that the Killing $k$-tensors on $S^n$
form an irreducible representation
$$\begin{picture}(48,30)(0,-10)
\put(0,0){\line(1,0){48}}
\put(0,8){\line(1,0){48}}
\put(0,16){\line(1,0){48}}
\put(0,0){\line(0,1){16}}
\put(8,0){\line(0,1){16}}
\put(16,0){\line(0,1){16}}
\put(40,0){\line(0,1){16}}
\put(48,0){\line(0,1){16}}
\put(29,3.5){\makebox(0,0){$\cdots$}}
\put(29,11.5){\makebox(0,0){$\cdots$}}
\put(67,8){\makebox(0,0){$({\mathbb{R}^{n+1}})$}}
\put(24,-9){\makebox(0,0){$\underbrace{\hspace{48pt}}_k$}}
\end{picture}\hspace{50pt}
\raisebox{12pt}{of\enskip${\mathrm{SL}}(n+1,{\mathbb{R}})$,}$$
with ${\mathrm{SL}}(n+1,{\mathbb{R}})$ acting on $S^n$ by projective
transformations. Thus,
$$\rule[-9pt]{0pt}{19pt}\textstyle\bigodot^2\!K^1(S^n)
=\bigodot^2\begin{picture}(8,16)(0,3)
\put(0,0){\line(1,0){8}}
\put(0,8){\line(1,0){8}}
\put(0,16){\line(1,0){8}}
\put(0,0){\line(0,1){16}}
\put(8,0){\line(0,1){16}}
\end{picture}\:({\mathbb{R}^{n+1}})
=\begin{picture}(16,16)(0,3)
\put(0,0){\line(1,0){16}}
\put(0,8){\line(1,0){16}}
\put(0,16){\line(1,0){16}}
\put(0,0){\line(0,1){16}}
\put(8,0){\line(0,1){16}}
\put(16,0){\line(0,1){16}}
\end{picture}\:({\mathbb{R}^{n+1}})
\;\oplus\;
\begin{picture}(8,18)(0,3)
\put(0,-8){\line(1,0){8}}
\put(0,0){\line(1,0){8}}
\put(0,8){\line(1,0){8}}
\put(0,16){\line(1,0){8}}
\put(0,24){\line(1,0){8}}
\put(0,-8){\line(0,1){32}}
\put(8,-8){\line(0,1){32}}
\end{picture}\:({\mathbb{R}^{n+1}})
=K^2(S^n)\oplus KY^3(S^2),$$
where $KY^3(S^n)$ denotes the space of Killing-Yano $3$-forms on the round 
$n$-sphere. In particular, there are no hidden symmetries. Viewed under 
${\mathrm{SO}}(n+1)$, the Riemannian motions of $S^n$, we find a finer 
decomposition
$$K^2(S^n)=\begin{picture}(16,16)(0,3)
\put(0,0){\line(1,0){16}}
\put(0,8){\line(1,0){16}}
\put(0,16){\line(1,0){16}}
\put(0,0){\line(0,1){16}}
\put(8,0){\line(0,1){16}}
\put(16,0){\line(0,1){16}}
\put(20,0){\makebox(0,0){$\circ$}}
\end{picture}\enskip({\mathbb{R}^{n+1}})
\;\oplus\;
\begin{picture}(16,8)(0,1)
\put(0,0){\line(1,0){16}}
\put(0,8){\line(1,0){16}}
\put(0,0){\line(0,1){8}}
\put(8,0){\line(0,1){8}}
\put(16,0){\line(0,1){8}}
\put(20,0){\makebox(0,0){$\circ$}}
\end{picture}\enskip({\mathbb{R}^{n+1}})
\;\oplus\;{\mathbb{R}},$$
each of which can be identified with special types of Killing $2$-tensors 
according to their trace $\sigma\equiv g^{ab}\sigma_{ab}$, specifically
\begin{itemize}
\item \rule[-8pt]{0pt}{8pt}$\begin{picture}(16,16)(0,3)
\put(0,0){\line(1,0){16}}
\put(0,8){\line(1,0){16}}
\put(0,16){\line(1,0){16}}
\put(0,0){\line(0,1){16}}
\put(8,0){\line(0,1){16}}
\put(16,0){\line(0,1){16}}
\put(20,0){\makebox(0,0){$\circ$}}
\end{picture}\enskip({\mathbb{R}^{n+1}})\;\leftrightarrow \sigma=0$, with 
dimension $(n-2)(n+1)(n+2)(n+3)/12$,
\item \rule[-8pt]{0pt}{8pt}$\begin{picture}(16,8)(0,1)
\put(0,0){\line(1,0){16}}
\put(0,8){\line(1,0){16}}
\put(0,0){\line(0,1){8}}
\put(8,0){\line(0,1){8}}
\put(16,0){\line(0,1){8}}
\put(20,0){\makebox(0,0){$\circ$}}
\end{picture}\enskip({\mathbb{R}^{n+1}})\;\leftrightarrow 
\nabla^a\nabla_a\sigma+2\frac{n+1}{n-1}\sigma=0$, with dimension $n(n+3)/2$,
\item ${\mathbb{R}}\,\leftrightarrow\sigma$ is constant, with dimension $1$.
\end{itemize}

\subsection{Complex projective space} 
Sumitomo and Tandai \cite[Theorem 2.2, case $p=2$]{ST} show that
there are no hidden symmetries on ${\mathbb{CP}}_m$.  Alternatively, the space
of Killing $k$-tensors on ${\mathbb{CP}}_m$ is identified in 
\cite{E_cpn} as   
\begin{equation}\label{k-tensors_on_cpm}
\rule[-20pt]{0pt}{20pt}\Big(\begin{picture}(90,10)(-2,5)
\put(0,0){\line(1,0){48}}
\put(0,8){\line(1,0){48}}
\put(0,16){\line(1,0){48}}
\put(0,0){\line(0,1){16}}
\put(8,0){\line(0,1){16}}
\put(16,0){\line(0,1){16}}
\put(40,0){\line(0,1){16}}
\put(48,0){\line(0,1){16}}
\put(29,3.5){\makebox(0,0){$\cdots$}}
\put(29,11.5){\makebox(0,0){$\cdots$}}
\put(72,8){\makebox(0,0){$({\mathbb{C}^{m+1}})$}}
\put(24,-9){\makebox(0,0){$\underbrace{\hspace{48pt}}_k$}}
\put(52,-1){\makebox(0,0){$\perp$}}
\end{picture}\Big)_{\mathbb{R}}^{k,k},\end{equation}
where the subscript $\perp$ means to take $J_{\alpha\beta}$-trace-free part, 
the superscript $k,k$ means to take the $J_\alpha{}^\beta$-type-$(k,k)$ part, 
and the subscript ${\mathbb{R}}$ means to take the real subspace. For example, 
if $k=2$, this is 
$$\Wedge_{\perp,{\mathbb{R}}}^{1,1}({\mathbb{C}}^{m+1})$$
of dimension $(m+1)^2-1=m(m+2)=\dim{\mathfrak{su}}(m+1)$.  As detailed
in~\cite{E_cpn}, the dimension of (\ref{k-tensors_on_cpm}) when $k=2$ is
$$\frac{m(m+1)^2(m+2)}2=\frac{[m(m+2)][m(m+2)+1]}2$$
and, in conjunction with Corollary~\ref{injection}, we conclude that 
$\bigodot^2\!K^1({\mathbb{C}}_m)\to K^2({\mathbb{CP}}_m)$ is an isomorphism.

\subsection{Quaternionic projective space} Matveev and Nikolayevsky~\cite{MN} 
show that ${\mathbb{HP}}_k$ has hidden symmetries, more precisely that the
cokernel of 
$\bigodot^2\!K^1({\mathbb{HP}}_k)\hookrightarrow K^2({\mathbb{HP}}_k)$ is an 
irreducible representation of ${\mathrm{Sp}}(k+1)$ of dimension 
$(k-2)(k+1)(2k+1)(2k+3)/6$. 

\subsection{The Cayley plane}\label{OP_2} Matveev and Nikolayevsky~\cite{MN}
show that ${\mathbb{OP}}_2$ has hidden symmetries, more precisely that the
cokernel of $\bigodot^2\!K^1({\mathbb{OP}}_2)\hookrightarrow
K^2({\mathbb{OP}}_2)$ is an irreducible representation of $F_4$ dimension $26$.
We shall come back to this example in~\S\ref{answer_for_F4/B4} using the
machinery of this article and some parabolic geometry.

\subsection{The homogeneous space \mbox{$E_6/F_4$}}\label{answer_for_E6/F4}
We can completely treat this case using the machinery of this article.  All the
computations in this subsection are done using the computer program LiE, as
explained in \S\ref{LiE_appendix},

\begin{thm}\label{E6_F4_hidden} The homogeneous space $E_6/F_4$ admits hidden
symmetries, i.e.~Killing $2$-tensors not in the range of
$$\textstyle\bigodot^2\!K^1(E_6/F_4)\to K^2(E_6/F_4).$$
Indeed, we may find a $78$-dimensional space of hidden symmetries
that may be identified with the irreducible representation
$\rule[-6pt]{0pt}{28pt}\esix{0}{1}{0}{0}{0}{0}$ of $E_6$.
\end{thm}
\begin{proof}
For $E_6/F_4$, we have 
$${\mathfrak{e}}_6=\esix{0}{1}{0}{0}{0}{0}
=\begin{picture}(6,6)(0,0)
\put(0,0){\line(1,0){6}}
\put(0,6){\line(1,0){6}}
\put(0,0){\line(0,1){6}}
\put(6,0){\line(0,1){6}}
\end{picture}\oplus K
=\ffour{0}{0}{0}{1}\oplus\ffour{1}{0}{0}{0}$$
and we may, therefore, identify various other tensor bundles on $E_6/F_4$ as
follows:
$$\begin{array}{rcl}
\begin{picture}(6,12)(0,2)
\put(0,0){\line(1,0){6}}
\put(0,6){\line(1,0){6}}
\put(0,12){\line(1,0){6}}
\put(0,0){\line(0,1){12}}
\put(6,0){\line(0,1){12}}
\end{picture}&\!=\!&\ffour{0}{0}{1}{0}\oplus\ffour{1}{0}{0}{0}\\[4pt]
\begin{picture}(12,6)(0,0)
\put(0,0){\line(1,0){12}}
\put(0,6){\line(1,0){12}}
\put(0,0){\line(0,1){6}}
\put(6,0){\line(0,1){6}}
\put(12,0){\line(0,1){6}}
\end{picture}&\!=\!&
\ffour{0}{0}{0}{0}\oplus\ffour{0}{0}{0}{1}\oplus\ffour{0}{0}{0}{2}\\[4pt]
\begin{picture}(18,6)(0,0)
\put(0,0){\line(1,0){18}}
\put(0,6){\line(1,0){18}}
\put(0,0){\line(0,1){6}}
\put(6,0){\line(0,1){6}}
\put(12,0){\line(0,1){6}}
\put(18,0){\line(0,1){6}}
\end{picture}&\!=\!&
\ffour{0}{0}{0}{0}\oplus\ffour{0}{0}{0}{1}\oplus\ffour{0}{0}{0}{2}
\oplus\ffour{0}{0}{0}{3}\oplus\ffour{0}{0}{1}{0}\\[4pt]
\begin{picture}(12,12)(0,2)
\put(0,0){\line(1,0){6}}
\put(0,6){\line(1,0){12}}
\put(0,12){\line(1,0){12}}
\put(0,0){\line(0,1){12}}
\put(6,0){\line(0,1){12}}
\put(12,6){\line(0,1){6}}
\end{picture}&\!=\!&
2\ffour{0}{0}{0}{1}\oplus\ffour{0}{0}{0}{2}\oplus\ffour{0}{0}{1}{0}
\oplus\ffour{0}{0}{1}{1}\\[4pt]
&&{}\oplus\ffour{1}{0}{0}{0}\oplus\ffour{1}{0}{0}{1}
\end{array}$$
In particular, notice that, the bundle $\,\begin{picture}(18,6)(0,0)
\put(0,0){\line(1,0){18}}
\put(0,6){\line(1,0){18}}
\put(0,0){\line(0,1){6}}
\put(6,0){\line(0,1){6}}
\put(12,0){\line(0,1){6}}
\put(18,0){\line(0,1){6}}
\end{picture}=\bigodot^3\!\Wedge^1$ admits a (unique) trivial subbundle and we 
may invoke Theorem~\ref{using_a_parallel_symmetric_3-tensor} to construct 
$${\mathfrak{e}}_6=K^1(E_6/F_4)\hookrightarrow K^2(E_6/F_4).$$
These Killing $2$-tensors are complementary to the range of
$\bigodot^2\!K^1(E_6/F_4)\to K^2(E_6/F_4)$ because, from
Corollary~\ref{injection}, this homomorphism is injective and its range,
therefore, splits into
$$\textstyle\bigodot^2\!\Big(\esix{0}{1}{0}{0}{0}{0}\Big)=
\esix{0}{0}{0}{0}{0}{0}\oplus\esix{0}{2}{0}{0}{0}{0}
\oplus\esix{1}{0}{0}{0}{0}{1},$$
none of which is ${\mathfrak{e}}_6$. 
\end{proof}
\begin{thm} Together with the hidden symmetries constructed in
Theorem~\ref{E6_F4_hidden}, the range of the canonical injection (according to 
Corollary~\ref{injection})
\begin{equation}\label{can_inject}
\textstyle\bigodot^2\!K^1(E_6/F_4)\hookrightarrow K^2(E_6/F_4)
\end{equation}
provides all of the Killing $2$-tensors on $E_6/F_4$.
\end{thm}
\begin{proof} According to Theorem~\ref{hidden_criterion}, the hidden 
symmetries induce parallel sections of 
$$\begin{array}{ccl}P
&\!\!=\!\!&\ffour{0}{0}{0}{1}
\oplus\ffour{0}{0}{0}{2}
\oplus\ffour{0}{0}{1}{1}
\oplus\ffour{1}{0}{0}{0}\\ 
\oplus\\ 
Q&\!\!=\!\!&\!\!\begin{array}[t]{c}\ffour{0}{0}{0}{0}
\oplus\ffour{0}{0}{1}{0}
\oplus 2\ffour{0}{0}{0}{1}
\oplus 2\ffour{0}{0}{0}{2}\\[4pt]
{}\oplus\ffour{0}{0}{0}{3}
\oplus\ffour{0}{0}{1}{1}
\oplus\ffour{0}{0}{2}{0}\\[4pt]
{}\oplus\ffour{1}{0}{0}{1}
\oplus\ffour{1}{0}{1}{0}\end{array}
\end{array}$$
and, sure enough, the hidden symmetries
$\ffour{1}{0}{0}{0}\oplus\ffour{0}{0}{0}{1}$ from Theorem~\ref{E6_F4_hidden}
may be spotted amongst the irreducible pieces of $P\oplus Q$.  Moreover, since
there is only one copy of $\ffour{1}{0}{0}{0}$, this accounts for all potential
realisations of
$$\esix{0}{1}{0}{0}{0}{0}=\ffour{1}{0}{0}{0}\oplus\ffour{0}{0}{0}{1}$$
as hidden symmetries. Most other representations of $E_6$ are immediately 
eliminated. For example, 
$$\esix{0}{0}{0}{1}{0}{0}
=2\ffour{0}{0}{1}{0}\oplus\ffour{0}{1}{0}{0}\oplus\ffour{1}{0}{0}{0}\oplus
\ffour{1}{0}{0}{1}$$
and $\rule{0pt}{16pt}\ffour{0}{1}{0}{0}$ is nowhere to be seen in $P\oplus Q$.
Similarly, 
$$\begin{array}{rcl}
\esix{0}{0}{2}{0}{0}{0}&=&\ffour{2}{0}{0}{0}\oplus\cdots\\[16pt]
\esix{0}{0}{3}{0}{0}{0}&=&\ffour{3}{0}{0}{0}\oplus\cdots
\end{array}$$
are impossible to realise amongst the parallel sections of $P\oplus Q$.  Well,
this second option may be immediately ruled out on dimension grounds: it has
dimension 1,559,376 whereas $P\oplus Q$ is a vector bundle of rank~41,145.  A
case-by-case search quickly reduces, without loss of generality, to the
following three potential representations of~$E_6$.
\begin{itemize}
\item[A] $\rule{0pt}{20pt}\esix{1}{0}{0}{0}{0}{0}
=\ffour{0}{0}{0}{0}\oplus\ffour{0}{0}{0}{1}$ 
\item[B] $\rule{0pt}{32pt}\esix{0}{0}{1}{0}{0}{0}=
\ffour{0}{0}{0}{1}\oplus\ffour{0}{0}{1}{0}\oplus\ffour{1}{0}{0}{0}$
\item[C] $\rule[-10pt]{0pt}{42pt}\esix{1}{0}{0}{0}{0}{1}=
\ffour{0}{0}{0}{0}\oplus2\ffour{0}{0}{1}{0}\oplus\ffour{0}{0}{0}{2}
\oplus\ffour{0}{0}{1}{0},$
\end{itemize}
all of which may be ruled out as follows. For case~A, if the proposed 
representation of $E_6$ is realised as hidden symmetries, then the outer 
automorphism of $E_6$ gives rise to
$$\esix{1}{0}{0}{0}{0}{0}\oplus\esix{0}{0}{0}{0}{0}{1}
=2\ffour{0}{0}{0}{0}\oplus2\ffour{0}{0}{0}{1},$$
which must also be realised.  However, there is only copy of
$\ffour{0}{0}{0}{0}$ in $P\oplus Q$ so this results in a contradiction.  Case B
may be eliminated similarly because there is only copy of $\ffour{0}{0}{1}{0}$ 
(alternatively, because there is only copy of $\ffour{1}{0}{0}{0}$) in 
$P\oplus Q$.

It remains to eliminate case~C, this option being more subtle because all 
$F_4$-pieces are available in $P\oplus Q$. To proceed, let us notice that this 
particular representation of $E_6$ is already realised within the 
decomposable Killing $2$-tensors
$$\textstyle\bigodot^2\!\Big(\esix{0}{1}{0}{0}{0}{0}\Big)=
\esix{1}{0}{0}{0}{0}{1}\oplus\cdots$$
and, therefore, as the parallel sections of a subbundle of
\begin{equation}\label{subbundle}\begin{array}{ccc}
\ffour{0}{0}{0}{0}\oplus\ffour{0}{0}{0}{1}\oplus\ffour{0}{0}{0}{2}
&=&\begin{picture}(12,6)(0,0)
\put(0,0){\line(1,0){12}}
\put(0,6){\line(1,0){12}}
\put(0,0){\line(0,1){6}}
\put(6,0){\line(0,1){6}}
\put(12,0){\line(0,1){6}}
\end{picture}\\ 
\oplus\\
\framebox{$\ffour{0}{0}{0}{1}$}\oplus\ffour{0}{0}{1}{0}\oplus\ffour{1}{0}{0}{1}
&=&\begin{picture}(6,6)(0,0)
\put(0,0){\line(1,0){6}}
\put(0,6){\line(1,0){6}}
\put(0,0){\line(0,1){6}}
\put(6,0){\line(0,1){6}}
\end{picture}\otimes\! K\\ 
\oplus\\
\ffour{0}{0}{0}{0}\oplus\ffour{0}{0}{0}{2}\oplus\ffour{2}{0}{0}{0}
&=&K\bigodot K\makebox[0pt][l]{.}\end{array}
\end{equation}
In particular, notice that the two copies of $\ffour{0}{0}{0}{0}$ lie on
different lines to the sole copy of $\ffour{0}{0}{1}{0}$ in this bundle.
Algebraically, regarding the decomposition~(\ref{traditional}), this
implies firstly that the action of ${\mathfrak{m}}$ on 
$\rule{0pt}{22pt}\esix{0}{1}{0}{0}{0}{0}$ gives a homomorphism
$${\mathfrak{m}}\otimes\ffour{0}{0}{0}{0}
=\rule{0pt}{16pt}\ffour{0}{0}{0}{1}\otimes\ffour{0}{0}{0}{0}
\to\ffour{0}{0}{0}{1}$$
where this particular copy of $\ffour{0}{0}{0}{1}$ is as marked
\framebox{\rule{0pt}{6pt}\qquad} on the second line of (\ref{subbundle}) and
then it takes (at least) two more applications of this ${\mathfrak{m}}$-action
to reach the sole instance of $\ffour{0}{0}{1}{0}$, as this particular
$F_4$-module lies on the same line as \framebox{\rule{0pt}{6pt}\qquad}\,.  In
any case, the minimal number of ${\mathfrak{m}}$-actions to arrive at
$\ffour{0}{0}{1}{0}$ from $\rule{0pt}{14pt}\ffour{0}{0}{0}{0}$ is manifestly
\underbar{{\em odd\/}}.  Notice that this statement is purely algebraic and
therefore independent of the particular realisation encountered here.  On the
other hand, if we consider the purported case~C as parallel sections of
$P\oplus Q$, then these $F_4$-modules $\rule{0pt}{14pt}\ffour{0}{0}{0}{0}$ and
$\rule{0pt}{14pt}\ffour{0}{0}{1}{0}$ uniquely occur in~$Q$.  The minimal number
of $\rule{0pt}{14pt}{\mathfrak{m}}$-actions from $\ffour{0}{0}{0}{0}$ to
$\rule{0pt}{14pt}\ffour{0}{0}{1}{0}$ for such a realisation would therefore be 
\underline{{\em even\/}}, a contradiction. 
\end{proof}

\subsection{The homogeneous space \mbox{${\mathrm{SU}}(6)/{\mathrm{Sp}}(3)$}}
\label{A5/C3}
Following on from \S\ref{interactions}, we can completely treat this case using
the machinery of this article.  At the outset, let us observe that 
$$\textstyle\begin{picture}(18,6)(0,0)
\put(0,0){\line(1,0){18}}
\put(0,6){\line(1,0){18}}
\put(0,0){\line(0,1){6}}
\put(6,0){\line(0,1){6}}
\put(12,0){\line(0,1){6}}
\put(18,0){\line(0,1){6}}
\end{picture}
=\bigodot^3\!\cthree{0}{1}{0}
=\cthree{0}{0}{0}\oplus\cthree{0}{2}{0}\oplus\cthree{0}{3}{0}
\oplus\cthree{1}{0}{1}
\oplus\cthree{0}{1}{0},$$
which includes a trivial bundle whence, from
Theorem~\ref{using_a_parallel_symmetric_3-tensor}, we may use a
corresponding parallel symmetric $3$-tensor to construct
\begin{equation}\label{cool_construction}
\afive{1}{0}{0}{0}{1}=K^1({\mathrm{SU}}(6)/{\mathrm{Sp}}(3))
\hookrightarrow K^2({\mathrm{SU}}(6)/{\mathrm{Sp}}(3)).\end{equation} 
However, in contrast to the case of $E_6/F_4$ studied in
\S\ref{answer_for_E6/F4}, these Killing $2$-tensors are decomposable.  Indeed,
there are no hidden symmetries at all: \begin{thm} For any open subset\/
$M\subseteq{\mathrm{SU}}(6)/{\mathrm{Sp}}(3)$, the canonical homomorphism
$$\textstyle\bigodot^2\!K^1(M)\longrightarrow K^2(M)$$
is an isomorphism.
\end{thm}
\begin{proof} According to Theorem~\ref{hidden_criterion}, it suffices to show 
that the bundle
$$\begin{array}{ccl}P
&\!\!=\!\!&\cthree{0}{1}{0}\oplus\cthree{0}{2}{0}\oplus\cthree{1}{1}{1}\\
\oplus\\ 
Q&\!\!=\!\!&\!\!\begin{array}[t]{c}
\cthree{0}{0}{0}\oplus\cthree{0}{1}{0}\oplus2\cthree{0}{2}{0}
\oplus\cthree{0}{3}{0}\oplus\cthree{1}{0}{1}\\[4pt]
{}\oplus\cthree{1}{1}{1}\oplus\cthree{2}{0}{2}
\oplus\cthree{2}{1}{0}\oplus\cthree{3}{0}{1}
\end{array}
\end{array}$$
has no parallel sections.  A flat parallel subbundle would generate a
representation of ${\mathrm{SU}}(6)$ and, therefore, it suffices to consider
the irreducible representations of dimension bounded by
${\mathrm{rank}}(P\oplus Q)=3570$.  Firstly, there are the Killing fields
corresponding to the irreducible representation~(\ref{adjoint_A5}) but
$\cthree{2}{0}{0}$ is not amongst the irreducible bundles that make up 
$P\oplus Q$.  It already follows that the Killing $2$-tensors in the range 
of (\ref{cool_construction}) cannot be
hidden.  The decomposable Killing $2$-tensors make up
$$\textstyle\bigodot^2\!K^1(M)=\bigodot^2(\afive{1}{0}{0}{0}{1})
=\begin{array}[t]{l}\afive{0}{0}{0}{0}{0}\oplus\afive{0}{1}{0}{1}{0}\\[4pt]
{}\oplus\afive{1}{0}{0}{0}{1}\oplus\afive{2}{0}{0}{0}{2}
\end{array}$$
and a case-by-case search quickly identifies another copy of
$$\afive{0}{1}{0}{1}{0}=\cthree{0}{0}{0}\oplus2\cthree{0}{1}{0}
\oplus\cthree{0}{2}{0}\oplus\cthree{1}{0}{1}$$
as the only contender for hidden symmetries: certainly, all these pieces occur
in \mbox{$P\oplus Q$}.  Notice, however, that $\cthree{0}{0}{0}$ and
$\cthree{1}{0}{1}$ only occur in~$Q$.  Thus, the minimal number of
${\mathfrak{m}}$-actions needed to reach $\cthree{1}{0}{1}$ from
$\cthree{0}{0}{0}$ in such a realisation would be \underbar{{\em even\/}}.
This is at variance with the true number, which is manifestly 
\underbar{{\em odd\/}}, as can be seen from the known realisation of
$\afive{0}{1}{0}{1}{0}$ as a parallel subbundle
of~(\ref{viewed_as_homogeneous}).
\end{proof}
The fact that $K\subset\Wedge^2$ is irreducible gives rise to the following
useful identities.
\begin{lemma} The Riemann curvature tensor $R_{abcd}$ on\/
${\mathrm{SU}}(6)/{\mathrm{Sp}}(3)$ satisfies
\begin{equation}\label{useful}\textstyle\frac23R_{ab}{}_{cd}
=R_a{}^{ef}{}_cR_{befd}-R_b{}^{ef}{}_cR_{aefd}\end{equation}
and 
\begin{equation}\label{also_useful}
\textstyle\frac23R_{ab}{}_{cd}-\frac13R_{adcb}
=R_a{}^{ef}{}_cR_{befd}-R_c{}^{ef}{}_bR_{aefd}.\end{equation}
\end{lemma}
\begin{proof}
With Riemann tensor symmetries, Theorem~\ref{LTS} implies that
$R_{abcd}\in\Gamma(K\bigodot K)$ and now, since $K=\cthree{2}{0}{0}$ is
irreducible, we conclude that there is an eigenvalue $\lambda$ so that  
$$R_{ab}{}^{ef}\mu_{ef}=\lambda\mu_{ab},\enskip\forall\;\mu_{ab}\in K.$$
But $\rank\cthree{2}{0}{0}=21$ and 
$R_{ab}{}^{ab}=\rank\Wedge^1=\rank\cthree{0}{1}{0}=14$. Hence, it follows 
that $\lambda=2/3$. In particular, we conclude that
\begin{equation}\label{R_squared}
\textstyle R_{ab}{}^{ef}R_{cdef}=\frac23R_{abcd}.\end{equation}
On the other hand, Theorem~\ref{LTS} says that 
$$R_{ab}{}^e{}_{[c}R_{d]e}{}^p{}_q
+R_{cd}{}^e{}_{[a}R_{b]e}{}^p{}_q=0$$
and tracing this identity over $dq$ yields (after relabelling some indices), 
$$\textstyle R_{ab}{}_{cd}=
\frac12R_{ab}{}^{ef}R_{cdef}+R_a{}^{ef}{}_cR_{befd}-R_b{}^{ef}{}_cR_{aefd}.$$
This is in general, but ${\mathrm{SU}}(6)/{\mathrm{Sp}}(3)$
imposes~(\ref{R_squared}), and we find (\ref{useful}), as required.  Now, we
may employ the Bianchi symmetry in the form $R_{befc}
=R_{cefb}-R_{efbc}$ to obtain
$$\textstyle R_b{}^{ef}{}_cR_{aefd}
=R_c{}^{ef}{}_bR_{aefd}-R^{ef}{}_{bc}R_{aefd}
=R_c{}^{ef}{}_bR_{aefd}+\frac12R^{ef}{}_{bc}R_{efad}.$$
Again using (\ref{R_squared}), we deduce that 
$$\textstyle R_b{}^{ef}{}_cR_{aefd}
=R_c{}^{ef}{}_bR_{aefd}+\frac13R_{adbc}.$$
Substituting into (\ref{useful}), we find~(\ref{also_useful}),
as required.\end{proof}
{From} the decompositions in (\ref{viewed_as_homogeneous}) it seems reasonable
on ${\mathrm{SU}}(6)/{\mathrm{Sp}}(3)$ that
$$\textstyle\bigodot^2\!\Wedge^1\ni\sigma_{ab}
\mapsto R_{ab}{}^{pq}R_{dep}{}^r\sigma_{qr}\in K\bigodot K$$
should be injective and, conversely, that
$$\textstyle K\bigodot K\ni\rho_{abcd}
\mapsto\rho_a{}^c{}_{bc}\in\bigodot^2\!\Wedge^1$$
should be surjective. Their composition is
\begin{equation}\label{their_composition}
\textstyle\bigodot^2\!\Wedge^1\ni\sigma_{ab}\mapsto
R_a{}^{cpq}R_{bcp}{}^r\sigma_{qr}\end{equation}
a symmetric endomorphism of $\bigodot^2\!\Wedge^1$, which we may rewrite using 
(\ref{also_useful}) as
$$\textstyle \sigma_{ab}\mapsto
R_a{}^c{}_b{}^dR_c{}^q{}_d{}^r\sigma_{qr}-\frac13R_a{}^q{}_b{}^r\sigma_{qr}.$$
The Riemann curvature tensor may be regarded as a symmetric operator 
\begin{equation}\label{second_kind}
\textstyle\bigodot^2\!\Wedge^1\ni\sigma_{ab}\mapsto
R_a{}^c{}_b{}^d\sigma_{cd}\in\bigodot^2\!\Wedge^1\end{equation}
often referred to as the {\em curvature operator of the second kind\/}. If 
$\sigma_{ab}\in\bigodot^2\!\Wedge^1$ lies in the $\lambda$-eigenspace for this 
operator, then we deduce that
\begin{equation}\label{R_R_sigma}\textstyle R_a{}^{cpq}R_{bcp}{}^r\sigma_{qr}
=\lambda(\lambda-\frac13)\sigma_{ab}\end{equation}
and, in particular, is non-zero provided that $\lambda\not=0,1/3$.  In fact,
the eigenvalues of (\ref{second_kind}) are as follows.  
\begin{lemma}\label{eigenvalues}
On ${\mathrm{SU}}(6)/{\mathrm{Sp}}(3)$, if we normalise the
metric so that $R_{ab}=g_{ab}$, then the eigenspaces of \eqref{second_kind} are
$$\begin{array}{rcccl}
\cthree{0}{0}{0}&\mbox{of dimension}&1&\mbox{with eigenvalue}&1,\\
\cthree{0}{1}{0}&\mbox{of dimension}&14&\mbox{with eigenvalue}&1/2,\\
\cthree{0}{2}{0}&\mbox{of dimension}&90&\mbox{with eigenvalue}&1/15.
\end{array}$$
\end{lemma}
\begin{proof} The trivial bundle $\cthree{0}{0}{0}$ comprises multiples of the
metric $g_{ab}$ and under (\ref{second_kind}) we have
$$g_{ab}\mapsto R_a{}^c{}_b{}^dg_{cd}=R_{ab}=g_{ab}$$
by dint of our normalisation. Regarding $\cthree{0}{1}{0}$, recall that 
$$\textstyle\Wedge^1\ni 
X^d\mapsto\phi_{bcd}X^d\in\cthree{0}{1}{0}\subset\bigodot^2\!\Wedge^1$$
is an isomorphism. Since $\phi_{bcd}$ is parallel, it follows that
$$R_{ab}{}^f{}_c\phi_{def}+R_{ab}{}^f{}_d\phi_{ecf}+R_{ab}{}^f{}_e\phi_{cdf}
=0$$
and tracing over $bc$ gives
$$\phi_{dea}+R_a{}^{cf}{}_d\phi_{ecf}+R_a{}^{cf}{}_e\phi_{cdf}=0.$$
Therefore, if $R_a{}^c{}_b{}^d\phi_{cde}=\nu\phi_{cde}$ for some $\nu$, as it 
must by irreducibility of $\cthree{0}{1}{0}$, then we see that $\nu=1/2$. The 
remaining eigenvalue is now forced by the trace of (\ref{second_kind}):
$$1\!\times\!1+14\!\times\!1/2+90\!\times\!1/15=R_a{}^a{}_b{}^b=14
\enskip\checkmark,$$
as required
\end{proof}
{From} Lemma~\ref{eigenvalues} and (\ref{R_R_sigma}) we finally conclude that 
$$\textstyle\bigodot^2\!\Wedge^1\supset\cthree{0}{2}{0}\ni\sigma_{bc}
\longmapsto R_{bc}{}^{pq}R_{dep}{}^r\sigma_{qr}\in K\bigodot K$$
is non-zero and (\ref{how_it_sits}) is forced by Schur's Lemma. 

\subsection{The homogeneous space $F_4/{\mathrm{Spin}}(9)$}
\label{answer_for_F4/B4}
This is the Cayley plane already mentioned in~\S\ref{OP_2}, whose Killing
$2$-tensors were determined by Matveev and Nikolayevsky in~\cite{MN}.  These
authors use an explicit model via the Octonians to conclude that the
hidden symmetries on $F_4/{\mathrm{Spin}}(9)$ realise the irreducible 
representation $\ffour{0}{0}{0}{1}$ of~$F_4$ and, in particular, have 
dimension~$26$.  

{From} the representation-theoretic viewpoint, we can employ
\cite[Display~(5.6)]{EW} to show that, for the Killing fields on 
$F_4/{\mathrm{Spin}}(9)$, we have
$${\mathfrak{f}}_4=\ffour{1}{0}{0}{0}=\Wedge^1\oplus K
=\bfour{0}{0}{0}{1}\oplus\bfour{0}{1}{0}{0},$$
whence for the decomposable Killing $2$-tensors we have
$$\textstyle\bigodot^2\!\ffour{1}{0}{0}{0}
=\ffour{0}{0}{0}{0}\oplus\ffour{0}{0}{0}{2}\oplus\ffour{2}{0}{0}{0},$$
realised as parallel sections of
$$\begin{array}{ccc}
\begin{picture}(12,6)(0,0)
\put(0,0){\line(1,0){12}}
\put(0,6){\line(1,0){12}}
\put(0,0){\line(0,1){6}}
\put(6,0){\line(0,1){6}}
\put(12,0){\line(0,1){6}}
\end{picture}&=&
\bfour{0}{0}{0}{0}\oplus\bfour{0}{0}{0}{2}\oplus\bfour{1}{0}{0}{0}
\\ 
\oplus\\
\begin{picture}(6,6)(0,0)
\put(0,0){\line(1,0){6}}
\put(0,6){\line(1,0){6}}
\put(0,0){\line(0,1){6}}
\put(6,0){\line(0,1){6}}
\end{picture}\otimes\! K&=&
\bfour{0}{0}{0}{1}\oplus\bfour{0}{1}{0}{1}\oplus\bfour{1}{0}{0}{1}
\\ 
\oplus\\
K\bigodot K&=&
\bfour{0}{0}{0}{0}\oplus\bfour{0}{0}{0}{2}\oplus\bfour{0}{2}{0}{0}
\oplus\bfour{2}{0}{0}{0}
\makebox[0pt][l]{.}\end{array}$$
Regarding hidden symmetries, Theorem~\ref{hidden_criterion} says that they 
will show up as parallel sections of 
$$\begin{array}{ccl}P
&\!\!=\!\!&\framebox{$\rule[-4pt]{0pt}{4pt}\bfour{0}{0}{0}{1}$}
\oplus\bfour{0}{0}{1}{1}\\
\oplus\\ 
Q&\!\!=\!\!&\!\!\begin{array}[t]{c}
\framebox{$\rule[-4pt]{0pt}{4pt}\bfour{0}{0}{0}{0}\oplus\bfour{1}{0}{0}{0}$}
\oplus\bfour{0}{0}{1}{0}\oplus\bfour{0}{0}{2}{0}\\[4pt]
{}\oplus\rule{0pt}{14pt}\bfour{0}{1}{1}{0}\oplus\bfour{0}{0}{0}{2}
\oplus\bfour{1}{0}{0}{2}\end{array}
\end{array}$$
and, sure enough, we can (uniquely) identify
$$\ffour{0}{0}{0}{1}
=\bfour{0}{0}{0}{1}\oplus\bfour{0}{0}{0}{0}\oplus\bfour{1}{0}{0}{0},$$
as marked. 
%
%
%
Whilst it is not immediately clear that this parallel subbundle of 
$P\oplus Q$ lifts to a parallel subbundle of the prolongation bundle
$\,\begin{picture}(12,6)(0,0)
\put(0,0){\line(1,0){12}}
\put(0,6){\line(1,0){12}}
\put(0,0){\line(0,1){6}}
\put(6,0){\line(0,1){6}}
\put(12,0){\line(0,1){6}}
\end{picture}\oplus
\begin{picture}(12,12)(0,2)
\put(0,0){\line(1,0){6}}
\put(0,6){\line(1,0){12}}
\put(0,12){\line(1,0){12}}
\put(0,0){\line(0,1){12}}
\put(6,0){\line(0,1){12}}
\put(12,6){\line(0,1){6}}
\end{picture}\oplus\begin{picture}(12,12)(0,2)
\put(0,0){\line(1,0){12}}
\put(0,6){\line(1,0){12}}
\put(0,12){\line(1,0){12}}
\put(0,0){\line(0,1){12}}
\put(6,0){\line(0,1){12}}
\put(12,0){\line(0,1){12}}
\end{picture}\,$ (as is implied by~\cite{MN}), it is easy to check that there 
are no other hidden symmetries.  For example, the only $F_4$-module of
dimension less than or equal to ${\mathrm{rank}}(P\oplus Q)=5556$ and branching
under ${\mathrm{Spin}}(9)$ to contain $\bfour{0}{0}{1}{1}$ is
$$\ffour{0}{0}{1}{1}=\bfour{0}{0}{1}{1}\oplus\bfour{0}{1}{0}{0}\oplus\cdots
\quad\mbox{($14$ pieces in all})$$
of rank $4096$, which is clearly not available as a subbundle of $P\oplus Q$.

In fact, we may identify the hidden symmetries $\ffour{0}{0}{0}{1}$ on the
Cayley plane ${\mathbb{OP}}_2$ using some parabolic
geometry~\cite{parabook}.  As a $|1|$-graded parabolic geometry for $E_6$, we
have
$${\mathbb{OP}}_2=\xesix{}{}{}{}{}{},$$
the detailed structure of which may be found in~\cite{S}.  
The {\em BGG complex}
\begin{equation}\label{BGG_complex}
0\to\esix{0}{0}{0}{0}{0}{1}\to\xesix{0}{0}{0}{0}{0}{1}
\stackrel{\nabla}{\longrightarrow}\xesix{-2}{0}{1}{0}{0}{1}
\stackrel{\nabla}{\longrightarrow}\cdots\end{equation}
realises this $27$-dimensional representation of $E_6$ as the kernel of an 
invariant first order differential operator on~${\mathbb{OP}}_2$. This 
representation branches
\begin{equation}\label{metrics_plus_hidden}
\esix{0}{0}{0}{0}{0}{1}=\ffour{0}{0}{0}{0}\oplus\ffour{0}{0}{0}{1}
\end{equation}
and so we see, from the Riemannian viewpoint, the $26$-dimensional hidden 
symmetries potentially sitting inside this $27$-dimensional kernel. To confirm  
that this is really the case, observe that, for any choice of 
{\em scale\/}, the corresponding {\em Weyl connection}~\cite[\S5.1]{parabook} 
$\nabla:\Wedge^1\to\Wedge^1\otimes\Wedge^1$ induces
$$\begin{array}{ccl}\bigodot^2\!\Wedge^1&=&
\xesix{-4}{2}{0}{0}{0}{0}\oplus\xesix{-2}{0}{0}{0}{0}{1}\\[12pt] 
\nabla\!\!\downarrow&&\begin{picture}(0,0)
\thicklines
\put(60,12){\vector(0,-1){14}}
\put(70,12){\vector(4,-1){56}}
\put(136,12){\vector(0,-1){14}}
\end{picture}\\[10pt]
\bigodot^3\!\Wedge^1&=&
\xesix{-6}{3}{0}{0}{0}{0}\oplus\xesix{-4}{0}{1}{0}{0}{1}\,,
\end{array}$$

\smallskip\noindent
with arrows depicting the only possible differential operators according to
their symbols.  On the right hand side we see the first BGG operator from
(\ref{BGG_complex}) save for its {\em weight\/}, which disappears with our
choice of scale.  In any case, viewed on $F_4/{\mathrm{Spin}}(9)$, this first
BGG operator now becomes
$$\begin{array}{ccl}\xesix{0}{0}{0}{0}{0}{1}&\raisebox{6pt}{$=$}&
\raisebox{7pt}{$\bfour{1}{0}{0}{0}\oplus\bfour{0}{0}{0}{0}$}\\[8pt]
\nabla\!\!\downarrow\enskip\;&&\begin{picture}(0,0)
\thicklines
\put(34,16){\vector(0,-1){21}}
\put(44,16){\vector(3,-1){64}}
\put(118,16){\vector(0,-1){21}}
\end{picture}\\[10pt]
\xesix{-2}{0}{1}{0}{0}{1}&\raisebox{6pt}{$=$}&
\raisebox{7pt}{$\bfour{1}{0}{0}{1}\oplus\bfour{0}{0}{0}{1}$}\end{array}$$

\smallskip \noindent and its $27$-dimensional kernel splits as above
(\ref{metrics_plus_hidden}) into constant multiples of the metric and the
$26$-dimensional space of hidden symmetries (whose trace will not be zero but
rather constrained to lie in the first non-trivial eigenspace of the
Laplacian).


Finally, it is an interesting exercise to find the tractor bundle
$$\addtolength{\arraycolsep}{-2pt}\begin{array}{ccccccc}
\esix{0}{0}{0}{0}{0}{1}&=&
\xesix{0}{0}{0}{0}{0}{1}&+&\xesix{-1}{1}{0}{0}{0}{0}
&+&\xesix{-1}{0}{0}{0}{0}{0}\\[10pt]
&&\|&&\|&&\|\\
&&\makebox[0pt]{$\bfour{1}{0}{0}{0}\oplus\bfour{0}{0}{0}{0}$\hspace{44pt}}
&&\bfour{0}{0}{0}{1}&&\bfour{0}{0}{0}{0}\end{array}$$
as a flat parallel subbundle of the prolongation bundle
$\,\begin{picture}(12,6)(0,0)
\put(0,0){\line(1,0){12}}
\put(0,6){\line(1,0){12}}
\put(0,0){\line(0,1){6}}
\put(6,0){\line(0,1){6}}
\put(12,0){\line(0,1){6}}
\end{picture}\oplus
\begin{picture}(12,12)(0,2)
\put(0,0){\line(1,0){6}}
\put(0,6){\line(1,0){12}}
\put(0,12){\line(1,0){12}}
\put(0,0){\line(0,1){12}}
\put(6,0){\line(0,1){12}}
\put(12,6){\line(0,1){6}}
\end{picture}\oplus\begin{picture}(12,12)(0,2)
\put(0,0){\line(1,0){12}}
\put(0,6){\line(1,0){12}}
\put(0,12){\line(1,0){12}}
\put(0,0){\line(0,1){12}}
\put(6,0){\line(0,1){12}}
\put(12,0){\line(0,1){12}}
\end{picture}\,$.

\section{Appendix: LiE routines}\label{LiE_appendix}

For the calculations in this article, there are two aspects of the LiE 
program~\cite{LiE} that we found useful. Firstly, LiE has the ability to 
branch any representation of any semisimple algebra to any symmetric 
subalgebra, for example ${\mathfrak{g}}\supset{\mathfrak{k}}$ as 
in~(\ref{traditional}). Secondly, LiE can determine from the cotangent bundle 
$\Wedge^1=\begin{picture}(6,6)(0,0)
\put(0,0){\line(1,0){6}}
\put(0,6){\line(1,0){6}}
\put(0,0){\line(0,1){6}}
\put(6,0){\line(0,1){6}}
\end{picture}$\,, as a homogeneous bundle on $G/K$, the various associated 
bundles such as $\Wedge^2=\begin{picture}(6,12)(0,2)
\put(0,0){\line(1,0){6}}
\put(0,6){\line(1,0){6}}
\put(0,12){\line(1,0){6}}
\put(0,0){\line(0,1){12}}
\put(6,0){\line(0,1){12}}
\end{picture}$\,, or $\begin{picture}(12,12)(0,2)
\put(0,0){\line(1,0){6}}
\put(0,6){\line(1,0){12}}
\put(0,12){\line(1,0){12}}
\put(0,0){\line(0,1){12}}
\put(6,0){\line(0,1){12}}
\put(12,6){\line(0,1){6}}
\end{picture}$\,, or $\begin{picture}(12,12)(0,2)
\put(0,0){\line(1,0){12}}
\put(0,6){\line(1,0){12}}
\put(0,12){\line(1,0){12}}
\put(0,0){\line(0,1){12}}
\put(6,0){\line(0,1){12}}
\put(12,0){\line(0,1){12}}
\end{picture}$ as homogeneous bundles.

\subsection{Branching} All symmetric spaces are covered by the LiE
programs in~\cite{EW}.  For example, here~\cite[Display (3.7)]{EW} is the
routine for $E_6/F_4$
\small$$\begin{array}l
\verb!# file branch_E6_F4.lie # ! \\
\verb!branch_E6_F4(vec v) = setdefault(E6);! \\
\verb!print("the branching of "+v+" from E6 to F4 is");! \\
\verb!branch(v,F4,res_mat(F4,E6))! 
\end{array}$$\normalsize
and it may saved and imported or simply copied and pasted into a running
implementation of LiE.  To use it to branch the adjoint representation of 
$E_6$ to $F_4$, one can firstly ask LiE what is this adjoint representation:
\small$$\verb+adjoint(E6)+$$\normalsize
returns \small$\verb+1X[0,1,0,0,0,0]+$\normalsize, which is LiE's notation for 
$\esix{0}{1}{0}{0}{0}{0}$. And now
\small$$\verb+branch_E6_F4([0,1,0,0,0,0])+$$\normalsize
returns \small$\verb!1X[0,0,0,1] +1X[1,0,0,0]!$\normalsize, which says that
$$\rule{0pt}{20pt}\esix{0}{1}{0}{0}{0}{0}
=\ffour{0}{0}{0}{1}\oplus\ffour{1}{0}{0}{0}.$$

As another example, here is the LiE routine used in \S\ref{interactions} to
branch various irreducible representations of ${\mathrm{SU}}(6)$ to
${\mathrm{Sp}}(3)$ (a cut-down version of \cite[Display (3.3)]{EW}).
\small$$\begin{array}l \verb!# file branch_A5_C3.lie # !  \\
\verb!branch_A5_C3(vec v) = setdefault(A5); ! \\
\verb!res_wt = [[1,0,0],[0,1,0],[0,0,1],[0,1,0],[1,0,0]]; ! \\
\verb!print("the branching of "+v+" from SU(6) to Sp(3) is"); ! \\
\verb!branch(v,C3,res_wt) !  
\end{array}$$\normalsize
For example, \small$\verb!branch_A5_C3([2,0,0,0,2])!$\normalsize\ yields
\begin{equation}\label{branch20002}\afive{2}{0}{0}{0}{2}
=\cthree{0}{2}{0}\oplus\cthree{2}{1}{0}\oplus\cthree{0}{0}{4},\end{equation}
as used in \S\ref{interactions} to spot the flat 
connection~(\ref{well_spotted}).  

\subsection{Associated bundles} Again, let us consider the homogeneous space 
$E_6/F_4$ as a typical example.  We set ${\mathfrak{k}}$
as the default Lie algebra:
\small$$\verb+setdefault(F4)+$$\normalsize and write
\small$\verb+TT+$\normalsize\ for the tangent bundle
${\mathfrak{g}}/{\mathfrak{k}}$:
\small$$\verb+TT=1X[0,0,0,1]+$$\normalsize
We may then ask LiE for all the associated bundles appearing in this article.
\begin{itemize}\addtolength{\itemsep}{-5pt}
\item $\Wedge^1=\begin{picture}(6,6)(0,0)
\put(0,0){\line(1,0){6}}
\put(0,6){\line(1,0){6}}
\put(0,0){\line(0,1){6}}
\put(6,0){\line(0,1){6}}
\end{picture}=$
\small$$\verb+TT+$$\normalsize
\item $\Wedge^2=\begin{picture}(6,12)(0,2)
\put(0,0){\line(1,0){6}}
\put(0,6){\line(1,0){6}}
\put(0,12){\line(1,0){6}}
\put(0,0){\line(0,1){12}}
\put(6,0){\line(0,1){12}}
\end{picture}=$
\small$$\verb+alt_tensor(2,TT)+$$\normalsize
\item $\bigodot^2\!\Wedge^1=\begin{picture}(12,6)(0,0)
\put(0,0){\line(1,0){12}}
\put(0,6){\line(1,0){12}}
\put(0,0){\line(0,1){6}}
\put(6,0){\line(0,1){6}}
\put(12,0){\line(0,1){6}}
\end{picture}=$
\small$$\verb+sym_tensor(2,TT)+$$\normalsize
\item $\Wedge^3=\begin{picture}(6,18)(0,5)
\put(0,0){\line(1,0){6}}
\put(0,6){\line(1,0){6}}
\put(0,12){\line(1,0){6}}
\put(0,18){\line(1,0){6}}
\put(0,0){\line(0,1){18}}
\put(6,0){\line(0,1){18}}
\end{picture}=$
\small$$\verb+alt_tensor(3,TT)+$$\normalsize
\item $\bigodot^3\!\Wedge^1=\begin{picture}(18,6)(0,0)
\put(0,0){\line(1,0){18}}
\put(0,6){\line(1,0){18}}
\put(0,0){\line(0,1){6}}
\put(6,0){\line(0,1){6}}
\put(12,0){\line(0,1){6}}
\put(18,0){\line(0,1){6}}
\end{picture}=$
\small$$\verb+sym_tensor(3,TT)+$$\normalsize
\item $\begin{picture}(12,12)(0,2)
\put(0,0){\line(1,0){6}}
\put(0,6){\line(1,0){12}}
\put(0,12){\line(1,0){12}}
\put(0,0){\line(0,1){12}}
\put(6,0){\line(0,1){12}}
\put(12,6){\line(0,1){6}}
\end{picture}=$
\small$$\verb+tensor(TT,alt_tensor(2,TT))-alt_tensor(3,TT)+$$\normalsize
\item $\begin{picture}(12,12)(0,2)
\put(0,0){\line(1,0){12}}
\put(0,6){\line(1,0){12}}
\put(0,12){\line(1,0){12}}
\put(0,0){\line(0,1){12}}
\put(6,0){\line(0,1){12}}
\put(12,0){\line(0,1){12}}
\end{picture}=$
\small$$\verb+sym_tensor(2,alt_tensor(2,TT))-alt_tensor(4,TT)+$$\normalsize
\item $K=$
\small$$\verb+adjoint+$$\normalsize
\item $\begin{picture}(6,6)(0,0)
\put(0,0){\line(1,0){6}}
\put(0,6){\line(1,0){6}}
\put(0,0){\line(0,1){6}}
\put(6,0){\line(0,1){6}}
\end{picture}\otimes K=$
\small$$\verb+tensor(TT,adjoint)+$$\normalsize
\item $K\bigodot K=$
\small$$\verb+sym_tensor(2,adjoint)+$$\normalsize
\end{itemize}
and also
\begin{itemize}
\item $P\equiv\begin{picture}(12,12)(0,2)
\put(0,0){\line(1,0){6}}
\put(0,6){\line(1,0){12}}
\put(0,12){\line(1,0){12}}
\put(0,0){\line(0,1){12}}
\put(6,0){\line(0,1){12}}
\put(12,6){\line(0,1){6}}
\end{picture}/\big(\begin{picture}(6,6)(0,0)
\put(0,0){\line(1,0){6}}
\put(0,6){\line(1,0){6}}
\put(0,0){\line(0,1){6}}
\put(6,0){\line(0,1){6}}
\end{picture}\otimes K\big)=$
$$\verb+tensor(TT,alt_tensor(2,TT))-alt_tensor(3,TT)-tensor(TT,adjoint)+$$
\item $Q\equiv\begin{picture}(12,12)(0,2)
\put(0,0){\line(1,0){12}}
\put(0,6){\line(1,0){12}}
\put(0,12){\line(1,0){12}}
\put(0,0){\line(0,1){12}}
\put(6,0){\line(0,1){12}}
\put(12,0){\line(0,1){12}}
\end{picture}/\big(K\bigodot K\big)=$
$$
\verb+sym_tensor(2,alt_tensor(2,TT))-alt_tensor(4,TT)-sym_tensor(2,adjoint)+$$
\end{itemize}
in case that (\ref{K1->K2}) is injective (cf.~Corollary~\ref{injection} and 
Theorem~\ref{hidden_criterion}).

\section{Appendix: affine Killing fields}\label{affine_appendix} 
For the sake of completeness, we present in this appendix a geometric
interpretation of the restriction $\nabla_aR_{bc}{}^d{}_e=0$ in case that
$\nabla_a$ is a general torsion-free affine connection, not necessarily arising
from a metric (as discussed at the start of~\S\ref{interactions}).  In the
presence of a metric there is no distinction between $1$-forms and vector
fields.  Without a metric, we can still construct the Lie derivative
${\mathcal{L}}_X\nabla_b$ of an affine connection along a vector field~$X^c$, 
defined as the infinitesimal change in $\nabla_b$ along the flow of~$X^c$. If 
$\nabla_b$ is torsion-free, then any such change will be a tensor 
$\Gamma_{ab}{}^c=\Gamma_{(ab)}{}^c$, in this case characterised by 
$$\Gamma_{ab}{}^cZ^b\equiv
{\mathcal{L}}_X(\nabla_aZ^c)-\nabla_a({\mathcal{L}}_XZ^c),\quad
\mbox{for any vector field $Z^c$}.$$
A simple calculation yields
\begin{equation}\label{Gamma}
\textstyle\Gamma_{ab}{}^c=\nabla_a\nabla_bX^c-R_{ad}{}^c{}_bX^d,
\end{equation}
which is symmetric in $ab$ because
$$\textstyle\Gamma_{[ab]}{}^c=\nabla_{[a}\nabla_{b]}X^c+R_{d[a}{}^c{}_{b]}X^d
=\frac32R_{[ab}{}^c{}_{d]}X^d=0$$
owing to~(\ref{curvature_on_one-forms}) and the Bianchi symmetry.  In other
words, the flow of $X^c$ preserves the given connection $\nabla_a$ if and only
if $\nabla_a\nabla_bX^c-R_{ad}{}^c{}_bX^d=0$, in which case $X^c$ is said to be an {\em
affine Killing field}.
If we now take
$$\begin{array}{rcl}U&\equiv&
\mbox{the tangent bundle $TM$ with its given torsion-free connection},\\
V&\equiv&\mbox{$\Wedge^1\otimes TM$ with its induced connection},\\
\makebox[9pt][l]{$\partial:V\to\Wedge^1\otimes U$\enskip 
to be the identity homomorphism,}\\
\makebox[9pt][l]{$\tilde\kappa:U\to\Wedge^1\otimes 
U=\Wedge^1\otimes\Wedge^1\otimes U$\enskip 
to be $X^c\mapsto R_{ad}{}^c{}_bX^d$,}
\end{array}$$
then conclusion (\ref{Ein}) from Theorem~\ref{prolong} merely confirms, via the
connection
\begin{equation}\label{affine_Killing_connection}
\begin{array}{c}TM\\[-5pt] \oplus\\[-2pt] \End(\Wedge^1)\end{array}
=\begin{array}{c}TM\\[-5pt] \oplus\\[-2pt] \Wedge^1\otimes TM\end{array}\ni
\left[\!\begin{array}{c}X^c\\ \phi_b{}^c\end{array}\!\right]
\stackrel{\nabla_a\,}{\longmapsto}
\left[\!\begin{array}{c}\nabla_aX^c-\phi_a{}^c\\ 
\nabla_a\phi_b{}^c-R_{ad}{}^c{}_bX^d\end{array}\!\right],\end{equation}
that $\Gamma_{ab}{}^c$ defined by (\ref{Gamma}) is, indeed, symmetric in $ab$.
Nevertheless, it is clear that the parallel sections of this bundle correspond
to affine Killing fields.  Moreover, this explicit formula
(\ref{affine_Killing_connection}) allows us to compute its curvature
$$\nabla_{[a}\nabla_{b]}
\left[\!\begin{array}{c}X^d\\ \phi_c{}^d\end{array}\!\right]
=\frac12\left[\!\begin{array}{c}0\\ 
R_{ab}{}^d{}_e\phi_c{}^e-R_{ab}{}^e{}_c\phi_e{}^d
+R_{ae}{}^d{}_c\phi_b{}^e-R_{be}{}^d{}_c\phi_a{}^e
-2(\nabla_{[a}R_{b]e}{}^d{}_c)X^e
\end{array}\!\right],$$
consistent with conclusion (\ref{Zwei}) from Theorem~\ref{prolong}, which says,
in particular, that the second line of this curvature should vanish when skewed
over $abc$.  In the affine locally symmetric case 
$\nabla_aR_{bc}{}^d{}_e=0$, we conclude that the kernel of this curvature is 
$TM\oplus \widehat{K}$, where 
$$\widehat{K} \equiv\{\phi_c{}^d\in\End(\Wedge^1)\mid 
R_{ab}{}^d{}_e\phi_c{}^e-R_{ab}{}^e{}_c\phi_e{}^d
+R_{ae}{}^d{}_c\phi_b{}^e-R_{be}{}^d{}_c\phi_a{}^e=0\}.$$
\begin{thm}\label{affine_LTS}
In the affine locally symmetric case $\nabla_aR_{bc}{}^d{}_e=0$, 
we have 
$$R_{ab}{}^c{}_d\in\Gamma(K\otimes \widehat{K}).$$
\end{thm}
\begin{proof} We already know that 
$R_{ab}{}^c{}_d\in\Gamma(K\otimes\End(\Wedge^1))$ from Theorem~\ref{LTS} but 
if we write out (\ref{gives}) from its proof in full, then  
$$R_{cd}{}^p{}_eR_{ab}{}^e{}_q-R_{cd}{}^e{}_qR_{ab}{}^p{}_e
+R_{ce}{}^p{}_qR_{ab}{}^e{}_d-R_{de}{}^p{}_qR_{ab}{}^e{}_c=0$$
and we also see that $R_{ab}{}^c{}_d\in\Gamma(\Wedge^2\otimes \widehat{K})$,
as required.
\end{proof}
\begin{cor}
A locally symmetric affine connection is locally homogeneous with respect to 
its affine Killing fields.
\end{cor}
\begin{proof} The parallel sections of $TM\oplus\End(\Wedge^1)$ are annihilated
by its curvature and are therefore constrained to $TM\oplus \widehat{K}$.  On
the other hand Theorem~\ref{affine_LTS} and the formula
(\ref{affine_Killing_connection}) show that $TM\oplus \widehat{K}$ is a
parallel subbundle.  In summary, $TM\oplus \widehat{K}\subseteq
TM\oplus\End(\Wedge^1)$ is the maximal parallel flat subbundle.
\end{proof}
If it so happens that $\nabla_a$ is a metric connection, then we can use the
metric to raise/lower indices and, since $R_{abcd}$ now enjoys Riemann tensor
symmetries, we may write (\ref{affine_Killing_connection}) as
$$\begin{array}{c}\Wedge^1\\[-5pt] \oplus\\[-2pt] 
\Wedge^1\otimes\Wedge^1\end{array}\ni
\left[\!\begin{array}{c}X_c\\ \phi_{bc}\end{array}\!\right]
\stackrel{\nabla_a\,}{\longmapsto}
\left[\!\begin{array}{c}\nabla_aX_c-\phi_{ac}\\ 
\nabla_a\phi_{bc}-R_{bc}{}^d{}_aX_d\end{array}\!\right]$$
and observe (\ref{Killing_connection}) as a parallel subbundle.  Regarding
their parallel sections, this observation corresponds to the evident fact that
in the Riemannian setting, Killing fields are affine Killing fields.
More discussion on affine symmetric spaces may be found in~\cite{Kowalski}.
	
\raggedright\raggedbottom

\end{document}